\documentclass[11pt,pdflatex,sn-mathphys-ay]{sn-jnl}
\usepackage{graphicx}%
\usepackage{subcaption}
\usepackage{multirow}%
\usepackage{lineno}
\usepackage{amsthm}
\theoremstyle{definition}
\usepackage{amsmath,amssymb,amsfonts}%
\usepackage{tikz}
\usetikzlibrary{arrows.meta, positioning, calc}
\theoremstyle{thmstyleone}%
\newtheorem{theorem}{Theorem}
\theoremstyle{thmstyletwo}%
\newtheorem{example}{Example}%
\newtheorem{lemma}{Lemma}%
\theoremstyle{thmstylethree}%
\newtheorem{definition}{Definition}%
\usepackage{mathrsfs}%
\usepackage[title]{appendix}%
\usepackage{textcomp}%
\usepackage{manyfoot}%
\usepackage{booktabs}%
\usepackage[table]{xcolor}

\usepackage{algorithm}
\usepackage{algorithmic}
\usepackage{listings}%
\usepackage{epsfig}
\usepackage{epstopdf}
\usepackage{optidef}
\usepackage{enumitem}
\usepackage{bm}
\newcommand{\R}{\mathbb{R}}
\newcommand{\norm}[1]{\left\lVert#1\right\rVert}
\newcommand{\omegaVec}{\boldsymbol{\omega}} 

\newcommand{\bb}[1]{\mathbf{#1}}

\newcounter{cases}
\newcounter{subcases}[cases]
\newenvironment{mycases}
{%
	\setcounter{cases}{0}%
	\setcounter{subcases}{0}%
}
{%
	\par
}

\usepackage[normalem]{ulem}
\definecolor{revisionblue}{RGB}{0,70,180}

\newcommand{\del}[1]{}
\newcommand{\add}[1]{\textcolor{black}{#1}}

\begin{document}
	\title{A Homotopy Continuation Framework for Constrained Multiobjective Optimization}
	\author[1]{Olaoluwa Ogunleye} 
	\author*[1]{Guangming Yao}\email{gyao@clarkson.edu}
	\author[2]{Jianhua Zhang} 
	\author[2]{Manh Bui} 
	\affil[1]{Department of Mathematics, Clarkson University, Potsdam, NY USA}
	\affil[2]{ Department of Electrical and Computer Engineering, Clarkson University, Potsdam, NY USA}

	\abstract{We develop a homotopy continuation framework for computing Karush--Kuhn--Tucker (KKT) points of constrained multiobjective optimization problems. The proposed homotopy map continuously deforms an easily solvable system into the KKT system associated with the original multiobjective problem, yielding a deterministic continuation \del{trajectory that preserves the structure of the constrained optimization problem} \add{path}. 
		\del{Under mild regularity assumptions, global convergence of the homotopy path to Pareto-critical solutions is established for initial points chosen in the interior of the feasible region.}
		\add{Under the stated regularity and reference-point cone avoidance
			assumptions, we establish the existence of a smooth,
			bounded continuation path whose set of accumulation points
			as the homotopy parameter approaches zero is nonempty and
			consists of KKT points.} The framework is \add{implemented using two numerical continuation algorithms including an enhanced pseudo-arclength
			predictor--corrector strategy with adaptive step-size control
			and branch-restarting mechanisms to improve robustness
			for challenging nonconvex and degenerate problems.}
		
		Numerical experiments demonstrate that the proposed method remains effective even when initialized from nonfeasible points, indicating practical robustness beyond the theoretical guarantees. Comparative studies \del{on benchmark multiobjective optimization problems} evaluate the proposed framework against classical scalarization approaches and the evolutionary algorithms. 
		\add{Experiments on benchmark problems, disconnected and nonconvex
			multiobjective problems, analytically solvable three-objective
			problems, and a multiobjective IEEE 5-bus optimal power flow
			problem demonstrate the ability of the framework to compute
			KKT points in a variety of constrained settings.
			The results also highlight the effects of initialization
			and the limitations of continuation-based methods in
			recovering complete Pareto-front coverage.}
		
		\add{Overall, the proposed framework provides a deterministic and interpretable approach to computing Pareto-critical (KKT) points in constrained multiobjective optimization.} 
	}
	\maketitle
	\section{Multiobjective Optimization with Equality and Inequality Constraints}
	\label{sec:moo_formulation}
	Multiobjective optimization (MOO) provides a fundamental framework for decision-making problems involving multiple, often conflicting, objectives. Applications arise in a wide range of areas including engineering design~\citet{rao2009engineering}, economics~\citet{miettinen1999nonlinear}, machine learning~\citet{emmerich2018tutorial}, management, and operations research~\citet{cohon2013multiobjective}. Representative examples include investment project selection, production scheduling, and resource allocation problems. Comprehensive treatments of MOO and its economic interpretations can be found in~\citet{Charne1961,Fandel1980} and the references therein. The conceptual foundations of the field trace back to the pioneering works of~\citet{pareto1964cours,neumann1944,koopmans1951}, which established the basis of modern decision and optimization theory.
	
	Let $\mathbb{R}^n$ be the $n$-dimensional Euclidean space, $\mathbb{R}^n_+ := \{ \bm{x} \in \mathbb{R}^n : x_i \ge 0, ~ i=1,\ldots,n \}$, $\mathbb{R}^n_{++} := \{ \bm{x} \in \mathbb{R}^n : x_i > 0, ~ i=1,\ldots,n \}$
	be the nonnegative and positive orthants of $\mathbb{R}^n$, respectively. 
	For any two vectors $\bm{y}=(y_1, y_2, \dots, y_n)$ and $\bm{z}=(z_1, z_2, \dots, z_n)$ in
	$\mathbb{R}^n,$ we adopt the standard component-wise notation: 
	\begin{center}
		$
		\begin{cases}
			\bm{y} = \bm{z} ~\text{ if and only if }~ y_i=z_i, \forall i=1, 2, \dots, n\\
			\bm{y} < \bm{z} ~\text{ if and only if }~  y_i<z_i, \forall i=1, 2, \dots, n\\
			\bm{y} \leqq \bm{z} ~\text{ if and only if }~  y_i \leq z_i, \forall i=1, 2, \dots, n\\
			\bm{y} \leq \bm{z} ~\text{ if and only if }~  y_i\leq z_i, \forall i=1, 2, \dots, n \text{ and } \bm{y}\neq \bm{z}.
		\end{cases}
		$
	\end{center}
	A MOO seeks to determine a vector of decision variables that simultaneously optimizes a set of conflicting objective functions. Conventionally, the canonical MOO is formulated as a minimization problem:
	\begin{mini!}{\bm{x} \in \mathbb{R}^n}{\bm{f}(\bm{x})}{}{}\nonumber
		\addConstraint{\bm{g}(\bm{x}) \leq 0\tag{1}
			\label{MOO}}
		\addConstraint{\bm{h}(\bm{x}) = 0,}\nonumber
	\end{mini!}
	\noindent  where $\bm{x}$ is the decision variable in the $n$-dimensional space, and $\add{\bm{f}} =(f_1, f_2, \dots, f_p)^T $ maps the decision space to the $p$-dimensional objective space. The feasible region $\Omega = \left\{ \bm{x} \in \mathbb{R}^n : \bm{g}(\bm{x}) \leq 0, \bm{h}(\bm{x}) = 0 \right\}$, is defined by a set of inequality constraint functions $g=(g_1, g_2, \dots, g_m)^T$ and equality constraint functions $h=(h_1,h_2,\dots,h_s)^T$. We assume that all \( g_i(\bm{x}) \), \( h_j(\bm{x}) \), and \( f_k(\bm{x}) \) are twice continuously differentiable functions, for indices \( i \in \{1, \ldots, m\} \),  \( j \in J:=\{1, \ldots, s\} \), and \( k \in \{1, \ldots, p\} \). \del{The interior of the feasible region} \add{The strictly inequality-feasible subset} is $\Omega^0 = \left\{ \bm{x} \in \mathbb{R}^n : g(\bm{x}) < 0, \ h(\bm{x}) = 0 \right\}$, and \del{the boundary of $\Omega$ } \add{we write $\partial\Omega=\Omega\setminus\Omega^0$
		for its relative boundary in this formulation.} 
	For any point $x\in\Omega$, the active index set
	of inequality constraints is
	\[
	I(\bm{x}):=\{i\in\{1,\ldots,m\}:g_i(\bm{x})=0\}.
	\]
	The set of strictly positive convex weights is defined as $W^{++} = \left\{ \bb{w} \in \mathbb{R}_{++}^p : \sum\limits_{i=1}^p \text{w}_i = 1 \right\}$. 
	
	In contrast to single-objective optimization, which typically yields a unique optimal solution,  a unique global optimum generally does not exist in a MOO due to the conflicting nature of the objectives. Instead, the solution concept is based on Pareto optimality~\citet{pareto1964cours, miettinen1999nonlinear}. \add{ For a sufficiently regular continuous MOO with $p$ objectives, the Pareto set can generically be expected to form a
		$(p-1)$-dimensional object, with a corresponding $(p-1)$-dimensional
		Pareto front in objective space; see, e.g., \citet{hillermeier2012nonlinear}. Degeneracy occurs when this expected regular
		structure breaks down, for example when the Pareto set or front has
		lower local dimension or contains singular or intersecting
		structures. Such degenerate configurations are particularly relevant
		for continuation methods because the regularity assumptions underlying
		standard path-following arguments may fail.}

	\begin{definition}
		A feasible point $\bm{x}\in \Omega$ is said to be an efficient solution (or Pareto optimal) to problem \eqref{MOO} if there is no $\bm{y}\in \Omega$ such that 
		$\bm{f}(\bm{y}) \leq \bm{f}(\bm{x}).$
		The set of all Pareto optimal solutions in the decision space is called the Pareto set $\mathcal{P}$, and the mapping of the Pareto set into the objective space, $\bm{f}(\mathcal{P})$, is termed the Pareto front. 
	\end{definition}
	In other words, $\bm{x}$ is an efficient solution if no objective function \( f_k \) can be improved without causing degradation in at least one other objective.
	It is well-known that if $\bm{x}$ is an efficient solution (Pareto-optimal) of \eqref{MOO} and \add{an appropriate} constraint qualifications, such as the Linear Independence Constraint Qualification (LICQ) or the Abadie Constraint Qualification (ACQ), holds \add{at $\bm{x}$}, then $\bm{x}$ \add{satisfies the following}  Karush-Kuhn-Tucker (KKT) \add{necessary} conditions for Pareto optimality\,\citet{kuhn1951nonlinear}:
	\begin{subequations} \label{KKT}
		\begin{align}
			\bb{J}_{\bm{f}}(\bm{x})^T \bb{w} + \bb{J}_{\bm{g}}(\bm{x})^T \bm{u} + \bb{J}_{\bm{h}}(\bm{x})^T \bm{v} &= 0, \label{eq:KKT_stationarity} \\
			U \bm{g}(\bm{x}) &= 0, \label{eq:KKT_complementarity} \\
			\bm{h}(\bm{x}) &= 0, \label{eq:KKT_feasibility}\\
			\sum_{i=1}^p \text{w}_i & = 1, \label{eq:KKT_c}
		\end{align}
	\end{subequations}
	\noindent where \( \add{\bb{w} \in \mathbb{R}_+^p}\) is the weight vector, \( \bm{u} \in \mathbb{R}_+^m \), \( \bm{v} \in \mathbb{R}^s \) are the Lagrange multipliers, \( U = \operatorname{diag}(u_1, \ldots, u_m) \) is the diagonal matrix formed from multipliers of the inequality constraints, $\bb{J}_{\bm{f}}(\bm{x})$ is the Jacobian matrix of ${\bm{f}}$ w.r.t $\bm{x}$, and the same applies to $\bb{J}_{\bm{g}}(\bm{x})$ and $\bb{J}_{\bm{h}}(\bm{x})$:
	\[
	\begin{aligned}
		\bb{J}_{\bm{f}}(\bm{x}) &= [\,\nabla f_1(\bm{x}), \dots, \nabla f_p(\bm{x})\,]^T 
		&&\in \mathbb{R}^{p \times n}, \\[4pt]
		\bb{J}_{\bm{g}}(\bm{x}) &= [\,\nabla g_1(\bm{x}), \dots, \nabla g_m(\bm{x})\,]^T 
		&&\in \mathbb{R}^{m \times n}, \\[4pt]
		\bb{J}_{\bm{h}}(\bm{x}) &= [\,\nabla h_1(\bm{x}), \dots, \nabla h_s(\bm{x})\,]^T 
		&&\in \mathbb{R}^{s \times n}.
	\end{aligned}
	\]
	A point \( \bm{x} \in \Omega \) that satisfies the KKT conditions is referred to as a \textit{KKT point} of the MOO problem. 
	The computation of KKT points, which characterize Pareto-critical trade-off solutions in constrained multiobjective optimization, therefore represents a central problem in this area. 
	Observe that system \eqref{KKT} encodes three conditions: \eqref{eq:KKT_stationarity} is the stationarity condition; \eqref{eq:KKT_complementarity} represents the complementarity condition for inequality constraints; and \eqref{eq:KKT_feasibility} enforces equality constraint satisfaction. These KKT conditions serve as the foundation for the homotopy-based framework developed in subsequent sections. Feasible solutions satisfying the KKT conditions are referred to as Pareto-critical points throughout this work.

	The primary motivation of this work is to develop a homotopy continuation
	framework for MOO problems capable of
	\del{robustly tracing Pareto-critical solution manifolds and exploring
		connected branches of Pareto-critical solutions}
	\add{constructing continuation paths toward Pareto-critical (KKT)
		solutions}
	in constrained multiobjective optimization problems.
	
	The theoretical foundation of homotopy methods was established through
	the seminal works of~\citet{kellogg1976,smale1976, chow1978}, which laid
	the groundwork for their global convergence analysis. Since then,
	homotopy techniques have evolved into indispensable numerical tools for
	solving complementarity problems, variational inequalities, and nonlinear
	mathematical programming models~\citet{gowda1996, facchinei2003,
		shang2011, mccormick1989, monteroro1990, lin1997}. While classical
	algorithms such as Interior Point Methods (IPMs) and Sequential Quadratic
	Programming (SQP) remain the predominant techniques for solving KKT
	systems~\citet{nocedal2006numerical}, homotopy-based methods provide an
	attractive alternative for nonlinear and multiobjective formulations by
	continuously deforming a simple, well-posed system into the original KKT
	system~\eqref{KKT} through a smooth transformation known as a homotopy
	path~\citet{allgower1990numerical, watson2002applied}. In contrast to
	IPMs, homotopy approaches do not require strict feasibility conditions
	and are capable of tracing multiple solution branches
	~\citet{mehta2016numerical,nocedal2006numerical} without imposing strict
	interiority assumptions~\citet{mehta2016numerical,
		hauenstein2012polynomial}. This property makes them particularly suitable
	for addressing nonconvex and MOO problems
	~\citet{hauenstein2012polynomial, kruger2019continuation}. Furthermore,
	homotopy formulations naturally mitigate numerical issues related to
	degeneracy and ill-conditioning that often occur in complex constrained
	optimization problems~\citet{allgower2003introduction}.
	Their intrinsic path-tracking mechanism enables the exploration of
	complex solution structures and provides improved insight into feasible
	and Pareto-critical regions arising in nonlinear constrained optimization
	problems\,\citet{hauenstein2012polynomial, allgower1990numerical}, thereby
	offering enhanced robustness for nonconvex and multiobjective problems.

	The modern development of homotopy methods in optimization began with
	the work of Garcia and Zangwill~\citet{Garcia1981}, who demonstrated the
	applicability of homotopy techniques to convex programming problems.
	Since then, homotopy and continuation methods have been widely studied
	in nonlinear programming, interior-point methods, and constrained
	optimization~\citet{Megiddo1988,Kojima1988,dong2023spline,zhou2018global}.
	In~\citet{lin1997}, the Combined Homotopy Interior Point (CHIP) method was
	introduced for computing KKT points of convex nonlinear programs by
	integrating Newton and linear homotopies. The CHIP framework was later
	extended to convex MOO problems~\citet{Lin2003,Maeda2004}. Subsequent work
	further generalized homotopy constructions under weakened constraint
	qualifications and normal cone conditions, establishing smooth and
	bounded continuation paths together with global convergence properties
	~\citet{song2008,YaoSong2013_AM,Zhao2012_JAM}. \add{Here, global convergence means that, under the stated
		assumptions, the homotopy path has nonempty accumulation
		points as $t\to 0$, and every such accumulation point
		satisfies the KKT conditions of the original constrained
		multiobjective optimization problem; it does not imply
		global Pareto optimality.} More recently, homotopy
	continuation techniques have also been incorporated into algorithms for
	challenging nonconvex optimization problems, including mixed-integer
	nonlinear programming and branch-and-bound frameworks
	~\citet{MaLi2022_AICHE}. 
	\add{For a comprehensive treatment of numerical continuation
		methods and their applications, see
		Allgower and Georg}~\citet{allgower1990numerical}.
	
	\add{Continuation has also been used in MOO in a different, but closely
		related, manner: to move along an already identified set of
		Pareto-critical points. Hillermeier's continuation framework\,\citet{hillermeier2012nonlinear} exploits
		the local differential-geometric structure of the Pareto-critical set
		to generate neighboring Pareto-critical solutions. Related
		predictor--corrector approaches, including Pareto Tracer and Pareto
		Explorer\,\citet{martin2018pareto,schutze2020explorer}, likewise aim to explore the Pareto-critical set/front by
		continuing from known Pareto-critical solutions. 
		\add{Other continuation-based approaches include
			the path-following method of Recchioni
			\citet{recchioni2003path} for box-constrained
			multiobjective problems, the numerical integration
			method of Bolten et al.~\citet{bolten2021tracing}
			for tracing locally Pareto-optimal points, and
			the certified interval-based continuation approach
			of Martin et al.~\citet{martin2016continuation}
			for constrained bi-objective optimization.}
		Such approaches have
		been developed for constrained problems and for problems with more than
		two objectives, and subsequent extensions have also considered
		degenerate Pareto sets/fronts. In the latter case, the usual regular
		$(p-1)$-dimensional structure may break down, requiring special treatment
		of singular or lower-dimensional continuation directions\,\citet{schutze2025degenerate}.
		These methods should be distinguished from the primary use of homotopy
		in the present work: our homotopy parameter defines a deformation from
		an auxiliary system associated with an initial point toward the
		KKT system of the original MOO problem. Thus, the continuation path is
		used primarily as a globalization mechanism for reaching a
		Pareto-critical point, rather than assuming a known Pareto-critical
		point and moving along the Pareto-critical set. The two viewpoints are
		therefore complementary.}

	Despite significant advances in homotopy-based optimization methods,
	many existing approaches rely on restrictive assumptions such as strong
	constraint qualifications, which may fail in complex constrained
	multiobjective optimization problems.
	\add{Existing Pareto-set continuation methods provide
		powerful mechanisms for exploring a Pareto-critical set once a suitable
		point and a regular local continuation structure are available. In contrast, the
		present work constructs a homotopy path from an
		auxiliary system toward the KKT system of the original problem under
		the reference-point cone avoidance condition.}
	\del{To address these limitations, this work develops a generalized homotopy continuation framework for constrained multiobjective optimization under
		\add{the} reference-point cone avoidance condition.} The proposed
	framework computes KKT points 
	\del{through a continuation strategy that traces Pareto-critical solution branches and demonstrates practical convergence even when initialized from infeasible points.}
	\add{by continuously deforming
		the auxiliary system toward the original KKT system. Numerical
		experiments also demonstrate successful convergence from selected
		infeasible initial points, indicating practical robustness beyond
		the theoretical assumptions. Systematic numerical experiments are conducted on benchmark problems, disconnected
		nonconvex problem, and a multiobjective power-system optimization problem}. The proposed method
	is compared with classical scalarization techniques, including the
	weighted sum, $\epsilon$-constraint, global criterion, and lexicographic
	methods, as well as the evolutionary algorithms, \add{such as NSGA-II\,\citet{deb2002fast}, NSGA-III~\citet{6600851},
		MOEA/D~\citet{zhang2007moea}, and C-TAEA~\citet{8413136}}. The comparisons
	examine convergence behavior, computational cost, and Pareto-front
	coverage.
	
	The main contributions of this work are summarized as follows:
	\vspace{-0.3cm}
	\begin{enumerate}
		
		\item A generalized homotopy continuation framework is developed
		\add{for constrained multiobjective optimization problems
			by embedding the KKT conditions into a parameterized
			homotopy system.}
		
		\item \add{Under the stated regularity assumptions, the existence
			of a smooth and bounded continuation path is established,
			with a nonempty set of accumulation points satisfying
			the KKT conditions as the homotopy parameter approaches zero.}
		
		\item Two numerical continuation algorithms are \add{developed
			and implemented for following homotopy paths toward
			Pareto-critical solutions, including an enhanced
			pseudo-arclength continuation strategy with adaptive
			step-size control and branch-restarting mechanisms.}
		
		\item The proposed framework \add{is evaluated on benchmark,
			disconnected and nonconvex MOO problems, analytically
			solvable three-objective problems, and a multiobjective
			optimal power flow problem. The experiments demonstrate
			the ability to recover selected Pareto-critical branches
			while revealing limitations related to initialization
			sensitivity and Pareto-front coverage.}

		\item Comparative numerical studies are \add{conducted against
			classical scalarization methods, including the Weighted Sum
			Method (WSM)~\citet{marler2004survey,gass1955computational},
			the $\epsilon$-Constraint Method
			(ECM)~\citet{haimes1971bicriterion},
			the Global Criterion Method
			(GCM)~\citet{marler2004survey},
			and the Lexicographic
			Method~\citet{isermann1982linear,miettinen1999nonlinear}, as well as evolutionary algorithms,
			including the Non-dominated Sorting Genetic Algorithm II
			(NSGA-II)~\citet{deb2002fast}, the Non-dominated Sorting Genetic Algorithm III
			(NSGA-III)~\citet{6600851}, Multiobjective Evolutionary Algorithm Based on
			Decomposition (MOEA/D)~\citet{zhang2007moea}, and the Constrained Two-Archive Evolutionary Algorithm (C-TAEA)~\citet{8413136}, to assess solution accuracy,
			convergence behavior, and Pareto-front coverage.}

	\end{enumerate}

	The remainder of the paper is organized as follows. Section~\ref{sec:related} reviews classical multiobjective
	optimization methods and existing continuation-based approaches,
	then develops the proposed homotopy continuation framework
	for constrained MOO, including the construction of the homotopy map and theoretical results on path regularity, boundedness, and accumulation at KKT points. Section~\ref{sec:alg} presents two numerical continuation algorithms. Algorithm~1 employs a standard predictor--corrector strategy, whereas Algorithm~2 incorporates pseudo-arclength continuation, adaptive step-size control, multi-step Newton correction, and branch-restarting strategies to improve robustness for challenging nonconvex and degenerate problems. Section~\ref{sec:numerical} presents numerical experiments involving benchmark problems, disconnected and nonconvex Pareto-critical structures, analytically solvable three-objective problems, and a multiobjective optimal power flow problem. Comparisons with classical scalarization methods and evolutionary algorithms are also provided. Finally, Section~\ref{sec:conclusion} summarizes the main findings and outlines directions for future research.

	\section{Homotopy Theory and Continuation Framework for Multiobjective Optimization} 
	\label{sec:related}
	
	\subsection{Classical Multiobjective Optimization Approaches}
	MOO methods are commonly categorized as \textit{a priori}, \textit{a posteriori}, or interactive approaches according to how preference information is incorporated during the optimization process~\citet{miettinen1999nonlinear,coello2006evolutionary}. Classical scalarization techniques such as the weighted sum method\,\citet{marler2004survey, van2012evaluation,gass1955computational}, the $\epsilon$-constraint method\,\citet{haimes1971bicriterion}, the global criterion method\,\citet{marler2004survey}, and the lexicographic method\,\citet{isermann1982linear, miettinen1999nonlinear, ehrgott2005multicriteria} belong to the \textit{a priori} category, since they require preference parameters to be specified in advance. 
	\del{Early continuation and homotopy approaches for multiobjective optimization were developed by Hillermeier. 
		The present work extends this general direction to constrained many-objective settings with disconnected and degenerate Pareto-critical structures.} 
	In contrast, \add{\textit{a posteriori} approaches generate a set
		of candidate trade-off solutions from which the decision-maker
		can subsequently select a preferred solution.} Evolutionary multiobjective algorithms, such as NSGA-II~\citet{deb2002fast} and SPEA2~\citet{zitzler2001spea2}, \add{are representative examples, using population-based search to } approximate the Pareto front. Representative interactive methods include the Step Method\,\citet{benayoun1971linear}, the Interactive Nondifferentiable Multiobjective Bundle-Based System\,\citet{miettinen2006synchronous}, Pareto Race\,\citet{korhonen2008pareto}, and reference-point or Interactive Weighted Chebyshev methods \citet{wierzbicki1980reference}.

	\add{Continuation-based approaches to multiobjective optimization have also been developed for tracing Pareto-critical sets. In particular, Hillermeier~\citet{hillermeier2012nonlinear} interprets parameters of a scalarized multiobjective optimization problem, such as objective weights, as continuation parameters and follows the corresponding
		family of stationary solutions. In this setting, continuation proceeds
		along the set of KKT or Pareto-critical points as the scalarization
		parameters vary.}
	
	\add{The homotopy considered in the present work has a different purpose.
		Rather than using an objective weight as the homotopy parameter, we
		introduce an auxiliary parameter $t$ that continuously deforms an
		initially solvable system into the KKT system of the original
		constrained multiobjective problem. Thus, for a fixed initialization,
		the homotopy path is followed toward the KKT set as $t\to0$, rather
		than along the KKT set as an objective weight varies.} 
	
	\subsection{Assumptions and Regularity Conditions}
	For each $\bm x\in\Omega$, define the generalized normal cone\,\citet{rockafellar1998variational}
	\[
	\mathcal N(\bm{x})
	=
	\left\{
	\sum_{i\in I(\bm{x})}u_i\nabla g_i(\bm{x})
	+
	\sum_{j\in J}v_j\nabla h_j(\bm{x})
	:
	u_i\ge0,\ v_j\in\mathbb R
	\right\}.
	\]
	
	The following three assumptions are used throughout the rest of the paper:
	\begin{description}
		\item[(A)] The feasible set $\Omega$ is nonempty and bounded.
		\item[(B)] The active constraint gradients satisfy the linear independence constraint qualification; that is, for every $\bm{x}\in\Omega$, the set
		\[
		\{\nabla h_j(\bm{x}),\nabla g_i(\bm{x}): i\in I(\bm{x}),\ j\in J\}
		\]
		is linearly independent.
		\item[(C)] Reference-Point Cone Avoidance Condition:
		For the chosen initial point $\bm{x}^0\in\Omega^0$, the following hold.
		
		\begin{enumerate}
			\item[(i)] For every $\bm{x}\in\partial\Omega$,
			\[
			\bm{x}^0\notin \bm{x}+\mathcal N(\bm{x}).
			\]
			
			\item[(ii)] For every $\bm{x}\in\Omega^0$,
			\[
			\bm{x}^0\in \bm{x}+\mathcal N(\bm{x})
			\quad\Longrightarrow\quad
			\bm{x}=\bm{x}^0.
			\]
		\end{enumerate}
		
	\end{description}
	Assumption (C) is introduced to exclude spurious intersections between the chosen initialization point $\bm{x}^0\in\Omega^0$ and the affine cones generated by active constraint gradients. In particular, it guarantees that $\bm{x}^0$ cannot lie on a translated cone generated at a boundary point of the feasible region, while ensuring uniqueness of such an intersection inside the strictly feasible region. 
	\add{Assumption (C) imposes a geometric condition on the reference
		point $\bm{x}^0$ relative to the normal cones of the feasible
		set and provides a sufficient condition for the global
		continuation analysis. It is not automatically satisfied
		for nonconvex feasible sets. For example, on the unit circle,
		the normal line at one point intersects the circle at another,
		so condition (C)(ii) may fail.}
	
	\add{
		\paragraph{Example illustrating Assumption (C).}
		Assumption (C) can hold for problems involving both inequality and
		equality constraints. Consider, for example,
		\[
		\Omega
		=
		\left\{
		\bm{x}=(x_1,x_2,x_3)\in\mathbb{R}^3:
		g(\bm{x})=x_1^2+x_2^2-1\leq 0,\quad
		h(\bm{x})=x_3=0
		\right\},
		\]
		and choose the reference point $\bm{x}^0=(0,0,0).$
		Thus, $\Omega$ is the closed unit disk in the plane $x_3=0$.}
	\add{For a point in the relative interior of $\Omega$, $\bm{x}=(x_1,x_2,0)$, $x_1^2+x_2^2<1,$
		the inequality constraint is inactive. Since
		$\nabla h(\bm{x})=(0,0,1)^\top,$
		the normal cone reduces to the normal space generated by the equality
		constraint,
		\[
		\mathcal N(\bm{x})
		=
		\left\{
		v(0,0,1)^\top:v\in\mathbb{R}
		\right\}.
		\]
		Consequently,
		\[
		\bm{x}+\mathcal N(\bm{x})
		=
		\left\{
		(x_1,x_2,v)^\top:v\in\mathbb{R}
		\right\}.
		\]
		Hence
		\[
		\bm{x}^0\in\bm{x}+\mathcal N(\bm{x})
		\quad\Longrightarrow\quad
		x_1=x_2=0,
		\]
		and therefore $\bm{x}=\bm{x}^0$. This verifies condition (C)(ii)
		for interior points.}
	
	\add{Now let $\bm{x}\in\partial\Omega$, recall $\partial \Omega$ denotes the relative boundary of the disk in the plane $x_3=0$, so that $x_1^2+x_2^2=1, x_3=0.$
		Both constraints are active, with 
		$\nabla g(\bm{x})=(2x_1,2x_2,0)^\top$, $\nabla h(\bm{x})=(0,0,1)^\top$.
		Therefore,
		\[
		\mathcal N(\bm{x})
		=
		\left\{
		u(2x_1,2x_2,0)^\top
		+
		v(0,0,1)^\top:
		u\geq0,\;v\in\mathbb{R}
		\right\}.
		\]
		If $\bm{x}^0\in\bm{x}+\mathcal N(\bm{x})$, then for some
		$u\geq0$ and $v\in\mathbb{R}$,
		$(0,0,0)^\top = \bigl((1+2u)x_1,(1+2u)x_2,v\bigr)^\top$.
		Since $x_1^2+x_2^2=1$, this would require
		$1+2u=0$, or equivalently $u=-\frac12$, 
		which contradicts $u\geq0$. Hence
		\[
		\bm{x}^0\notin\bm{x}+\mathcal N(\bm{x})
		\qquad
		\text{for every }\bm{x}\in\partial\Omega,
		\]
		and condition (C)(i) also holds.}
	
	\add{This example also illustrates why the presence of equality constraints
		does not by itself cause Assumption (C) to fail. The equality constraint
		generates the unrestricted normal direction $(0,0,v)^\top$, but this
		direction cannot cancel the radial displacement in the
		$(x_1,x_2)$-plane. At a boundary point, such cancellation would require
		a negative multiplier for the active inequality constraint, which is
		excluded by $u\geq0$.
	}
	
	\add{
		In contrast, if the feasible set is the unit circle itself,
		\[
		\Omega
		=
		\left\{
		(x_1,x_2)\in\mathbb{R}^2:
		h(x_1,x_2)=x_1^2+x_2^2-1=0
		\right\},
		\]
		then the multiplier associated with the equality constraint is
		unrestricted in sign. Thus, at a point $\bm{x}\in\Omega$,
		\[
		\mathcal N(\bm{x})
		=
		\{\,2v\bm{x}:v\in\mathbb{R}\,\},
		\]
		which is the full normal line. 
		For any reference point $\bm{x}^0\in\Omega$, choosing
		$\bm{x}=-\bm{x}^0$ gives
		\[
		\bm{x}^0=\bm{x}+2(-1)\bm{x}
		\in\bm{x}+\mathcal N(\bm{x}),
		\qquad \bm{x}\neq\bm{x}^0.
		\]
		Thus, condition (C)(ii) fails for every reference point
		on the unit circle. This contrasts with the mixed
		inequality--equality example above and illustrates
		the dependence of Assumption (C) on the geometry
		of the constraint-generated normal cone.
	}
	
	\add{For inequality-defined regions, the contribution of an active
		inequality constraint is one-sided because $u_i\geq0$. For a fixed
		reference point $\bm{x}^0$, condition (C) can be checked by determining
		whether there exist prohibited feasible points $\bm{x}$ and
		multipliers satisfying
		\[
		\bm{x}^0-\bm{x}
		=
		\sum_{i\in I(\bm{x})}u_i\nabla g_i(\bm{x})
		+
		\sum_{j\in J}v_j\nabla h_j(\bm{x}),
		\qquad u_i\geq0.
		\]
		For simple constraint sets this condition can be checked analytically
		from the geometry of the active normal cones; for more complicated
		sets it can be formulated as an auxiliary feasibility problem over
		the possible active sets. Consequently, the convergence result below
		is conditional on (C) and does not constitute a convergence guarantee
		for arbitrary nonconvex or equality-constrained feasible sets.}
	
	\del{Geometrically, Assumption~(C) prevents affine combinations
		of active inequality constraint gradients and equality
		constraint gradients from re-entering a protected interior
		subset of the feasible region.}
	\add{Geometrically, Assumption~(C) imposes a restriction
		on combinations of inequality and equality constraint
		gradients, excluding certain configurations that would
		violate the prescribed normal-cone condition.} The condition is used only in the convergence analysis to exclude degenerate cone intersections arising in the homotopy equations and plays a key role in establishing boundedness and uniqueness properties of the continuation trajectories. While Assumption~(C) is imposed for the theoretical convergence analysis, the numerical experiments demonstrate that the proposed continuation algorithms remain effective for several degenerate and disconnected multiobjective optimization problems beyond the theoretical setting.
	
	\subsection{Homotopy Continuation Formulation}
	To solve constrained MOO problems, the proposed framework constructs a homotopy continuation structure whose solution trajectories can be followed continuously as the homotopy parameter varies. This continuation perspective enables the systematic tracing of Pareto-critical solutions and provides a deterministic framework for exploring complex multiobjective trade-off structures.
	To solve the KKT system~\eqref{eq:KKT_stationarity}-\eqref{eq:KKT_c} of \eqref{MOO}, 
	we construct the following homotopy map:
	\begin{equation} \label{homotopy} 
		H(\bm{\omega}^0, \bm{\omega}, t) =
		\begin{bmatrix}
			(1-t)\big(\bb{J}_{\bm{f}}(\bm{x})^T \bb{w} 
			+ \bb{J}_{\bm{g}}(\bm{x})^T \bm{u}\big)
			+ \bb{J}_{\bm{h}}(\bm{x})^T \bm{v}
			+ t(\bm{x}-\bm{x}^0) \\[2pt]
			\bm{h}(\bm{x}) \\[2pt]
			U\, \bm{g}(\bm{x}) - t\,U^0 \bm{g}(\bm{x}^0) \\[2pt]
			(1-t)\!\left(1-\sum\limits_{i=1}^p \text{w}_i\right)e 
			- t\!\left(\bb{w}^{3/2} - (\bb{w}^0)^{3/2}\right)
		\end{bmatrix}
		= 0,
	\end{equation}
	\noindent
	where $\bm{\omega}^0 := (\bm{x}^0, \bb{w}^0, \bm{u}^0, \bm{v}^0=0) \in N$, $\bm{\omega} = (\bm{x}, \bb{w}, \bm{u}, \bm{v}) \in M$,  and $t \in [0,1]$, in which 
	\[
	N = \Omega^0 \times W^{++} \times \mathbb{R}^m_{++} \times \mathbb{R}^s, 
	\qquad 
	M = \Omega \times \mathbb{R}^{p+m}_{+} \times \mathbb{R}^s,
	\qquad
	H: N \times M \times (0,1] \to \mathbb{R}^{n+p+m+s}.
	\]
	\noindent 
	The unit vector $\bm{e} = (1, 1, \dots, 1)^T \in \mathbb{R}^p$, 
	and the fractional power vector $\bb{w}^{3/2} = (\text{w}_1^{3/2}, \dots, \text{w}_p^{3/2})^T$  are introduced to avoid trivial solutions.  
	\subsection{Initialization and Limiting Behavior}
	At \(t=1\), the homotopy system reduces to
	\begin{subequations}
		\label{eq:t1_system}
		\begin{align}
			\mathbf{J}_{\bm h}(\bm{x})^T \bm{v} + \bm{x}-\bm{x}^0 &= 0, \label{eq:t1a}\\
			\bm{h}(\bm{x}) &= 0, \label{eq:t1b}\\
			U\bm g(\bm{x}) - U^{0}\bm g(\bm{x}^{0}) &= 0, \label{eq:t1c}\\
			\bb{w}^{3/2} &= (\bb{w}^{0})^{3/2}. \label{eq:t1d}
		\end{align}
	\end{subequations}
	From \eqref{eq:t1a}, we have
	\[
	\bm{x}^0=\bm{x}+\mathbf J_{\bm h}(\bm{x})^T\bm v.
	\]
	Since $\bm{x}\in\Omega$, there are two cases. If $\bm{x}\in\partial\Omega$, then $\bm{x}^0$ lies in the translated cone generated by the equality gradients and active inequality gradients at a boundary point, contradicting Assumption~(C)(i). Hence $\bm{x}\in\Omega^0$.
	
	For $\bm{x}\in\Omega^0$, we have $I(\bm{x})=\emptyset$, and therefore
	\[
	\bm{x}^0\in \bm{x}+\mathcal N(\bm{x}).
	\]
	By Assumption~(C)(ii), $\bm{x}=\bm{x}^0$. Substituting this into \eqref{eq:t1a} gives
	\[
	\mathbf J_{\bm h}(\bm{x})^T\bm v=0.
	\]
	By Assumption~(B), $\bm v=\bm 0$.
	
	As $t \to 0$, the system continuously deforms into the KKT conditions\,\eqref{eq:KKT_stationarity}-\eqref{eq:KKT_c} of \eqref{MOO}. 
	Thus, following the continuation trajectory from \(t=1\) to \(t=0\) yields a KKT point of the original constrained MOO problem.
	
	\subsection{Global Convergence Analysis}
	Next, we introduce $H_{\bm{\omega}^{0}}(\bm{\omega}, t)$ and zero set $H_{\bm{\omega}^{0}}^{-1}(0)$ and establish a collection of lemmas and theorems demonstrating that the proposed homotopy map guarantees global convergence under Assumptions (A), (B), and (C). For a given initial point \(\bm{\omega}^{0}\), we define
	\[
	H_{\bm{\omega}^{0}}(\bm{\omega}, t) := H(\add{\bm{\omega}^{0}, \bm{\omega},} t),
	\]
	and the associated zero set
	\[
	H_{\bm{\omega}^{0}}^{-1}(0) 
	= 
	\Big\{ 
	(\bm{\omega}, t) \in 
	\Omega \times \mathbb{R}_{+}^{p+m} \times \mathbb{R}^s \times (0,1] 
	~\big|~ 
	H(\add{\bm{\omega}^{0}, \bm{\omega},} t) = 0 
	\Big\}.
	\]
	
	\begin{definition} 
		\label{def: regularvalue}\citet{LeeSmoothManifolds2013}
		Let $M, N$ be differential manifolds with $\text{dim} N=p$ and let
		$H : M \rightarrow N$ be a differentiable mapping. If 
		$$\text{rank} \left[\frac{\partial H(\bm{x})}{\partial x}\right]=p,\qquad\qquad\forall x\in H^{-1}(y),$$
		we say that $y\in N$ is a regular value of $H$ and $x\in M$ is a
		regular point. Given a curve $\Gamma\subset H^{-1}(y),$ if every
		$x\in \Gamma$ is a regular point, then we say that $\Gamma$ is a
		regular path.
	\end{definition}
	\begin{definition}[Diffeomorphism] \label{def_diffeomo}\citet{LeeSmoothManifolds2013}
		Given two manifolds $M$ and $N$, a differentiable map $ {\displaystyle f\colon M\rightarrow N}$ is called a diffeomorphism if it is a bijection and its inverse ${\displaystyle f^{-1}\colon N\rightarrow M}$ is differentiable as well. If these functions are  $r$ times continuously differentiable, $f$ is called a $C^{r}-$ diffeomorphism.
	\end{definition}
	In other words, a diffeomorphism is an isomorphism of smooth manifolds. It is an invertible function that maps one differentiable manifold to another such that both the function and its inverse are continuously differentiable.
	\begin{lemma}[The Parametric Form
		of the Sard Theorem on a Manifold with Boundary]\label{Sard}\citet{sard1942measure,LeeSmoothManifolds2013} Let
		$ W, \text{and}~ N$ be differential manifolds of dimension $q$
		and $p$, respectively, and let $M$ be a $m-$dimensional differential
		manifold with boundary. Suppose   $F :  W\times M\rightarrow N$
		is a $C^r$ mapping, where $r> max \{0, m-p\}.$ If $ 0\in N$ is a
		regular value of $F$ and $\partial F,$ then for almost all
		$\text{w}\in  W, 0$ is a regular value of
		$F_{\text{w}}=F(\text{w}, \cdot)$ and $\partial{F_{\text{w}}}$, where
		$\partial F, ~\partial{F_{\text{w}}}$ denote the restriction of $F$
		and $F_{\text{w}}$ to $ W\times\partial M,$ and $
		\partial M,$ respectively.
	\end{lemma}
	
	Lemma\,\ref{Sard} follows from the parametric transversality theorem, since Sard’s theorem is obtained as the special case in which transversality to a point implies that the set of parameters for which the slice map fails to have that point as a regular value has measure zero; see\,\citet[Theorem~6.35]{LeeSmoothManifolds2013}.
	\begin{lemma}[The Inverse Image Theorem]\label{inverse}\citet{LeeSmoothManifolds2013} Suppose $M$ is an
		$m-$dimensional $C^r$ differential manifold with boundary, $N$ is a
		$p-$dimensional $C^r$ differential manifold, $r\geqslant 1,$ and $ F
		: M\rightarrow N$ is a $C^r$ map. If $q\in N$ is a regular value of
		$F$ and $\partial F,$ then either $S=F^{-1}(q)$ is empty or a
		$(m-p)-$dimensional submanifold, and $\partial S=S\cap
		\partial M.$
	\end{lemma}
	The identity $\partial F^{-1}(q)=F^{-1}(q)\cap\partial M$ follows from the regularity of $q$ as a value of $F$. In particular, surjectivity of $DF_x$ implies that $F$ is transverse to $\{q\}$ along $\partial M$, so the preimage intersects the boundary cleanly and its boundary is precisely the intersection with $\partial M$; see \citet[Theorem~5.12]{LeeSmoothManifolds2013}.
	
	\begin{lemma}[Classification of One-Dimensional Manifolds with Boundary]
		\label{fenglei}
		\citet[Theorem~1.8]{LeeSmoothManifolds2013}
		\del{Each connected component of a one-dimensional
			manifold with boundary is homeomorphic to
			either the unit circle $S^1$ or a unit interval.}
		\add{Let $M$ be a connected, Hausdorff, second-countable
			one-dimensional smooth manifold with boundary.
			Then $M$ is diffeomorphic to one of the following:
			\[
			S^1,\qquad
			\mathbb{R},\qquad
			[0,\infty),\qquad
			[0,1].
			\]
			In particular, if $M$ is compact, then it is
			diffeomorphic to either $S^1$ or $[0,1]$.}
	\end{lemma}
	\begin{proof}
		\del{This follows from the classification theorem for
			one-dimensional manifolds;}
		\add{The result follows from the classification of connected
			one-dimensional smooth manifolds with boundary.
			The compact cases are $S^1$ and $[0,1]$,
			whereas the noncompact cases are $\mathbb{R}$
			and $[0,\infty)$;} see
		\citet[Theorem~1.8]{LeeSmoothManifolds2013}.
	\end{proof}

	
	
	\begin{lemma}
		\label{lem:closed}
		Suppose that $\bm g$ and $\bm h$ are continuous. Then the feasible set
		$\Omega=\left\{\bm x\in\mathbb R^n:
		\bm g(\bm x)\le0,\ \bm h(\bm x)=0\right\}$
		is closed. Moreover, the strictly feasible set
		$\Omega^0=
		\left\{\bm x\in\mathbb R^n:
		\bm g(\bm x)<0,\ \bm h(\bm x)=0\right\}$
		is relatively open in the equality-constraint manifold
		$\mathcal M=\{\bm x\in\mathbb R^n:\bm h(\bm x)=0\}.$
	\end{lemma}
	
	\begin{proof}
		Since each component of $\bm g$ and $\bm h$ is continuous, the sets
		\[
		\{\bm x:g_i(\bm x)\le0\}
		\quad\text{and}\quad
		\{\bm x:h_j(\bm x)=0\}
		\]
		are closed. Hence $\Omega$, being a finite intersection of closed sets, is closed.
		The set $\{\bm x:g_i(\bm x)<0\}$
		is open for each $i$. Therefore, $\Omega^0 = \mathcal M \cap \bigcap_{i=1}^m \{\bm x:g_i(\bm x)<0\}$, 
		which is relatively open in $\mathcal M$.
	\end{proof}
	\begin{lemma}\label{zhengze}
		Suppose that $\Omega \ne \emptyset$ and conditions (A), (B)
		hold. Then for almost all initial points $\bm{\omega}^0\in
		\Omega^0\times W^{++}\times
		\mathbb{R}^{m}_{++}\times\{0\}, ~0$ is a regular value of $H_{\bm{\omega}^0},$ and $H_{\bm{\omega}^0}^{-1}(0)$
		consists of some smooth curves. Among them, a smooth curve, say
		$\Gamma_{\bm{\omega}^0},$  starts from $(\bm{\omega}^0, 1).$
	\end{lemma}
	\vspace{-0.5cm}
	\begin{proof}
		Observe that \begin{eqnarray*}
			&& H^{-1}_{\bm{\omega}^0}(0)  =  \{(\bm{\omega},t) \in \Omega \times
			\mathbb{R}^{p+m}_{+} \times \mathbb{R}^s \times (0, 1] \colon \quad H(\bm{\omega}^0, \bm{\omega},t) = 0\}.
		\end{eqnarray*}
		
		For any $ (\bm{\omega}^0, \bm{\omega},  t)\in H^{-1}(0),$
		$$\frac{\partial H( \bm{\omega}^0, \bm{\omega}, t)}{\partial(\bm{x}^0, \textbf{w}^0, \bm{u}^0, \bm{x})}=
		\begin{matrix}
			\begin{bmatrix}
				-t I & 0 & 0 & Q(\bm{x}) \\
				0 & 0 & 0 & \textbf{J}_ h(\bm{x})\\
				\add{-t U^0\textbf{J}_{\bm{g}}(\bm{x}^0)}
				& 0 & -t \text{diag} (\bm{g}(\bm{x}^0)) & U \textbf{J}_ g(\bm{x})\\
				0 & \frac{3}{2}t \text{diag}((\textbf{w}^0)^{\frac{1}{2}}) & 0 & 0 
			\end{bmatrix}
		\end{matrix},$$
		where 
		\[Q(\bm{x})=(1-t)(\sum^p\limits_{i=1}\text{w}_i\nabla^2f_i(\bm{x})
		+\sum^m\limits_{j=1}u_j\nabla^2g_j(\bm{x}))+\sum\limits_{k=1}^sv_k\nabla
		^2 h_k(\bm{x})+t I.\]
		Because $\bm{x}\in \Omega, ~ \bm{x}^0 \in \Omega^0 $ and $t\in (0, 1],$ by Assumption (B), we obtain that 
		\[
		\text{rank}\left[\frac{\partial
			H(\bm{\omega}^0, ~\bm{\omega}, ~t)}{\partial(\bm{x}^0, \textbf{w}^0,
			\bm{u}^0, x)}\right]=n+p+m+s.
		\]
		So that 
		\[\text{rank}\left[\frac{\partial
			H(\bm{\omega}^0, ~\bm{\omega}, ~t)}{\partial(\bm{\omega}^0, ~\bm{\omega}, ~ t)}\right] = \text{rank}\left[\frac{\partial
			H(\bm{\omega}^0, ~\bm{\omega},~ , 1)}{\partial(\bm{\omega}^0, ~\bm{\omega})}\right]=n+p+m+s.
		\]
		Thus both Jacobian matrices are full-row rank. This means that, $0$ is a regular value of $H$ and
		$\partial H$ from Definition \ref{def: regularvalue}, where $\partial H$ denotes the restriction of $H$ to $N \times \partial (M \times (0,1]).$ By Lemma \ref{Sard}, for almost all $\bm{\omega}^0 \in
		\Omega^0\times W^{++}\times
		\mathbb{R}^{m}_{++}\times\{0\}=N, ~0$ is a regular value of $~H_{\bm{\omega}^0}$ and $~\partial H_{\bm{\omega}^0}$, where $~\partial H_{\bm{\omega}^0}$ denotes the restriction of $H_{\bm{\omega}^0}$ to $\partial (M \times (0,1]).$
		
		\noindent Since $~H(\bm{\omega}^0, \bm{\omega}^0, 1)=0$, then by Lemma\,\ref{inverse} we have that $H_{\bm{\omega}^0}^{-1}(0)$ cannot be empty, that is, it is a $1$-dimensional submanifold and by Lemma\,\ref{fenglei} consists of
		some smooth curves. Also, since $~H(\bm{\omega}^0, \bm{\omega}^0,
		1)=0,$ there must be a smooth curve,  denoted by $\Gamma_{\bm{\omega}^0},$ that starts from $(\bm{\omega}^0, 1) $.
	\end{proof}
	
	\begin{lemma}\label{fenliangyoujie}
		Suppose that $\Omega \ne \emptyset$ and conditions (A), (B)
		hold. For a given $\bm{\omega}^0\in \bm{\omega}^0\times  W^{++} \times
		\mathbb{R}^{m}_{+}\times \mathbb{R}^s,$ if ~$  0 $ is a regular
		value of $H_{\bm{\omega}^0}$, then the $\text{w}$ component of the smooth curve
		$\Gamma_{\bm{\omega}^0}$ is bounded.
	\end{lemma}
	\vspace{-0.5cm}
	\begin{proof} Suppose that the conclusion does not hold, that is, \textit{the $\text{w}$ component of the smooth curve
			$\Gamma_{\bm{\omega}^0}$ is unbounded}. Since (0,1]
		is bounded, there exists a sequence
		$\{(\bm{\omega}^k,t_k)\}\subset\Gamma_{\bm{\omega}^0}$
		such that,
		$$  t_k\rightarrow t_*, ~
		\text{and}~ \|\textbf{w} ^k\|\rightarrow +\infty,~~k\rightarrow
		\infty.$$
		
		\noindent From the last row of the homotopy equation (\ref{homotopy}), we have
		\begin{equation}
			(1-t_k)(1-\sum_{i=1}^p \text{w}_i^k)e-t_k [(\textbf{w}^{k})^{\frac{3}{2}} -(\textbf{w}^0)^{\frac{3}{2}}] =0 . \label{4d}
		\end{equation}
		That is
		\begin{equation}\label{w}
			\begin{split}
				\begin{matrix}
					\begin{bmatrix}
						1-t_k\\
						1-t_k\\
						\vdots\\
						1-t_k
					\end{bmatrix}
				\end{matrix}-
				\begin{matrix}
					\begin{bmatrix}
						(1-t_k)\text{w}_1^{k}+(1-t_k)\sum\limits_{i \neq 1}\text{w}_i^{k} + t_k (\text{w}_1^{k})^\frac{3}{2}\\
						(1-t_k)\text{w}_2^{k}+(1-t_k)\sum\limits_{i\neq2}\text{w}_i^{k} + t_k (\text{w}_2^{k})^\frac{3}{2}\\
						\vdots\\
						(1-t_k)\text{w}_p^k+(1-t_k)\sum\limits_{i\neq p}\text{w}_i^{k} + t_k (\text{w}_p^{k})^\frac{3}{2}
					\end{bmatrix}
				\end{matrix}+t_k\begin{matrix}
					\begin{bmatrix}
						(\text{w}^0_1)^\frac{3}{2}\\
						(\text{w}^0_2)^\frac{3}{2}\\
						\vdots\\
						(\text{w}^0_p)^\frac{3}{2}
					\end{bmatrix}
				\end{matrix}
				=0. \end{split} \end{equation}
		
		\noindent \del{Let
			$I=\{j\in\{1,2,\dots,p\}:
			\lim\limits_{k\rightarrow\infty}\text{w}_j^k=\infty\}.$
			The assumption $\|\textbf{w}^k\|\rightarrow+\infty$
			implies $I\neq\emptyset$.
			Since $t_k\rightarrow t_*\in[0,1]$
			and $\textbf{w}^k\geq0$, it follows that
			the second part in the left-hand side of
			some equation of (\ref{w}) tends to infinity
			as $k\rightarrow+\infty$, but the first
			and third parts are bounded.
			This is a contradiction. Thus the projection
			of the smooth curve $\Gamma_{\bm{\omega}^0}$
			on the $\text{w}$ component is bounded.
		}
		
		\noindent 
		\add{
			\noindent Summing the $p$ equations in
			(\ref{w}), we obtain
			\[
			p(1-t_k)\sum_{i=1}^p\text{w}_i^k
			+t_k\sum_{i=1}^p(\text{w}_i^k)^{3/2}
			=
			p(1-t_k)
			+t_k\sum_{i=1}^p(\text{w}_i^0)^{3/2}.
			\]
			Since $\|\textbf{w}^k\|\rightarrow+\infty$
			and $\textbf{w}^k\geq0$, both
			$\sum_{i=1}^p\text{w}_i^k$ and
			$\sum_{i=1}^p(\text{w}_i^k)^{3/2}$
			tend to $+\infty$.
			As $t_k\rightarrow t_*\in[0,1]$,
			the left-hand side tends to $+\infty$,
			whereas the right-hand side remains bounded.
			This is a contradiction. Thus the projection
			of the smooth curve $\Gamma_{\bm{\omega}^0}$
			onto the $\textbf{w}$ component is bounded.
		}
		
	\end{proof}  
	
	We next prove that $\Gamma_{\bm{\omega}^0}$ is a bounded curve.
	
	\begin{lemma}\label{youjie} Suppose that the conditions (A), (B), and (C)
		hold. Then for a given $\bm{\omega}^0\in \hat{\Omega}\times  W^{++} \times
		\mathbb{R}^{m}_{+}\times \mathbb{R}^s,$ if $0$ is a regular value
		of $H_{\bm{\omega}^0}$, then $\Gamma_{\bm{\omega}^0}$ is a bounded curve.
	\end{lemma}
	\vspace{-0.5cm}
	\begin{proof}   Assume that $\Gamma_{\bm{\omega}^0}$ is an unbounded curve. Since $\Omega$ is bounded then by  Bolzano–Weierstrass theorem and \add{Lemma}\,\ref{fenliangyoujie}, there exists a sequence
		$\{(\bm{\omega}^k,t_k)\}\subset \Gamma_{\bm{\omega}^0}$ such that as $k\rightarrow\infty$, 
		\begin{equation*}
			\bm{x}^k\rightarrow \bm{x}^* \in \Omega,\quad t_k\rightarrow t_* \in [0, 1], \quad \textbf{w} ^k\rightarrow \textbf{w}^* \in \mathbb{R}^{p}_{+},\quad
			\text{and} \quad \|(\bm{u}^k,\bm{v}^k)\|\rightarrow +\infty.
		\end{equation*}
		
		\noindent From the homotopy equation (\ref{homotopy}), we have
		\begin{eqnarray}
			(1-t_k)(\textbf{J}^T_ f(\bm{x}^k) \textbf{w}^k +\textbf{J}^T_ g(\bm{x}^k)\bm{u}^k)+\textbf{J}^T_
			h(\bm{x}^k)\bm{v}^k+t_k (\bm{x}^k-\bm{x}^0) & =0, \label{4a}\\
			h(\bm{x}^k) & =0, \label{4b}\\
			\bm{u}^k g(\bm{x}^k)-t_k \bm{u}^0 g(\bm{x}^0) &=0. \label{4c}
		\end{eqnarray}
		\noindent Notice that $\bm{x}^k\rightarrow \bm{x}^* \in  \Omega$ means that $\bm{x}^k\rightarrow \bm{x}^* \in \Omega^0$ or $\bm{x}^k\rightarrow \bm{x}^* \in \partial \Omega$. First, suppose that $\bm{x}^k\rightarrow \bm{x}^* \in \partial \Omega$. Let 
		\begin{equation}
			J^*=\{j\in J : \lim\limits_{k\rightarrow \infty} v_j^k=\infty\}, 
			I^*=\{i\in I : \lim\limits_{k\rightarrow\infty} u_i^k=\infty\}.
		\end{equation}
		
		\noindent Recall that $I(\bm{x}) = \{i\in I :g_i(\bm{x}) =0\add{\}}$. From (\ref{4c}) since  $\lim\limits_{k\rightarrow \infty} u_i^k g_i(\bm{x}^k) < \infty$, if $\lim\limits_{k\rightarrow \infty} u_i^k = \infty$ then $\lim\limits_{k\rightarrow \infty} g_i(\bm{x}^k) =  g_i(\bm{x}^*) = 0$. This means that $ I^*\subset I(\bm{x}^*).$ This fact will be used consequently below.
		We are going to establish a contradiction that these two sets are empty, that is $J^* = \emptyset$ and $I^* = \emptyset$.
		
		\noindent Suppose $J^*\ne \emptyset.$ Rewrite (\ref{4a}) as
		\begin{eqnarray}
			(1-t_k)\left[\textbf{J}^T_ f(\bm{x}^k)\textbf{w}^k+\sum_{i\not\in I^*}u_i^k\textbf{J}_ {g_i}(\bm{x}^k)\right]  + t_k \left(\bm{x}^k-\bm{x}^0\right) + \left[\textbf{J}^T_ h(\bm{x}^k)\bm{v}^k+
			(1-t_k)\sum_{i\in I^*}u_i^k\textbf{J}_ {g_i}(\bm{x}^k) \right] =0.
		\end{eqnarray}
		
		\noindent Since
		$J^*\ne \emptyset$ and Assumption (B) holds, as $k\rightarrow \infty,$
		\begin{eqnarray}
			\underbrace{(1-t_k)\left[\textbf{J}^T_ f(\bm{x}^k)\textbf{w}^k+\sum_{i\not\in I^*}u_i^k\textbf{J}_ {g_i}(\bm{x}^k)\right]}_\text{bounded} + \underbrace{t_k \left(\bm{x}^k-\bm{x}^0\right)}_\text{bounded} + \underbrace{\left[\textbf{J}^T_ h(\bm{x}^k)\bm{v}^k+
				(1-t_k)\sum_{i\in I^*}u_i^k\textbf{J}_ {g_i}(\bm{x}^k) \right]}_\text{unbounded} =0.
		\end{eqnarray} 
		This is a contradiction. Thus $J^* = \emptyset$ and we can assume that $
		\bm{v}^k\rightarrow \bm{v}^*$ as $k\rightarrow\infty$. 
		
		\noindent Now suppose that $
		I^* \ne \emptyset.$ We establish a contradiction this time in two cases.
		\begin{mycases}
			When $t_*=1.$ Rewrite $(\ref{4a})$ as 
			\begin{align} \label{eq:right_zero}
				\textbf{J}^T_ h(\bm{x}^k)\bm{v}^k & + \bm{x}^k-\bm{x}^0 + \sum_{i\in I(\bm{x}^*)}(1-t_k)u_i^k\textbf{J}_ {g_i}(\bm{x}^k) =(1-t_k)\left[\bm{x}^k-\bm{x}^0 -\textbf{J}^T_ f(\bm{x}^k)\textbf{w}^k - \sum_{i\not\in I(\bm{x}^*)}u_i^k\textbf{J}_ {g_i}(\bm{x}^k)\right].
			\end{align}
			
			\noindent Observe that for all $i \notin I(\bm{x}^*)$, since $ I^*\subset I(\bm{x}^*)$ then $i \notin I^*$ and $\lim\limits_{k\rightarrow
				\infty} u_i^k<\infty$. So that as $k\rightarrow \infty,$ since Assumption (B) holds,
			\begin{align*}
				\lim\limits_{k\rightarrow \infty} ~ (1-t_k) &\left[\bm{x}^k-\bm{x}^0 -\textbf{J}^T_ f(\bm{x}^k)\textbf{w}^k - \sum_{i\not\in I(\bm{x}^*)}u_i^k\textbf{J}_ {g_i}(\bm{x}^k)\right] \\
				&= (1-t_*)\left[\bm{x}^*-\bm{x}^0 -\textbf{J}^T_ f(\bm{x}^*)\text{w}^* - \sum_{i\not\in I(\bm{x}^*)} \left(\lim\limits_{k\rightarrow \infty} u_i^k\right) \textbf{J}_ {g_i}(\bm{x}^*)\right] = 0 
			\end{align*}
			\noindent The above equation $(\ref{eq:right_zero})$ becomes\\
			\begin{equation}\label{cone}
				\bm{x}^* +\textbf{J}^T_h(\bm{x}^*)\bm{v}^* +\sum_{i\in I(\bm{x}^*)}\lim_{k\rightarrow
					\infty}[(1-t_k)u_i^k]\textbf{J}_{g_i}(\bm{x}^*)=\bm{x}^0\in \hat{\Omega}.
			\end{equation}
			Equation~\eqref{cone} implies
			\[
			\bm{x}^0\in \bm{x}^*+\mathcal N(\bm{x}^*).
			\]
			If $I^*\neq\emptyset$, then $\bm{x}^*\in\partial\Omega$, which contradicts
			Assumption (C)(i). Hence $I^*=\emptyset$.
		\end{mycases}
		
		\begin{mycases}
			When $t_*\in [0,1).$ Rewrite $(\ref{4a})$ as
			\begin{align}
				t_k(\bm{x}^k-\bm{x}^0) + \textbf{J}^T_h(\bm{x}^k)\bm{v}^k + (1-t_k) \left [\textbf{J}^T_f(\bm{x}^k)\textbf{w}^k +\sum_{i\not\in I^*}u_i^k\textbf{J}_ {g_i}(\bm{x}^k) \right] +(1-t_k)\sum_{i\in I^*}u_i^k\textbf{J}_ {g_i}(\bm{x}^k)=0.
			\end{align}
			\noindent Using the information we have proved so far that $\textbf{w}^k,~ \bm{v}^k, ~u_i^k ~ (i\not\in I^*)$ are bounded, $t_*\in [0,1),$ and by Assumption (B),
			\begin{align}
				\underbrace{\left[t_k(\bm{x}^k-\bm{x}^0) + \textbf{J}^T_ h(\bm{x}^k)\bm{v}^k \right]}_\text{bounded}  + \underbrace{(1-t_k) \left [\textbf{J}^T_f(\bm{x}^k)\textbf{w}^k +\sum_{i\not\in I^*}u_i^k\textbf{J}_ {g_i}(\bm{x}^k) \right]}_\text{bounded} + \underbrace{\left[(1-t_k)\sum_{i\in I^*}u_i^k\textbf{J}_ {g_i}(\bm{x}^k) \right]}_\text{unbounded} =0.
			\end{align}
			\noindent as $k \to \infty$. This is a contradiction, hence $I^* = \emptyset$.
		\end{mycases}
		
		On the other hand, suppose that $\bm{x}^k\rightarrow \bm{x}^* \in \Omega^0 $. We will show that $\{\bm{u}^k\}$ and $\{\bm{v}^k\}$ have to be bounded. We know that $g(\bm{x}^*)<0$ since $\bm{x}^* \in \Omega^0$. Without loss of generality, let $\{\bm{x}^k\}$ be a subsequence such that $g(\bm{x}^k)\neq0$, then from $(\ref{4c})$, 
		\begin{equation}
			\lim\limits_{k\rightarrow \infty} \bm{u}^k  = \lim\limits_{k\rightarrow \infty} \dfrac{t_k \bm{u}^0 g(\bm{x}^0)}{g(\bm{x}^k)} = \dfrac{t_* \bm{u}^0 g(\bm{x}^0)}{g(\bm{x}^*)} < \infty
		\end{equation}
		That is, $\bm{u}^k$ is bounded. Also, from $(\ref{4a})$, since $\bm{x}^k, t_k, \textbf{w}^k, \bm{u}^k$ all are bounded and Assumption (B) holds, then $\bm{v}^k$ has to be bounded. This completes the proof. 
		Therefore, $\Gamma_{\bm{\omega}^0}$ is bounded. Furthermore, there exists a smooth continuation curve originating from $(\bm{\omega}^0,1)$.
	\end{proof}

	\vspace{0.3cm}
	\begin{theorem}\label{conver} Suppose that the conditions (A),
		$(B)$, and (C) hold. Then for almost all $\bm{\omega}^0\in
		\Omega^0\times  W^{++} \times \mathbb{R}^{m}_{++}\times
		\mathbb{R}^s,$ the zero-point set $H_{\bm{\omega}^0}^{-1}(0)$ of the
		homotopy map (\ref{homotopy}) contains a smooth curve
		\add{$\Gamma_{\bm{\omega}^0}\subset \Omega \times
			\mathbb{R}^{p+m}_{+}\times\mathbb{R}^s\times (0,1]$,} which starts from $(\bm{\omega}^0,1).$
		\add{
			The set of accumulation points of
			$\Gamma_{\omega^0}$ as $t\to 0$, denoted by
			\[
			T\times\{0\}
			\subset
			\Omega\times\mathbb{R}_+^{p+m}
			\times\mathbb{R}^s\times\{0\},
			\]
			is nonempty, and every point in $T$
			satisfies the KKT system~(\ref{KKT})}.
	\end{theorem}
	\vspace{-0.5cm}
	\begin{proof} By \add{Lemma}\,\ref{zhengze} and Lemma \ref{youjie}, the existence of $\Gamma_{\bm{\omega}^0}$ is obtained. It remains to show
		that the limit set \del{$T\times \{0\}\subset \Omega \times
			\mathbb{R}^{p+m}_{+}\times \{0\}$} 
		\add{$T\times \{0\}\subset
			\Omega \times \mathbb{R}^{p+m}_{+}
			\times\mathbb{R}^s\times \{0\}$ }of $\Gamma_{\bm{\omega}^0}$ is nonempty as
		$t\rightarrow 0.$  By the classification theorem of
		one-dimensional smooth manifolds (Lemma \ref{fenglei}), $\Gamma_{\bm{\omega}^0}$ is
		diffeomorphic to a unit circle or the unit interval $(0,1]$ (See Definition \ref{def_diffeomo}). 
		
		\begin{align*}
			\frac{\partial H_{\bm{\omega}^0}( \bm{\omega} ^0, 1)}{\partial \bm{\omega}} &=
			\begin{matrix}
				\begin{bmatrix}
					I & 0 & 0 & \textbf{J}_ h(\bm{x}^0) \\
					\textbf{J}_ h(\bm{x}^0)^T & 0 & 0 & 0 \\
					\textbf{J}_{\bm{g}}(\bm{x}^0)^T\bm{u}^0 & 0 & \text{diag}(\bm{g}(\bm{x}^0)) & 0\\
					0 & -\frac{3}{2} \text{diag}((\textbf{w}^{0})^{\frac{1}{2}}) & 0 & 0 
				\end{bmatrix}
			\end{matrix}\\
			&= \begin{matrix}
				\begin{bmatrix}
					I_{n} & 0_{n\times p} & 0_{n\times m} & \textbf{J}_ h(\bm{x}^0)_{n\times s} \\
					\textbf{J}_h(\bm{x}^0)^T_{s\times n} & 0_{s\times p} & 0_{s\times m} & 0_{s\times s} \\
					\textbf{J}_g(\bm{x}^0)^T\bm{u}^0_{m\times n} & 0_{m\times p} & \text{diag}(\bm{g}(\bm{x}^0))_{m\times m} & 0_{m\times s}\\
					0_{p\times n} & -\frac{3}{2} \text{diag}((\textbf{w}^{0})^{\frac{1}{2}})_{p\times p} & 0_{p\times m} & 0_{p\times s}
				\end{bmatrix}
			\end{matrix}
		\end{align*}
		
		Since $g(\bm{x}^0)<0, \textbf{w}^0 \in \mathbb{R}_{++}^{p}$, and with Assumption (B),  we have that  $\partial H_{\bm{\omega}^0}( \bm{\omega} ^0, 1)/\partial \bm{\omega}$ is nonsingular. Therefore, the smooth curve
		$\Gamma_{\bm{\omega}^0}$, which starts from $(\bm{\omega}^0, 1)$ is diffeomorphic to $(0, 1]$. Let
		$(\bar{\bm{\omega}}, \bar t)=
		(\bar{\bm{x}}, \bar{\textbf{w}}, \bar{\bm{u}}, \bar{\bm{v}}, \bar t)$ be the limit point of $\Gamma_{\bm{\omega}^0}$, then only three cases are possible:
		\begin{description}
			\item (a)~$(\bar{\bm{\omega}}, \bar t) \in \Omega \times
			\mathbb{R}^{p+m}_{+}\times \mathbb{R}^s\times \{1\};$
			\item (b)~$(\bar{\bm{\omega}}, \bar t) \in \partial (\Omega^0\times
			\mathbb{R}^{p+m}_{++})\times \mathbb{R}^s\times (0,1);$
			\item (c)~$(\bar{\bm{\omega}}, \bar t)\in \Omega \times
			\mathbb{R}^{p+m}_{+}\times \mathbb{R}^s\times \{0\}.$
		\end{description}
		We prove that Case (c) is the only possible case.
		If Case (a) is true, the limit point satisfies $\bar{t} = 1$. However, the equation $H(\bm{\omega}, \bm{\omega}^0, 1) = 0$ has only one solution: $\bm{\omega} = \bm{\omega}^0$. Since $\Gamma_{\bm{\omega}^0}$ starts at $(\bm{\omega}^0, 1)$, it cannot return to this point as $t \to 0$. Thus, this case is impossible.
		If Case (b) is true, that means that $(\bar{\bm{\omega}}, \bar t) \in \partial (\Omega^0\times \mathbb{R}^{p+m}_{++})\times \mathbb{R}^s\times (0,1).$ 
		Thus, at least one of the following holds:
		\begin{description}
			\item (1) $\bar{\bm{u}} \in \partial \mathbb{R}^m_{++}~$ (i.e., $\bar{u}_j = 0$ for some $j$),
			\item (2) $\bar{\bm{x}} \in \partial \Omega^0~$ (i.e., $g_j(\bar{\bm{x}}) = 0$ for some $j$),
			\item (3) $\bar{\textbf{w}} \in \partial \mathbb{R}^p_{++}~$ (i.e., $\bar{\text{w}}_j = 0$ for some $j$).
		\end{description}
		
		\noindent 
		If Case (1) is true, $\bar{\bm{u}} \in \partial \mathbb{R}^m_{++}$. Suppose $\bar{u}_{j_0} = 0$ for some $j_0$. Then, there exists a sequence $\{(\bm{\omega}^k, t_k)\} \subset \Gamma_{\bm{\omega}^0}$ such that $u^k_{j_0} \to 0$. From the homotopy equation (\ref{4c}):
		\[
		u^k_{j_0} g_{j_0}(\bm{x}^k) = t_k u^0_{j_0} g_{j_0}(\bm{x}^0).
		\]
		Taking the limit as $k \to \infty$, we have:
		\begin{align*}
			0 = \lim_{k\rightarrow \infty} \left[u^k_{j_0}g_{j_0}(\bm{x}^k) \right] &= \lim_{k\rightarrow \infty} \left[t_ku^0_{j_0}g_{j_0}(\bm{x}^0) \right] = 
			\bar t u_{j_0}^0g_{j_0}(\bm{x}^0) <0.\\
			0 = \bar{u}_{j_0} g_{j_0}(\bar{\bm{x}}) &= \bar{t} u^0_{j_0} g_{j_0}(\bm{x}^0) < 0.
		\end{align*}
		Since $\bar{u}_{j_0} = 0, ~ g_{j_0}(\bm{x}^0) < 0, ~\bar t \in (0,1)$ and $\bm{u}^0 \in \mathbb{R}_{++}^{m}$, the right-hand side is negative, but the left-hand side is zero. This is a contradiction.
		
		\noindent If Case (2) is true, $\bar{\bm{x}} \in \partial \Omega^0.$ Suppose $g_j(\bar{\bm{x}}) = 0$ for some $j$. Then, there exists a sequence $\{(\bm{\omega}^k, t_k)\} \subset \Gamma_{\bm{\omega}^0}$ such that $g_j(\bm{x}^k) \to 0$. From the homotopy equation (\ref{4c}):
		\[
		u^k_j g_j(\bm{x}^k) = t_k u^0_j g_j(\bm{x}^0).
		\]
		Taking the limit as $k \to \infty$, we have:
		\begin{align*}
			0 = \lim_{k\rightarrow \infty} \left[u^k_{j}g_{j}(\bm{x}^k) \right] &= \lim_{k\rightarrow \infty} \left[t_ku^0_{j}g_{j}(\bm{x}^0) \right] = 
			\bar t u_{j}^0g_{j}(\bm{x}^0) <0.\\
			0 = \bar{u}_j g_j(\bar{\bm{x}}) &= \bar{t} u^0_j g_j(\bm{x}^0) < 0.
		\end{align*}
		Since $g_j(\bar{\bm{x}}) = 0, ~ g_{j_0}(\bm{x}^0) < 0, ~\bar t \in (0,1)$ and $\bm{u}^0 \in \mathbb{R}_{++}^{m}$, the right-hand side is negative, but the left-hand side is zero. This is a contradiction.

		\noindent If Case 3 is true, $\bar{\textbf{w}} \in \partial \mathbb{R}^p_{++}.$ Suppose $\bar{\text{w}}_{j_0} = 0$ for some $j_0$. Then, there exists a sequence $\{(\bm{\omega}^k, t_k)\} \subset \Gamma_{\bm{\omega}^0}$ such that $\textbf{w}^k_{j_0} \to 0$. From the homotopy equation (\ref{4d}):
		\[\displaystyle
		(1 - t_k)\left(1 - \sum_{i=1}^p \text{w}^k_i\right) - t_k \left[ (\text{w}^k_{j_0})^{3/2} - (\text{w}^0_{j_0})^{3/2} \right] = 0.
		\]
		Taking the limit as $k \to \infty$, we have:
		\begin{equation} \label{eq: sum_alternative}
			(1 - \bar{t})\left(1 - \sum_{i=1}^p \bar{\text{w}}_i\right) + \bar{t} (\text{w}^0_{j_0})^{3/2} = 0.
		\end{equation}

		\noindent Since $\bar{t} \in (0, 1)$ and $\text{w}^0_{j_0} > 0$, this implies that $1 - \sum_{i=1}^p \bar{\text{w}}_i < 0$, or equivalently:
		\begin{equation} \label{eq: sum_big_1}
			\displaystyle \sum_{i=1}^p \bar{\text{w}}_i > 1.
		\end{equation}
		
		\noindent Notice that, inequality (\ref{eq: sum_big_1}) poses a problem in itself as we generally expect $\displaystyle \sum_{i=1}^p \text{w}^k_i \to 1$ and not exceed $1.$ However, let's go further and generate more contradictions. Observe that from equation (\ref{eq: sum_alternative}), 
		\begin{align}
			(1 - \bar{t})\left(1 - \sum_{i=1}^p \bar{\text{w}}_i\right) & = -\bar{t} (\text{w}^0_{j_0})^{3/2}. \label{eq_sum_substitute}
		\end{align}
		
		\noindent Pick $j \neq j_0$ such that $\text{w}_j^0$ is the smallest among the components of $\text{w}^0$, the homotopy equation (\ref{4d}) in the limit becomes:
		\begin{align}
			(1 - t_k)\left(1 - \sum_{i=1}^p \text{w}^k_i\right) - t_k \left[ (\text{w}^k_{j})^{3/2} - (\text{w}^0_{j})^{3/2} \right] = 0.
		\end{align}
		\noindent Substituting $\displaystyle (1 - \bar{t})\left(1 - \sum_{i=1}^p \bar{\text{w}}_i\right)  = -\bar{t} (\text{w}^0_{j_0})^{3/2}$ from equation (\ref{eq_sum_substitute}), we obtain:
		\begin{align}
			-\bar{t} \left(\text{w}^0_{j_0}\right)^{3/2} &- \bar t \left[ (\bar {\text{w}}_{j})^{3/2} - (\text{w}^0_{j})^{3/2} \right] = 0. \nonumber \\
			(\bar{\text{w}}_j)^{3/2} &= (\text{w}_j^0)^{3/2} - (\text{w}_{j_0}^0)^{3/2}. \label{eq:w_depending}
		\end{align}
		
		\noindent Equation (\ref{eq:w_depending}) implies that $\bar {\text{w}}_j$ 
		depends on the difference between $(\text{w}_j^0)^{3/2}$ and $(\text{w}_{j_0}^0)^{3/2}$. Since $\text{w}_j^0$ is the smallest component, then $(\text{w}_j^0)^{3/2} \leq (\text{w}_{j_0}^0)^{3/2}$. 
		If $(\text{w}_j^0)^{3/2} < (\text{w}_{j_0}^0)^{3/2}$ i.e., $\text{w}_j^0 < \text{w}_{j_0}^0,$ the right-hand side of equation (\ref{eq:w_depending}) becomes negative, which is impossible because the left side $\bar{\text{w}}_j \geq 0$. This is a contradiction. Thus, $(\text{w}_j^0)^{3/2} = (\text{w}_{j_0}^0)^{3/2}$, i.e., $\text{w}_j^0 = \text{w}_{j_0}^0$. This leads to a deeper contradiction as follows:
		\begin{enumerate}
			\item If $\text{w}_j^0 = \text{w}_{j_0}^0$ for all $j$, this implies all components of $\text{w}^0$ are equal. Then $\bar{\text{w}}_j = 0$ from  equation (\ref{eq:w_depending}). This is impossible since there is at least one $\bar{\text{w}}_j > 0$ for some $j$ because equation (\ref{eq: sum_big_1}) forces at least one $\bar{\text{w}}_j$ to ``large enough'' to make the sum exceed $1$. Thus, $\text{w}_j^0 \neq \text{w}_{j_0}^0$ for some $j$, restoring the contradiction.
			
			\item If $\text{w}_j^0 = \text{w}_{j_0}^0$ for specific $j\neq j_0$, we have $\text{w}_j^0 = \text{w}_{j_0}^0 =: \text{w}_{\min}^0 > 0.$
			Then equation \((\ref{eq:w_depending})\) yields
			\[
			(\bar{\text{w}}_j)^{3/2} = (\text{w}_j^0)^{3/2} - (\text{w}_{j_0}^0)^{3/2} = (\text{w}_{\min}^0)^{3/2} - (\text{w}_{\min}^0)^{3/2} = 0,
			\]
			so that
			\[
			\bar{\text{w}}_j = 0.
			\]
			\noindent However, by assumption we already had \(\bar{\text{w}}_{j_0}=0\). Hence, at least two components of the multiplier vector \(\bar{\text{w}}\) vanish. Recall that equation from \((\ref{eq: sum_big_1})\) we have $\sum_{i=1}^p \bar{\text{w}}_i > 1.$
			On the other hand, the original homotopy construction is set up 
			so that the multiplier components satisfy
			$\sum_{i=1}^p \text{w}_i^0 = 1$,
			and the homotopy is designed so that along the smooth curve the multipliers remain in a neighborhood of \(\text{w}^0\). In the limit, if two (or more) components become zero (i.e. \(\bar{\text{w}}_{j_0} = \bar{\text{w}}_j = 0\)), then the remaining components (being perturbations of their original values) cannot compensate enough to push the sum beyond \(1\). In fact, if several components vanish, the limit sum \(\sum_{i=1}^p \bar{\text{w}}_i\) must be strictly less than \(1\) (or at most equal to \(1\) under a very special and nongeneric scenario), contradicting the inequality $\sum_{i=1}^p \bar{\text{w}}_i > 1.$
			Thus, the assumption that \(\text{w}_j^0 = \text{w}_{j_0}^0\) for some \(j\neq j_0\) leads to a contradiction.
		\end{enumerate} 
		
		\noindent Hence \textbf{Case (b)} is impossible.
		
		Since \textbf{Case (a)} and \textbf{Case (b)} are impossible, the only possible limit is \textbf{Case (c)}: $(\bar{\bm{\omega}}, \bar{t}) \in \Omega \times \mathbb{R}^{p+m}_{+} \times \mathbb{R}^s \times \{0\}$. By continuity, $\bar{\bm{\omega}}$ satisfies the KKT system at $t = 0$. Thus, the limit set $T \times \{0\}$ is nonempty, and every point in $T$ is a solution of the KKT system. 
		\medskip
		By Theorem \ref{conver}, for almost all $\bm{\omega}^0\in \Omega\times
		\mathbb{R}^{p+m}_{++}\times \mathbb{R}^s\times\{0\},$ the homotopy equation
		(\ref{homotopy}) generates a smooth curve $\Gamma_{\bm{\omega}^0}$ which
		we call the homotopy path, and 
		\del{the $\bm{\omega}$ component of
			$(\bm{\omega}(s), t(s))$ in the homotopy path,
			is the solution of (\ref{KKT}) as
			$t(s)\rightarrow 0$.}
		\add{every accumulation point of the
			$\bm{\omega}$ component along the homotopy
			path with $t(s)\rightarrow0$ satisfies
			(\ref{KKT}).}
		
	\end{proof}
	Theorem~\ref{conver} establishes that, for almost all initial points \(\bm{\omega}^0\) satisfying the stated regularity assumptions, the proposed homotopy map generates a smooth continuation path \(\Gamma_{\bm{\omega}^0}\) originating from \((\bm{\omega}^0,1)\). As the homotopy parameter approaches zero, the continuation trajectory \del{converges to a nonempty limit set}
	\add{has a nonempty set of accumulation points} whose \(\bm{x}\)-components satisfy the KKT conditions of the constrained MOO problem~\eqref{MOO}. The result highlights the importance of constructing an appropriate homotopy map that smoothly interpolates between an initial well-posed system and the original KKT system, \del{thereby ensuring stable continuation trajectories
		and robust convergence behavior.}
	\add{thereby providing a theoretical basis for
		continuation toward the KKT set under the stated
		assumptions.}
	
	The proposed framework can be interpreted as a constraint-based predictor--corrector homotopy continuation method for constrained multiobjective nonlinear optimization. The homotopy formulation directly embeds the KKT conditions, including stationarity, feasibility, complementarity, and normalization conditions, into a unified continuation system $H(\bm{x},\textbf{w},\bm{u},\bm{v},t)=0,$
	where \(t\in[0,1]\) denotes the homotopy parameter. Unlike scalarization-based approaches that repeatedly solve independent optimization problems for different parameter selections, the proposed framework \del{continuously follows branches of KKT points
		as $t\to0$.}
	\add{follows a homotopy path toward the KKT set
		as $t\to0$.} The complementarity structure
	\[
	u_i g_i(\bm{x})-t u_i^0 g_i(\bm{x}^0)=0
	\]
	establishes strong connections with complementarity homotopies and interior-point continuation methods. However, unlike classical interior-point algorithms that typically converge to a single minimizer, the proposed continuation framework \del{is designed to explore complex solution structures,
		including disconnected and nonconvex branches.}
	\add{incorporates numerical strategies for exploring
		complex nonconvex and disconnected solution structures,
		without guaranteeing that a single continuation path
		recovers all such branches.}
	
	\section{Numerical Continuation Algorithms}
	\label{sec:alg}
	
	This section presents two numerical continuation algorithms used in the proposed homotopy framework. 
	Algorithm~\ref{alg:homotopy} gives the basic predictor--corrector continuation method, which is suitable for relatively smooth and connected solution branches. Algorithm~\ref{alg:robust_homotopy} extends this method to more difficult MOO problems involving nonconvex feasible regions, degeneracy, near-singular Jacobians, and disconnected Pareto-critical structures. 
	The main differences between the two algorithms are the tangent-orientation strategy, the correction mechanism, the use of pseudo-arclength continuation, adaptive step-size control, and branch restarting.
	\add{For $p$ objectives, a simplex-lattice discretization
		of the previously defined positive weight simplex
		$W^{++}$ is used to construct a finite set
		$W=\{\mathbf{w}_1,\ldots,\mathbf{w}_P\}
		\subset W^{++}$ for Algorithms~\ref{alg:homotopy}
		and~\ref{alg:robust_homotopy}.}
	
	\add{A regular distribution of weights in $W^{++}$ does not, in general,
		guarantee a uniform distribution of the corresponding Pareto-critical
		solutions. The relationship between weights and Pareto-critical points
		may be nonlinear and nonunique, so different weights may
		lead to the same or nearby solutions. This effect can be particularly pronounced for
		nonconvex or degenerate MOOs. Therefore, the weight discretization is
		used as a sampling strategy rather than as a guarantee of uniform
		Pareto-front coverage. For disconnected or degenerate problems,
		Algorithm~2 additionally employs multiple branch/restart initializations
		to improve the exploration of distinct Pareto-critical components.}
	
	\add{The use of positive weights in the homotopy formulation should not be
		confused with classical weighted-sum scalarization. Here, the weights
		enter the Pareto-criticality/KKT system and the associated homotopy
		construction; the algorithm does not independently minimize a weighted
		sum of the objectives for each weight. 
		Consequently, unlike global weighted-sum minimization,
		the KKT formulation is not inherently restricted to
		supported Pareto-optimal points and can admit
		Pareto-critical points associated with nonconvex
		portions of the Pareto front. Nevertheless, this does not guarantee that every
		nonconvex portion or every connected component of the Pareto front will
		be reached from a finite set of weights and initializations.}
	
	\subsection{Basic Predictor--Corrector Homotopy Continuation}
	
	\begin{algorithm}[h]
		\caption{Homotopy Continuation Method (Predictor-Corrector)}
		\label{alg:homotopy}
		\begin{algorithmic}[1]
			\STATE Objectives $\bm{f} \in \R^p$, Constraints $\bm{g} \in \R^m,~ \bm{h} \in \R^s$
			\STATE Initial guess $\bm{x}_0 \in \R^n$, step size $\alpha_0$, tolerance $\epsilon$, and parameters $\bm{u}_0, \bm{v}_0, t_0$
			\STATE \add{Generate $\bb{W} = \{\bb{w}_1, \bb{w}_2, \dots, \bb{w}_P\} \subset W^{++}$ using the prescribed simplex-lattice discretization}
			\FOR{\textbf{each} $\bb{w} \in \bb{W}$}
			\STATE Define state vector $\omegaVec \gets [\bm{x}_0, \bb{w}, \bm{u}_0, \bm{v}_0, t_0]^T$
			\STATE $k \gets 1$ 
			\add{\STATE $\alpha \gets \alpha_0$}
			\WHILE{$|t| > \epsilon$ \textbf{and} $k \leq k_{max}$}
			\STATE Evaluate the homotopy and its Jacobian:
			$H \gets H(\omegaVec),\quad J_H \gets D H(\omegaVec).$
			\STATE \textbf{Phase 1: Prediction (Tangent Step)}
			\STATE Compute tangent vector $\xi \in \text{null}(J_H)$ such that $J_H \cdot \bm{\xi} = 0$
			\STATE Normalize: $\xi \gets \xi / \norm{\xi}$
			
			\STATE \textit{Orientation Check:}
			\IF{$\det\!\left( \begin{bmatrix} J_H \\ \bm{\xi}^\top \end{bmatrix} \right) < 0$}
			\STATE $\bm{\xi} \gets -\bm{\xi}$
			\ENDIF       
			\STATE Predict next point: $\omegaVec_{pred} \gets \omegaVec + \alpha \cdot \bm{\xi}$
			\STATE Adapt step size:        
			\[
			\alpha \gets
			\begin{cases}
				\min(\alpha_{\max},2\alpha),
				& \|H\|<0.01,\\
				\max(\alpha_{\min},\alpha/2),
				& \|H\|>1,\\
				\alpha,
				& \text{otherwise}.
			\end{cases}
			\]
			\add{where $\alpha_{\min}$ is a minimum allowable step size. Alternatively, use simply $\alpha/2$ if no minimum step size is imposed.}
			\STATE \textbf{Phase 2: Correction (Single Newton-Raphson)}
			\STATE Evaluate $H(\omegaVec_{pred})$ and Jacobian $J_{pred} = DH(\omegaVec_{pred})$        
			\STATE Compute the minimum-norm Newton correction:
			\[
			\Delta\omegaVec
			\gets
			-J_{\mathrm{pred}}^\dagger H(\omegaVec_{pred}),
			\text{ where } J_{\mathrm{pred}}^\dagger \text{  denotes
				the Moore--Penrose pseudoinverse}.\]
			\STATE Update state: $\omegaVec \gets \omegaVec_{pred} + \Delta \omegaVec$
			\STATE Extract $t$ from $\omegaVec$
			\STATE $k \gets k + 1$
			\ENDWHILE   
			\IF{$|t| \leq \epsilon$}
			\STATE Extract $ \bm{x}^*$ from $\omegaVec$
			\STATE $P_{set} \gets P_{set} \cup \{ \bm{x}^*, \bm{f}(\bm{x}^*)\}$
			\ENDIF
			\ENDFOR
			\STATE $P_{set}$: \add{candidate KKT points}
		\end{algorithmic}
	\end{algorithm}
	Although the homotopy framework is theoretically designed to converge to KKT points under the assumptions stated earlier, its numerical implementation requires careful path tracking. 
	The basic continuation algorithm follows a standard predictor--corrector strategy. 
	At each step, a tangent direction is computed from the null space of the homotopy Jacobian, and an Euler-type prediction is used to advance along the solution path. \add{
		Assuming that the homotopy Jacobian has full row rank,
		the determinant-based orientation test is well defined
		regardless of the number of objectives.
		If
		$J_H\in\mathbb{R}^{N\times(N+1)}$ is the Jacobian of the homotopy
		system and $\boldsymbol{\xi}\in\mathbb{R}^{N+1}$ is its normalized null vector, then
		\[
		\begin{pmatrix}
			J_H\\
			\boldsymbol{\xi}^{T}
		\end{pmatrix}
		\in\mathbb{R}^{(N+1)\times(N+1)}.
		\]
		Thus, the determinant-based orientation test applies for an arbitrary
		number $p$ of objectives; the dimension $N$ changes with the
		corresponding homotopy system.} 
	A Moore--Penrose Newton correction is then applied to project the predicted point back onto the homotopy manifold. 
	\add{This approach is suitable for well-behaved problems
		where the homotopy path is smooth and regular,
		without severe turning points or near-singular
		Jacobians.}
	
	\subsection{Enhanced Continuation for Nonconvex and Disconnected Problems}
	\begin{algorithm}[htbp]
		\caption{Robust Pseudo-Arclength Homotopy Continuation Method}
		\label{alg:robust_homotopy}
		\begin{algorithmic}[1]
			\STATE Given objectives $\bm f\in\mathbb{R}^p$, inequality constraints $\bm g\in\mathbb{R}^m$, equality constraints $\bm h\in\mathbb{R}^s$.
			\STATE Choose branch restart points $\mathcal X_0=\{\bm x_0^{(1)},\dots,\bm x_0^{(R)}\}$, \add{Generate} weights $\mathcal W=\{\bm w^{(1)},\dots,\bm w^{(N_w)}\}$\add{$\subset
				W^{++}$} , initial step size $\alpha_0$, tolerance $\epsilon$, and maximum iterations $k_{\max}$.
			\STATE Initialize the Pareto-critical solution set $P_{\mathrm{set}}\gets\emptyset$.
			
			\FOR{each restart point $\bm x_0^{(r)}\in\mathcal X_0$}
			\FOR{each weight vector $\bm w\in\mathcal W$}
			\STATE Initialize
			$\bm{\omega}_0=\left[\bm x_0^{(r)},\bm w,\bm u_0,\bm v_0,t_0\right]^T,
			\qquad t_0=1.$    
			
			\STATE \add{Set $\bm{\omega}\gets\bm{\omega}_0$
				and $\alpha\gets\alpha_0$.}
			\STATE \add{Compute an initial unit tangent
				$\xi_{\mathrm{old}}\in
				\operatorname{null}(DH(\bm{\omega}_0))$,
				with orientation chosen toward the desired
				continuation branch.}
			
			\STATE Set $k\gets0$.
			
			\WHILE{$|t|>\epsilon$ and $k<k_{\max}$}
			\STATE Evaluate the homotopy system and Jacobian: $H(\bm{\omega}), J_H(\bm{\omega})=DH(\bm{\omega}).$
			\STATE \textbf{Prediction step.}
			\STATE Compute a tangent vector
			$\xi\in\operatorname{null}(J_H(\bm{\omega})),
			\qquad \|\xi\|=1.$
			\STATE Orient the tangent continuously:
			\add{$\text{if } \xi^T\xi_{\mathrm{old}}<0,\quad \xi\gets-\xi.$}
			\STATE Predict $\bm{\omega}_{\mathrm{pred}}=\bm{\omega}+\alpha\xi.$
			
			\STATE \textbf{Correction step: pseudo-arclength Newton iteration.}
			\STATE Set $\bm{\omega}^{(0)}\gets\bm{\omega}_{\mathrm{pred}}$.
			\FOR{$j=0,1,\dots,j_{\max}$}
			\STATE Form the augmented residual and Jacobian
			$\small{R(\bm{\omega}^{(j)})=
				\begin{bmatrix}
					H(\bm{\omega}^{(j)})\\
					\xi^T(\bm{\omega}^{(j)}-\bm{\omega}_{\mathrm{pred}})
				\end{bmatrix},          
				J_{\mathrm{aug}}(\bm{\omega}^{(j)})=
				\begin{bmatrix}
					DH(\bm{\omega}^{(j)})\\
					\xi^T
				\end{bmatrix}.}$
			\STATE Solve $J_{\mathrm{aug}}(\bm{\omega}^{(j)})\Delta\bm{\omega}
			=
			-R(\bm{\omega}^{(j)}).$
			\STATE Update
			$\bm{\omega}^{(j+1)}=\bm{\omega}^{(j)}+\Delta\bm{\omega}.$     \IF{$\|R(\bm{\omega}^{(j+1)})\|\leq\epsilon_{\mathrm{Newton}}$}
			\STATE break
			\ENDIF
			\ENDFOR
			
			\IF{\add{the corrected point is finite and
					$\|R(\bm{\omega}^{(j+1)})\|
					\leq\epsilon_{\mathrm{Newton}}$}}
			\STATE \add{Accept
				$\bm{\omega}_{\mathrm{new}}
				\gets\bm{\omega}^{(j+1)}$.}
			\ELSE
			\STATE \add{Reduce $\alpha$ and retry the
				prediction--correction step without
				updating $\bm{\omega}$.}
			\STATE \textbf{continue}
			\ENDIF
			
			\STATE Adapt the step size $\alpha$: increase $\alpha$ after fast Newton convergence and decrease $\alpha$ after slow convergence or failed correction.
			\STATE Set $\xi_{\mathrm{old}}\gets\xi$, $\bm{\omega}\gets\bm{\omega}_{\mathrm{new}}$, and $k\gets k+1$.
			\ENDWHILE
			\IF{$|t|\leq\epsilon$}
			\STATE Extract $\bm{x}^*$ from $\bm{\omega}$.
			\STATE \add{Store the candidate Pareto-critical point:}
			$P_{\mathrm{set}}\gets P_{\mathrm{set}}\cup\{(\bm{x}^*,\bm f(\bm{x}^*))\}.$
			\ENDIF
			\ENDFOR
			\ENDFOR
			\STATE Return $P_{\mathrm{set}}$.
		\end{algorithmic}
	\end{algorithm}
	For more challenging MOO problems, the basic predictor--corrector method may become unstable or may fail to follow disconnected or highly curved solution structures. 
	Therefore, Algorithm~\ref{alg:robust_homotopy} introduces several modifications.
	
	First, the determinant-based orientation check in Algorithm~\ref{alg:homotopy} is replaced by a continuous tangent-orientation rule. 
	Instead of relying on the sign of a determinant, the new method compares the current tangent direction with the previous tangent direction and reverses the orientation whenever their inner product is negative. 
	This reduces abrupt changes in direction and improves path-tracking stability.
	
	Second, the one-step Moore--Penrose correction is replaced by a multi-step Newton corrector applied to an augmented pseudo-arclength system. These continuation enhancements are related to modern predictor--corrector continuation approaches for multiobjective optimization, including the Pareto Tracer framework developed by~\citet{martin2018pareto}.
	The additional condition
	\[
	\xi^T(\bm{\omega}-\bm{\omega}_{\mathrm{pred}})=0
	\]
	keeps the corrected iterate on a hyperplane normal to the tangent direction at the predicted point. 
	\add{
		This improves robustness near folds and turning points,
		where the homotopy parameter alone may not serve as
		a reliable local continuation parameter, and can
		also improve numerical behavior in highly nonlinear
		regions. However, genuinely singular or bifurcation
		points may require additional numerical treatment.
	}
	
	Third, Algorithm~\ref{alg:robust_homotopy} incorporates adaptive step-size control. 
	The step size is increased after fast Newton convergence and decreased when the corrector converges slowly or fails. 
	This allows the method to take larger steps on smooth portions of the branch while using smaller steps near difficult regions.
	
	\add{
		Finally, multiple initializations from distinct feasible
		reference points are employed to improve the exploration
		of disconnected feasible regions and Pareto-critical
		solution structures. Since a local continuation method
		cannot automatically move between disconnected
		components, these independent continuation runs
		provide opportunities to reach different components,
		although complete coverage is not guaranteed.
	}
	Since a local continuation method cannot automatically jump between disconnected branches, the restart strategy allows the algorithm to explore multiple components of the solution structure.
	
	Overall, Algorithm~\ref{alg:homotopy} is intended for standard continuation on smooth connected branches, while Algorithm~\ref{alg:robust_homotopy} is designed for more difficult many-objective, nonconvex, degenerate, and disconnected constrained optimization problems. 
	\add{
		Both algorithms seek candidate Pareto-critical points
		rather than guaranteed globally Pareto-optimal solutions,
		since the homotopy framework approaches the KKT system
		as the homotopy parameter tends to zero.
	}
	
	\section{Computational Experiments}
	\label{sec:numerical}
	This section presents several computational experiments designed to evaluate the proposed homotopy continuation framework on constrained multiobjective optimization problems with varying levels of complexity. The experiments include benchmark problems, nonconvex and disconnected many-objective problems, and an IEEE 5-bus optimal power flow application.
	\del{All numerical experiments were conducted in Google Colab
		using a 2021 MacBook Pro equipped with an Apple M1 Pro
		processor and 16 GB of unified memory.}
	\add{All numerical experiments were conducted locally
		on a Windows laptop equipped with a 13th-generation
		Intel Core i7-1355U processor (1.70 GHz) and 16 GB
		of RAM. 
		The numerical algorithms and computational experiments
		were implemented in Python~3.14 using scientific
		computing, numerical optimization, and data
		visualization libraries.
		All computations were performed in double precision,
		and consistent numerical tolerances were used
		where applicable to facilitate comparisons
		across methods.} \add{The Python implementation and computational scripts
		are available at
		\url{https://github.com/guangmingyao/Homotopy-MOO}.}
	
	To evaluate the proposed homotopy continuation framework, we compare it with several widely used MOO approaches, including the Weighted Sum Method (WSM), the $\epsilon$-Constraint Method (ECM), the Global Criterion Method (GCM), the Lexicographic Method, and the evolutionary algorithm NSGA-II, \add{NSGA-III and the constrained two-archive evolutionary algorithm (C-TAEA)}. These methods remain popular in engineering and scientific applications due to their simplicity, interpretability, and established computational frameworks.
	
	\add{Suitable methods are selected to test the benchmark problems using consistent stopping tolerances, feasibility checks, and evaluation metrics to ensure fair comparison.
		For the deterministic scalarization methods
		(Weighted Sum, $\epsilon$-Constraint,
		Global Criterion, and Lexicographic optimization),
		constrained single-objective subproblems are solved
		to obtain candidate Pareto-optimal or Pareto-critical
		solutions, depending on the scalarization method
		and numerical solution accuracy.
	} The proposed homotopy framework is implemented as a continuation scheme applied directly to the KKT system. Although the lexicographic method does not generate a full Pareto set, it is included here as a classical deterministic baseline for comparison of representative Pareto-critical solutions and tradeoff behavior. NSGA-II is included as a stochastic \textit{a posteriori} baseline that generates a finite nondominated approximation of the Pareto front. Across all examples, we report Pareto-critical solutions, computational cost, and objective trade-off behavior.

	\add{To quantify the quality of the computed nondominated
		sets, let $A=\{\bm{a}_1,\ldots,\bm{a}_{N_A}\}$
		denote the normalized objective vectors obtained by
		a given method and let
		$R=\{\bm{r}_1,\ldots,\bm{r}_{N_R}\}$
		denote the normalized empirical reference front, when an analytical front is available, \(R\) describes a sampled analytical reference front.
		For a point $\bm{y}$ and a finite set $S$, define
		\[d(\bm{y},S)=\min_{\bm{s}\in S}\|\bm{y}-\bm{s}\|_2.\]
		The following Pareto-front quality measures, including the inverted generational distance (IGD), the hausdorff distance, the spacing coefficient or variation (Spacing CV), the maximum gap, and the coverage, are used:}
	\begin{itemize}
		\item\add{The number of nondominated solutions produced by a method after filtering the computed feasible solutions for Pareto dominance is denoted by $N_{\mathrm{ND}}$.} 
		
		\item \add{The maximum equality-constraint residual is}
		\[
		V_{\mathrm{eq}}
		=
		\max_{\bm{x}\in A_{\mathrm{dec}}}
		\max_{1\leq j\leq s}|h_j(\bm{x})|.
		\]
		Here \(A_{\mathrm{dec}}\) denote the decision vectors corresponding to the reported nondominated objective vectors in \(A\).
		\item \add{The maximum inequality-constraint violation is}
		\[
		V_{\mathrm{ineq}}
		=
		\max_{\bm{x}\in A_{\mathrm{dec}}}
		\max_{1\leq i\leq m}
		\max\{g_i(\bm{x}),0\}.
		\]

		\item
		\add{ \textbf{The inverted generational distance (IGD)}\,\citet{van1999multiobjective,15892}: }
		\[\operatorname{IGD}(A,R) =
		\frac{1}{N_R}
		\sum_{\bm{r}\in R}d(\bm{r},A)\]
		measures how well the computed set represents the reference front.
		Smaller values indicate better overall approximation and coverage.
		\smallskip 
		\item
		\add{The \textbf{Hausdorff distance\,\citet{6151115}:}}
		\[d_H(A,R)
		=
		\max\left\{
		\max_{\bm{a}\in A}d(\bm{x},R),
		\;
		\max_{\bm{r}\in R}d(\bm{r},A)
		\right\}\]
		measures the largest discrepancy between the computed and reference
		sets. Smaller values are preferable.
		\smallskip 
		\item
		\add{The \textbf{spacing coefficient of variation (Spacing CV)} is evaluated for sets containing
			at least two distinct points with positive mean
			nearest-neighbor distance}
		\[
		\operatorname{CV}_{\mathrm{sp}}
		=
		\frac{\operatorname{std}(d_1,\ldots,d_{N_A})}
		{\operatorname{mean}(d_1,\ldots,d_{N_A})},
		\]
		where $d_i=\|\bm a_{(i+1)}-\bm a_{(i)}\|_2$,
		$i=1,\ldots,N_A-1$, is the distance between
		consecutive normalized objective vectors ordered
		by the first objective.
		Smaller values indicate a more uniform distribution of
		computed points.
		\smallskip 
		
		\item
		\add{The \textbf{maximum gap} measures the largest
			Euclidean distance between consecutive points on
			a bi-objective front. After ordering the normalized
			objective vectors by their first objective,
			\[
			G_{\max}(A)
			=
			\max_{1\le i<N_A}
			\|\bm a_{(i+1)}-\bm a_{(i)}\|_2.
			\]
			Smaller values indicate fewer large gaps between
			consecutive computed points. This measure is used
			for the bi-objective problem in Example~1.}
		
		\smallskip 
		
		\item
		\add{The \textbf{objective-range coverage} measures
			the fraction of the reference-front range spanned
			by the computed nondominated set in each objective:
			\[
			C_j(A,R)
			=
			\frac{
				\max_{\bm a\in A}a_j-\min_{\bm a\in A}a_j
			}{
				\max_{\bm r\in R}r_j-\min_{\bm r\in R}r_j
			},
			\qquad j=1,\ldots,p.
			\]
			The denominator is assumed positive. Values
			closer to $1$ indicate that the computed set spans
			a range comparable to that of the reference front.
			Coverage is reported separately for each objective
			and does not necessarily imply uniform coverage
			of the entire Pareto front.}

		\item 
		\add{Hypervolume (HV)~\citet{zitzler2007hypervolume,797969}
			measures the volume of the objective space dominated by
			the obtained nondominated solutions and bounded by a
			prescribed reference point. For a minimization problem,
			it is defined as}
		\begin{equation}
			\mathrm{HV}(S;\bm r)
			=
			\operatorname{Vol}\left(
			\bigcup_{\bm f\in S}
			\prod_{i=1}^{p}[f_i,r_i]
			\right),
		\end{equation}
		\add{where $S$ is the set of nondominated objective vectors,
			$\bm r=(r_1,\ldots,r_m)^T$ is a reference point satisfying
			$r_i\geq f_i$ for all $\bm f\in S$, and
			$\operatorname{Vol}$ denotes the $p$-dimensional
			Lebesgue measure. A larger HV indicates a larger
			dominated region and generally reflects better
			convergence and/or diversity.}
		
	\end{itemize}
	
	\begin{table}[htbp]
		\centering
		\caption{Comparison of multiobjective optimization methods.}
		\label{tab:comparison_methods}
		\renewcommand{\arraystretch}{1.2}
		\small
		
		\begin{tabular}{p{1.7cm} p{1.5cm} p{2.7cm} p{8.2cm}}
			\hline
			\textbf{Method}
			& \textbf{Category}
			& \textbf{Mechanism}
			& \textbf{Solution-set characteristics} \\
			\hline
			
			Homotopy
			& Continuation method
			& Continuation applied to the KKT system
			& Computes Pareto-critical solutions satisfying the KKT
			conditions; recovery of particular branches depends
			on initialization and continuation behavior \\
			\hline
			
			Weighted Sum
			& Deterministic scalarization
			& Linear weighted aggregation with $\mathbf{w}$
			& Varying positive weights can recover supported
			Pareto-optimal solutions; unsupported portions of
			nonconvex Pareto fronts may be missed \\
			\hline
			
			$\epsilon$-Constraint
			& Deterministic scalarization
			& Optimize one objective while bounding the others
			& Can recover supported and unsupported Pareto-optimal
			solutions by varying the objective bounds, provided
			the scalar subproblems are solved globally \\
			\hline
			
			Global Criterion
			& Deterministic scalarization
			& Minimize distance to an ideal point $\bm{f}^{\ast}$
			& Produces compromise solutions determined by the
			selected distance metric, objective scaling, and
			reference point \\
			\hline
			
			Lexicographic
			& Deterministic scalarization
			& Sequential optimization according to objective priorities
			& Produces solutions determined by the prescribed
			objective-priority ordering \\
			\hline
			
			NSGA-II
			& Evolutionary algorithm
			& Nondominated sorting and crowding-distance selection
			& Produces a population-based approximation of the
			Pareto front using convergence and diversity mechanisms \\
			\hline
			
			\add{NSGA-III}
			& \add{Evolutionary algorithm}
			& \add{Reference-direction-based nondominated sorting}
			& \add{Uses reference directions to promote diversity
				in many-objective optimization and approximate
				different regions of the Pareto front} \\
			\hline
			
			\hline
			\add{C-TAEA}
			& \add{Constrained evolutionary algorithm}
			& \add{Two-archive evolutionary search}
			& \add{Uses complementary archives to balance
				convergence, diversity, and constraint handling
				in constrained multiobjective optimization} \\
			\hline
			
			\add{MOEA/D}
			& \add{Evolutionary algorithm}
			& \add{Decomposition into scalar optimization subproblems}
			& \add{Approximates the Pareto front through a collection
				of neighboring scalar subproblems; coverage depends
				on the decomposition and weight/reference-vector design} \\
			\hline
			
		\end{tabular}
	\end{table}

	\del{Table~\ref{tab:comparison_methods} summarizes representative MOO approaches and highlights their differences in formulation, solution mechanisms, and typical Pareto coverage. Classical scalarization methods transform a multiobjective problem into one or more single-objective optimization problems whose KKT systems generate Pareto-critical solutions. The Pareto regions recovered depend strongly on the scalarization strategy, with weighted-sum methods typically restricted to convex regions of the Pareto front. In contrast, homotopy continuation methods operate as deterministic path-following procedures applied directly to the KKT system, while NSGA-II produces stochastic population-based approximations of the Pareto front.}
	\add{Table~\ref{tab:comparison_methods} summarizes representative MOO
		approaches and distinguishes their solution mechanisms and the types of
		Pareto solutions they are designed to recover. In the weighted-sum
		method, varying positive objective weights generates supported
		Pareto-optimal solutions; consequently, non-supported portions of a
		nonconvex Pareto front may not be recovered. The
		$\epsilon$-constraint method instead optimizes one objective while
		placing bounds on the others and, by varying these bounds, can recover
		both supported and non-supported Pareto-optimal solutions. The global
		criterion and lexicographic methods generate compromise or
		priority-dependent solutions according to their respective
		scalarizations. The proposed homotopy method applies continuation to the KKT system to compute Pareto-critical solutions, whereas NSGA-II uses a population-based stochastic search to approximate the nondominated solution set. } 
	
	Additional implementation details for the comparison algorithms are
	provided in Appendices~\ref{appendix:A}--\ref{appendix:B}.
	
	\add{The following four examples are designed to evaluate complementary
		aspects of the proposed homotopy continuation method, progressing from
		a constrained benchmark problem to degenerate and disconnected
		many-objective problems, a practical engineering application, and
		controlled problems with analytically known Pareto fronts. In
		particular,}
	
	\begin{itemize}
		\item \textbf{Example~1:} Constrained bi-objective benchmark problem.
		
		This example compares the proposed method with existing deterministic
		approaches and NSGA-II. It examines the quality and distribution of
		the computed Pareto-front approximation, constraint satisfaction,
		and sensitivity to the equality-constraint tolerance and initial
		point.
		
		\item \textbf{Example~2:} Degenerate and disconnected four-objective
		problem.
		
		This example tests the method on a more challenging many-objective
		problem having disconnected feasible components and a degenerate
		Pareto structure. It is used to assess the continuation and
		multiple-initialization strategy, Pareto-front coverage, and sensitivity
		to the construction of the numerical reference set.
		
		\item \textbf{Example~3:} Three-objective IEEE 5-bus optimal power
		flow problem.
		
		This example evaluates the method on a constrained engineering
		application with multiple competing objectives. It provides a test
		of the proposed approach on a problem for which an analytical Pareto
		front is not available and permits comparison with modern
		many-objective evolutionary algorithms.
		
		\item \textbf{Example~4:} Three-objective problems with analytically
		known Pareto fronts.
		
		\add{Two closely related problems with convex and concave Pareto-front
			geometries are considered. Because their Pareto sets, Pareto fronts,
			and interior Pareto-critical solutions are known analytically, this
			example permits direct accuracy assessment. The concave case also
			provides a controlled investigation of unsupported Pareto-critical
			points and the effect of initialization on homotopy branch selection.}
	\end{itemize}
	
	\add{The numerical comparisons are tailored to the purpose of each example.
		Where an analytical Pareto front is available, the computed solutions
		are compared directly with the exact front; otherwise, an empirical
		reference set is used. 
		The comparisons use appropriate measures of
		solution quality, feasibility, computational cost,
		coverage, and distribution, depending on the
		characteristics of each example.
		Modern evolutionary multiobjective algorithms,
		including NSGA-III, C-TAEA, and MOEA/D, are included where appropriate. The
		purpose of these experiments is not to establish universal superiority
		of one method, but to assess the accuracy, feasibility, coverage,
		robustness, and limitations of the proposed continuation approach under
		different objective, constraint, and Pareto-front structures.}
	\begin{example}[\textbf{Constrained Bi-Objective Benchmark Problem}] 
		\label{EX2}
		
		This example serves as a benchmark validation problem for the proposed framework. The experiments investigate feasible-region structure, comparison with classical scalarization and evolutionary methods, constraint feasibility of reported literature solutions, computational efficiency, and robustness with respect to initialization.
		Consider the MOO problem:
		\[
		\begin{aligned}
			\min_{x \in \mathbb{R}^5} \quad 
			\begin{bmatrix}
				f_1(\bm{x}) \\
				f_2(\bm{x})
			\end{bmatrix}  
			&=
			\begin{bmatrix}
				x_1^2 + x_2^2 + x_3^2 + x_4^2 + x_5^2 \\
				3x_1 + 2x_2 - \tfrac{1}{3}x_3 + 0.01(x_4 - x_5)^3
			\end{bmatrix}, \\\\
			\text{s.t.} \quad
			\begin{bmatrix}
				h_1(\bm{x}) \\
				h_2(\bm{x})
			\end{bmatrix}  &= 
			\begin{bmatrix}
				4x_1 - 2x_2 + 0.8x_3 + 0.6x_4 + 0.5x_5^2 \\
				x_1 + 2x_2 - x_3 - 0.5x_4 + x_5 - 2
			\end{bmatrix}
			= 0.\\
			g(\bm{x}) &= x_1^2 + x_2^2 + x_3^2 + x_4^2 - 10 \leq 0.
		\end{aligned}
		\]
	\end{example}
	
	\medskip
	\noindent \textbf{Feasible Region Approximation}

	\del{Due to the sparsity of feasible points under direct random
		sampling, a projection-based strategy is employed to
		approximate the feasible decision space.}
	\add{Since the equality constraints restrict the feasible
		set to a lower-dimensional subset of the decision space,
		uniform random sampling is unlikely to produce points
		satisfying these constraints exactly. \add{This difficulty is also highlighted
			in the benchmark study of equality-constrained multiobjective
			optimization by Cuate et al.~\citet{cuate2020benchmark}.} Therefore, a
		projection-based strategy is employed to approximate
		the feasible decision space.}
	Uniform samples $\bm{x}_0 \in [-5,5]^5$ are generated and projected onto the feasible set by solving
	\begin{equation}
		\label{eq:projection-problem}
		\begin{aligned}
			\min_{\bm{z} \in \mathbb{R}^5} \quad & \|\bm{z} - \bm{x}_0\|_2^2 \\
			\text{s.t.} \quad & g(\bm{z}) \le 0, \\
			& h_1(\bm{z}) = 0, \\
			& h_2(\bm{z}) = 0 .
		\end{aligned}
	\end{equation}
	
	\begin{figure}[h]
		\hspace*{-0.5cm}
		\centering
		\includegraphics[scale=0.35]{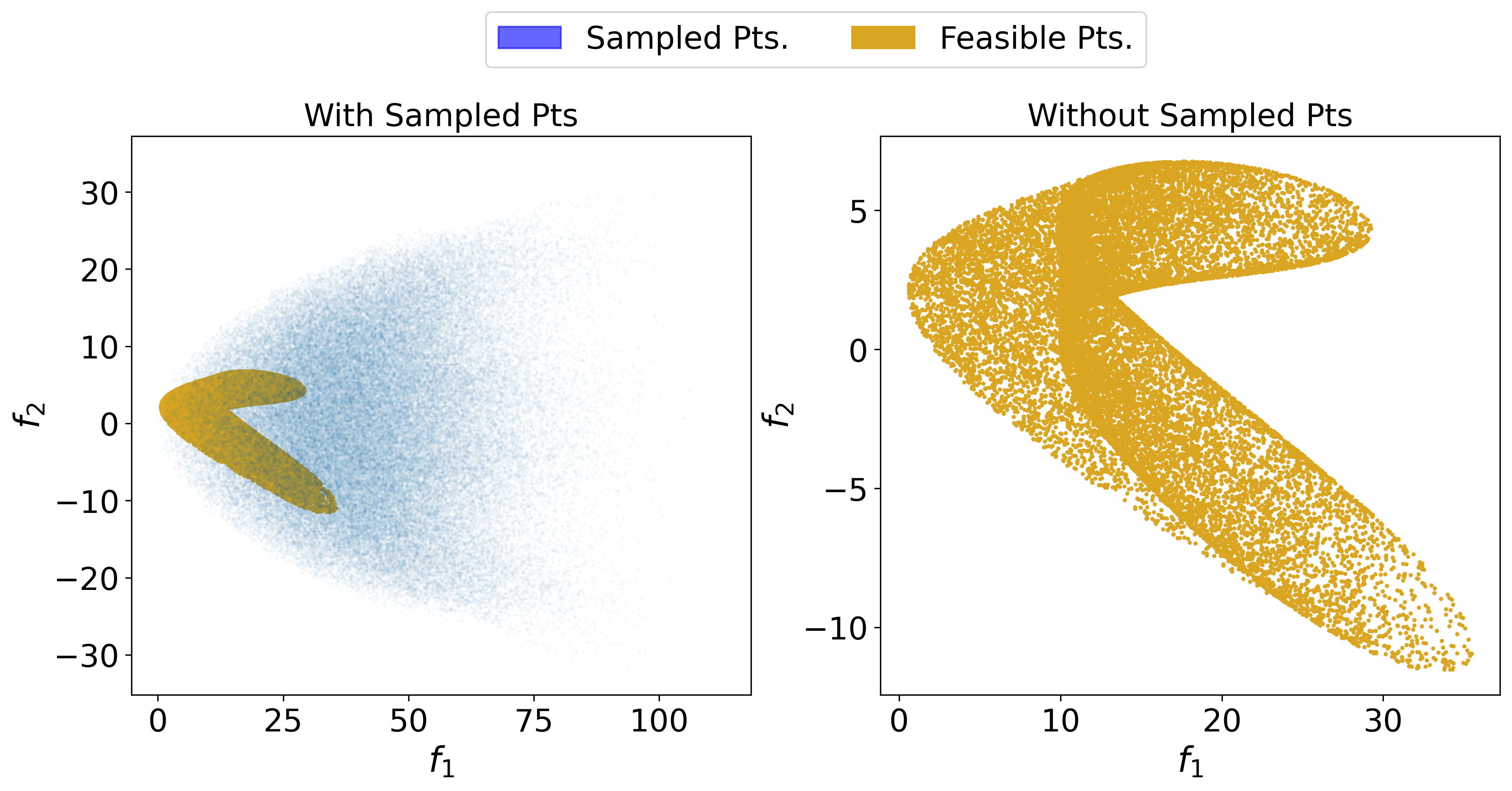}
		\caption{Example~\ref{EX2}: \add{Approximation of the feasible
				objective-space region.} Blue points \add{in the left panel} represent $200{,}000$ uniformly sampled diagnostic points from the hypercube $[-5,5]^5$; \add{while } golden points denote $20{,}000$ feasible candidate solutions obtained by projecting onto the constraint set. The \del{left panel displays all diagnostic samples (faint blue) together with the projected feasible candidates (gold), while the} right panel shows only the projected feasible candidates \del{and the corresponding homotopy-based Pareto front.} \add{Pareto-critical
				solutions obtained by the different MOO methods are compared separately
				in Fig.~\ref{fig:pareto_plots_ex1}.}}
		\label{Fig:points}
	\end{figure}
	The resulting optimization problem identifies the nearest feasible point to $\bm{x}_0$ in the Euclidean sense and is solved using the Sequential Least Squares Programming (SLSQP) algorithm. The projected feasible points are then used to approximate the feasible objective-space structure and provide a reference for subsequent Pareto-critical solution comparisons. 
	Repeating this procedure over a large number of samples yields an empirical approximation of the feasible \del{set}\add{objective-space region}, which serves as a reference for subsequent \del{Pareto-structure identification}\add{comparison of Pareto-critical
		solutions}. In Fig.~\ref{Fig:points}, \del{blue points represent uniformly sampled decision variables, while golden points denote feasible solutions obtained through projection onto the constraint set. The left panel illustrates the relationship between the sampled and projected points, whereas the right panel isolates the feasible solutions together with the homotopy-generated Pareto-critical structures.} \add{the left panel shows the diagnostic samples together with the projected feasible candidates, whereas the right panel isolates the projected feasible candidates. The
		Pareto-critical solutions produced by the different MOO methods are
		shown and compared separately in Fig.~\ref{fig:pareto_plots_ex1}.} 
	
	\medskip \noindent \textbf{Pareto-critical Solution Comparison}
	
	Fig.\,\ref{fig:pareto_plots_ex1} \add{illustrates} the Pareto structure and numerical results for Example\,\ref{EX2}. This compares the Pareto solutions obtained by the proposed homotopy,  weighted sum, global criterion, NSGA-II, lexicographic, and $\epsilon-$constraint methods. \del{The weighted sum and homotopy approaches recover identical Pareto-critical points, confirming that the homotopy method accurately solves the KKT system of the weighted-sum formulation.} 
	\add{For a prescribed objective-weight vector $\bm{w}$, the endpoint of the proposed homotopy satisfies the corresponding weighted KKT
		system of the original constrained MOO problem. The direct weighted-sum
		method and the proposed homotopy therefore target solutions satisfying
		the same weighted stationarity conditions when the same weights are
		used. Their computed points need not, however, coincide numerically:
		in a nonconvex problem the weighted KKT system may admit multiple
		solutions, and the direct optimization and continuation procedures may
		converge to different KKT points. Accordingly, Fig.~\ref{fig:pareto_plots_ex1}
		should be interpreted as a comparison of the Pareto-critical structures
		recovered by the different numerical procedures, rather than as evidence
		that the weighted-sum and homotopy methods produce identical points.
	} 
	The global criterion and $\epsilon-$constraint methods generate \del{continuous Pareto curves} \add{approximations of Pareto curves,} while NSGA-II produces a dense but stochastic approximation of the Pareto front. In contrast, the lexicographic method yields only extreme solutions due to its strict priority structure. These results highlight the deterministic and structure-preserving nature of the homotopy approach relative to both classical scalarization techniques and evolutionary methods.

	\begin{figure}[htb]
		\hspace*{-0.5cm}
		\centering
		\includegraphics[scale=0.35]{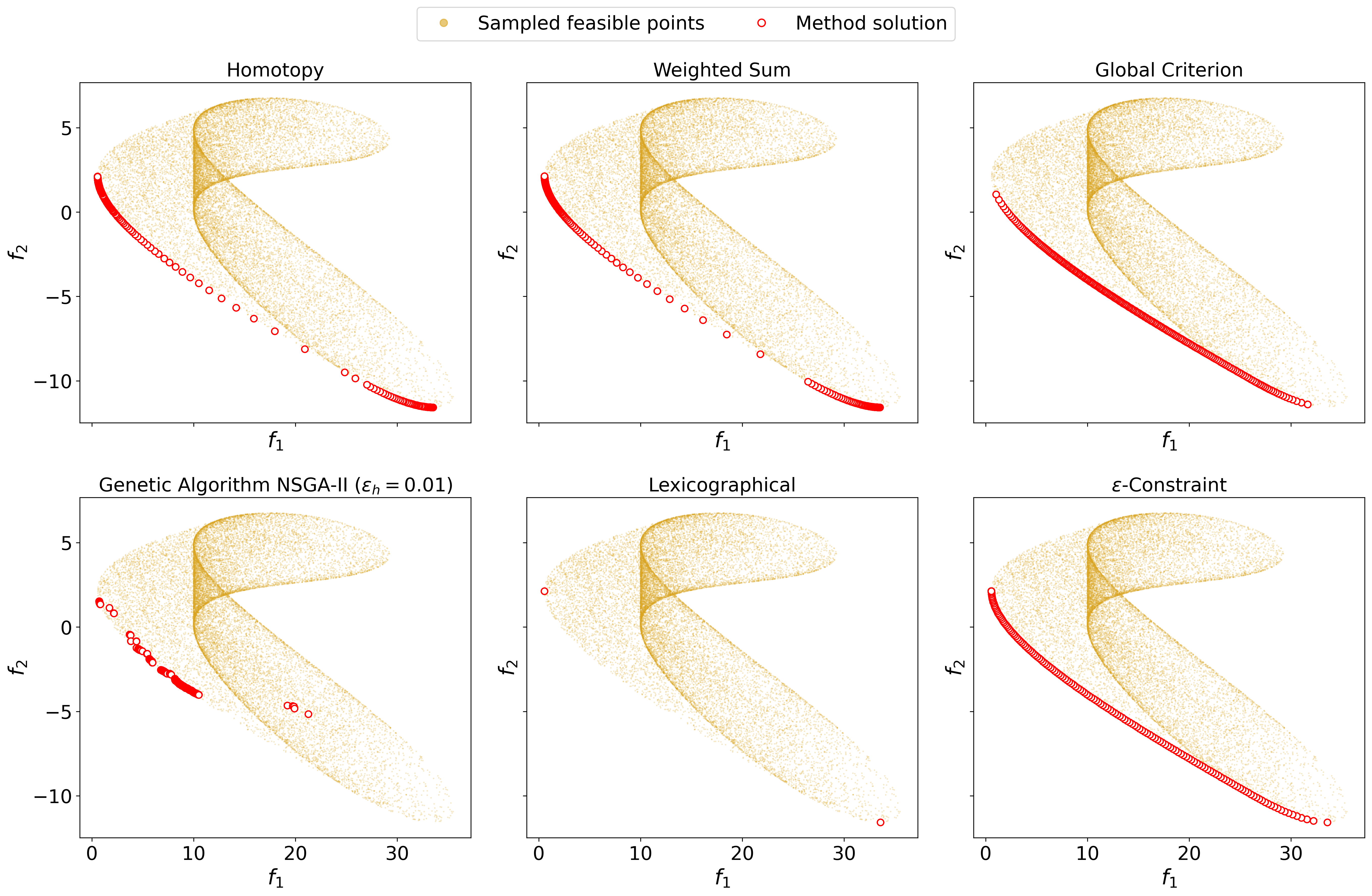}
		\caption{Example~\ref{EX2}: Feasible region obtained by projecting $20{,}000$ randomly generated candidate points, sampled uniformly from the hypercube $[-5,5]^5$, onto the constraint set. \del{Each blue point represents a feasible solution in the $(f_{1}, f_{2})$ objective space.} Pareto fronts generated by the weighted sum, proposed homotopy, global criterion, genetic algorithm NSGA-II, Lexicographic, and $\epsilon-$constraint method are denoted in red color. For NSGA-II, the equality constraints are imposed using the feasibility tolerance $\tau_{\mathrm{eq}}=10^{-2}$; the sensitivity of the computed front to this tolerance is examined separately in Fig.~\ref{fig:nsga-tolerance}.} 
		\label{fig:pareto_plots_ex1}
	\end{figure}

	\begin{table}[h!]
		\centering
		\caption{Normalized Pareto-front quality measures for Example~1.
			The empirical reference front is formed from the nondominated union
			of the highly feasible solutions generated by the four deterministic
			front-generating methods.}
		\label{tab:front-quality}
		\begin{tabular}{lrrrrrr}
			\hline
			Method & ND pts. & IGD & Hausdorff & Spacing CV & Max gap & Coverage \\
			\hline
			Homotopy
			& 198 & 0.00653 & 0.07684 & 2.390 & 0.1553 & 0.999 \\
			
			Weighted Sum
			& 198 & 0.00754 & 0.09059 & 2.629 & 0.1857 & 0.999 \\
			
			Global Criterion
			& 199 & 0.01627 & 0.08024 & 0.378 & 0.0243 & 0.908--0.928 \\
			
			$\epsilon$-Constraint
			& 150 & 0.00227 & 0.02074 & 0.342 & 0.0418 & 1.000 \\
			\hline
		\end{tabular}
	\end{table}

	\add{The quantitative comparison in Table\,\ref{tab:front-quality} confirms that the differences visible in
		Fig.~\ref{fig:pareto_plots_ex1} concern primarily the distribution of the
		computed points along the Pareto front. The proposed homotopy method
		and the weighted-sum method provide nearly complete objective-range
		coverage and exhibit similar distribution characteristics. For this
		example, the homotopy method gives slightly smaller IGD, Hausdorff
		distance, spacing variation, and maximum gap than the weighted-sum
		method.}
	
	\add{The global-criterion and $\epsilon$-constraint methods generate more
		uniformly distributed objective vectors. In particular, their spacing
		coefficients of variation are substantially smaller than those of the
		homotopy and weighted-sum methods. The global-criterion method,
		however, does not recover the extreme portions of the objective range
		as completely, whereas the $\epsilon$-constraint method provides full
		range coverage in this experiment. Of the 199 feasible solutions
		generated by the $\epsilon$-constraint method, 150 remain after
		nondominance filtering.}
	
	\add{These observations concern the distribution of a finite set of
		computed points and should not be interpreted as a general ranking of
		the optimization methods. In particular, uniform spacing of the
		continuation or scalarization parameter does not in general imply
		uniform spacing in objective space.}
	
	\medskip\noindent \add{\textbf{Sensitivity of NSGA-II to the equality-feasibility tolerance}}
	
	\add{For NSGA-II, the equality constraints cannot in general be expected to
		be satisfied exactly by the population generated during the
		evolutionary search. We therefore implement the equality constraints
		through the feasibility conditions
		\(|h_j(\bm{x})| \leq \tau_{\mathrm{eq}}, j=1,2, \)
		where $\tau_{\mathrm{eq}}>0$ is an equality-feasibility tolerance.
		This parameter should not be confused with the parameter used in the
		$\epsilon$-constraint method.}
	
	\add{To examine the sensitivity of NSGA-II to this choice, we repeated the
		calculation using
		$\tau_{\mathrm{eq}} = 10^{-1}$, $10^{-2}$, $10^{-3}$, $10^{-4}$,
		while keeping the remaining NSGA-II parameters fixed. In each run, a
		population size of 800 and 800 generations were used, corresponding to
		640,000 candidate evaluations.}
	
	\begin{figure}[htbp]
		\centering
		\includegraphics[width=0.8\textwidth]{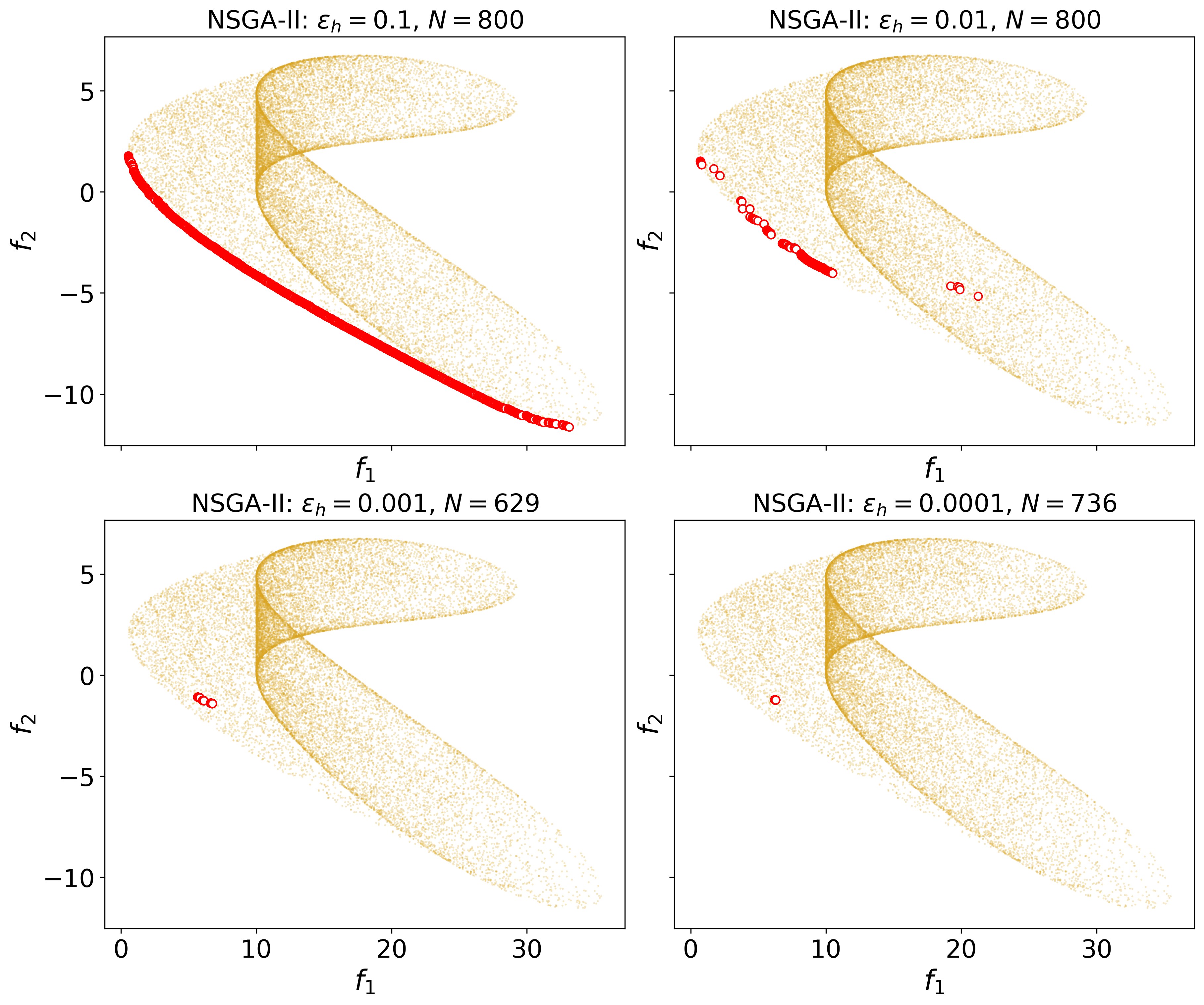}
		\caption{\add{Example~\ref{EX2}: Sensitivity of the NSGA-II solutions to the
				equality-feasibility tolerance $\tau_{\mathrm{eq}}$ for Example~1.
				The four calculations use
				$\tau_{\mathrm{eq}}=10^{-1},10^{-2},10^{-3},10^{-4}$,
				respectively, while all other NSGA-II settings are held fixed.
				Decreasing the tolerance improves equality-constraint satisfaction,
				but under the fixed computational budget the recovered solutions
				become increasingly concentrated in a smaller portion of the
				Pareto front.}}
		\label{fig:nsga-tolerance}
	\end{figure}
	
	\add{The results in Fig.~\ref{fig:nsga-tolerance} show that the apparent
		Pareto-front coverage obtained by NSGA-II depends strongly on the
		equality-feasibility tolerance. For
		$\tau_{\mathrm{eq}}=10^{-1}$, NSGA-II produces a broad distribution of
		solutions across almost the entire objective range. However, the
		maximum equality residual is approximately $10^{-1}$, so these points
		should be interpreted as solutions satisfying a relatively loose
		feasibility band rather than as high-accuracy solutions of the original
		equality-constrained problem.}
	
	\add{When the tolerance is reduced to
		$\tau_{\mathrm{eq}}=10^{-2}$, the maximum equality residual decreases
		accordingly, but the recovered objective-space coverage is substantially
		reduced. Further tightening to $10^{-3}$ and $10^{-4}$ causes the
		computed solutions to become increasingly concentrated in a small
		portion of the front under the same population size and number of
		generations. Thus, improved equality feasibility is accompanied by a
		loss of front coverage for the fixed computational budget used in this
		experiment.}
	
	\add{This behavior is also reflected in the normalized front-quality
		measures. Using the highly feasible nondominated solutions from the
		deterministic methods as an empirical reference front, the results are
		summarized in Table~\ref{tab:nsga-tolerance}.}
	
	\begin{table}[htbp]
		\centering
		\caption{Sensitivity of NSGA-II to the equality-feasibility tolerance
			for Example~1. Front-quality measures are computed in normalized
			objective space relative to the empirical reference front constructed
			from the highly feasible deterministic solutions.}
		\label{tab:nsga-tolerance}
		\begin{tabular}{c r r r r r}
			\hline
			$\tau_{\mathrm{eq}}$
			& Points
			& Max.\ $|h_j|$
			& IGD
			& Hausdorff
			& $f_1/f_2$ coverage \\
			\hline
			$10^{-1}$
			& 800
			& $9.999\times10^{-2}$
			& 0.00536
			& 0.0253
			& 0.988 / 0.978 \\
			
			$10^{-2}$
			& 800
			& $1.000\times10^{-2}$
			& 0.15242
			& 0.5983
			& 0.622 / 0.489 \\
			
			$10^{-3}$
			& 629
			& $1.000\times10^{-3}$
			& 0.42010
			& 1.1005
			& 0.0329 / 0.0236 \\
			
			$10^{-4}$
			& 736
			& $9.999\times10^{-5}$
			& 0.43785
			& 1.1195
			& 0.0018 / 0.0011 \\
			\hline
		\end{tabular}
	\end{table}
	
	\add{The first row of Table~\ref{tab:nsga-tolerance} illustrates why front
		coverage alone can be misleading for equality-constrained problems.
		With $\tau_{\mathrm{eq}}=10^{-1}$, NSGA-II achieves an IGD of
		approximately $5.36\times10^{-3}$ and covers approximately
		$98.8\%$ and $97.8\%$ of the reference ranges of $f_1$ and $f_2$,
		respectively. These apparently favorable front-quality measures are
		obtained, however, while permitting equality residuals of order
		$10^{-1}$.}
	
	\add{In contrast, the deterministic methods considered above satisfy the
		equality constraints to substantially higher numerical accuracy. For
		example, the maximum equality residuals in the present calculations are
		approximately $1.25\times10^{-12}$ for the proposed homotopy method,
		$9.57\times10^{-11}$ for the weighted-sum method,
		$9.39\times10^{-11}$ for the global-criterion method, and
		$2.18\times10^{-11}$ for the $\epsilon$-constraint method.}
	
	\add{Consequently, the NSGA-II results should not be compared with the
		deterministic methods solely on the basis of the number or visual
		distribution of the computed objective vectors. For the present
		equality-constrained example, there is a clear trade-off between
		equality-feasibility accuracy and Pareto-front coverage for NSGA-II
		under a fixed computational budget. The experiment is intended to
		illustrate this sensitivity rather than to provide a general ranking
		of NSGA-II relative to the deterministic methods.
	}

	\begin{table}[h]
		\centering
		\caption{\del{Comparison of CPU time and function evaluations across methods for Example~\ref{EX2}}.\add{Computational results for the methods used to generate the Pareto-front approximations shown in Fig.~\ref{fig:pareto_plots_ex1}.}}
		\begin{tabular}{l c c c c }
			\toprule[1.5pt]
			\\[-1.5ex]
			\textbf{Method} & \textbf{CPU Time (s)} & \textbf{$f_1$} & \textbf{$f_2$} & \textbf{$(g, h_1, h_2)$} \\
			\\[-1.5ex]
			\midrule
			\\[-1.5ex]
			Proposed Homotopy  & 2.10 & $6{,}559$ & $6{,}559$ & $6{,}559$ \\
			\\[-1.5ex]
			Weighted Sum Method  & 0.73 & $4{,}820$ & $4{,}820$ & $5{,}699$ \\
			\\[-1.5ex]
			$\epsilon$-Constraint Method & 1.26 & $5{,}911$ & $7{,}030$ & $7{,}052$ \\
			\\[-1.5ex]
			
			Global Criterion Method & 2.51 & $6{,}386$ & $6{,}408$ & $7{,}579$ \\
			\\[-1.5ex]
			
			Genetic Algorithm NSGA2 & 100.38 & $640{,}000$ & $640{,}000$ & $640{,}000$ \\
			\\[-1.5ex]
			
			\toprule[1.5pt]
		\end{tabular}
		\label{tab:cpu_feval_comparisonEX2}
	\end{table}
	
	\begin{table}[htbp]
		\centering
		\caption{Comparison of the six methods for Example~\ref{EX2}.
			Front-quality measures are computed in normalized objective space
			relative to the empirical reference front constructed from the highly
			feasible deterministic solutions. For NSGA-II,
			$\tau_{\rm eq}=10^{-2}$.}
		\label{tab:example1_all_methods}
		\small
		\begin{tabular}{lrrrrrrr}
			\toprule
			Method
			& $N_{\rm ND}$
			& $\max |h_j|$
			& $\max(g,0)$
			& IGD
			& Hausdorff
			& Spacing CV
			& $f_1/f_2$ coverage \\
			\midrule
			Proposed Homotopy
			& 198 & $1.25{\times}10^{-12}$ & $3.55{\times}10^{-15}$
			& 0.00653 & 0.07684 & 2.390 & 0.999/0.999 \\
			
			Weighted Sum
			& 198 & $9.57{\times}10^{-11}$ & $9.53{\times}10^{-11}$
			& 0.00754 & 0.09059 & 2.629 & 0.999/1.000 \\
			
			Global Criterion
			& 199 & $9.39{\times}10^{-11}$ & $9.30{\times}10^{-11}$
			& 0.01627 & 0.08024 & 0.378 & 0.928/0.908 \\
			
			$\epsilon$-Constraint
			& 150 & $2.18{\times}10^{-11}$ & $6.49{\times}10^{-11}$
			& 0.00227 & 0.02074 & 0.342 & 1.000/1.000 \\
			
			Lexicographical
			& 2 & $1.39{\times}10^{-16}$ & $0$
			& 0.24502 & 0.71372 & -- & 1.000/1.000 \\
			
			NSGA-II ($\tau_{\rm eq}=10^{-2}$)
			& 800 & $1.00{\times}10^{-2}$ & $0$
			& 0.15242 & 0.59832 & 9.419 & 0.622/0.489 \\
			\bottomrule
		\end{tabular}
	\end{table}
	
	\medskip\noindent \textbf{Computational Performance and Efficiency}
	
	Table\,\ref{tab:cpu_feval_comparisonEX2} compares the computational cost of the considered methods for Example\,\ref{EX2} in terms of CPU time and the total number of function evaluations. The proposed homotopy method computes Pareto-critical points with moderate computational cost, requiring only a few thousand evaluations of the objective and constraint functions. \del{When parallelization is employed, the CPU time is further reduced while maintaining the same number of function evaluations.}
	In contrast, the weighted-sum method is computationally inexpensive in terms of CPU time but requires a comparable number of objective evaluations, reflecting its reliance on repeated scalar optimization solves. The $\epsilon$-constraint and global criterion methods incur higher computational costs due to the additional constraint handling and penalty evaluations required during the optimization process. NSGA-II is significantly more expensive than the deterministic methods, both in CPU time and function evaluations, as it relies on population-based sampling and repeated fitness evaluations to approximate the Pareto front.
	Overall, the results demonstrate that the proposed homotopy approach provides a favorable balance between computational efficiency and solution quality, achieving KKT-consistent Pareto-critical points at a substantially lower cost than evolutionary MOO methods.
	
	\add{Table~\ref{tab:example1_all_methods} provides a quantitative comparison
		of the six methods using a common empirical reference front and common
		feasibility criteria. The proposed homotopy and weighted-sum methods
		produce very similar approximations of the Pareto front, with essentially
		complete objective-range coverage. The proposed homotopy method yields
		slightly smaller IGD and Hausdorff distance than the weighted-sum method,
		with values of $0.00653$ and $0.07684$, compared with $0.00754$ and
		$0.09059$, respectively. It also gives smaller equality- and
		inequality-constraint residuals. These differences are modest, however,
		and the results should be interpreted as showing comparable performance
		of the two methods for this example rather than numerical superiority of
		one method over the other.}
	
	\add{The $\epsilon$-constraint method gives the best front-quality measures in
		this experiment, attaining the smallest IGD ($0.00227$), Hausdorff
		distance ($0.02074$), and spacing coefficient of variation ($0.342$),
		together with complete objective-range coverage. The global-criterion
		method also produces a relatively uniform distribution of solutions, as
		indicated by its spacing coefficient of variation, but its objective-range
		coverage is lower than that of the homotopy, weighted-sum, and
		$\epsilon$-constraint methods.}
	
	\add{The lexicographical method returns only the two priority-order solutions.
		Consequently, its near-complete objective-range coverage reflects the fact
		that these solutions are located near the two objective extremes and
		should not be interpreted as coverage of the interior of the Pareto
		front. This is also reflected in its substantially larger IGD and
		Hausdorff distance. For this reason, a spacing coefficient of variation is
		not reported for the lexicographical method.}
	
	\add{For NSGA-II with $\tau_{\rm eq}=10^{-2}$, the equality residual reaches
		$10^{-2}$ and the resulting front has substantially larger IGD and
		Hausdorff distance and lower objective-range coverage than the highly
		feasible deterministic fronts. As discussed in the preceding sensitivity
		study, these results depend strongly on the equality-feasibility
		tolerance: tightening the tolerance improves equality feasibility but,
		for the fixed population size and generation budget used here,
		substantially reduces front coverage.}
	
	\add{Overall, the results do not indicate that the proposed homotopy method
		dominates the other methods on every numerical metric. Rather, for
		Example~\ref{EX2}, it produces a highly feasible, nearly complete
		approximation of the Pareto-critical front with front-quality measures
		comparable to those of established scalarization methods. Its principal
		distinction is that these solutions are generated through a continuation
		framework for tracing KKT-consistent Pareto-critical solutions, rather
		than through a sequence of independent scalar optimization problems.}

	\medskip\noindent \textbf{Constraint Feasibility Validation}
	
	Table \ref{tab:comparison_literature_example1} 
	assess the reliability of reported solutions in the literature - \citet{shang2011} and  
	\citet{zhao2012quasinormal}, together with the solution obtained by our proposed homotopy method:

	\begin{itemize}
		\item 
		The solution reported in~\citet{shang2011},
		$(0.3077,\,0.5374,\,-0.2703,\,-0.1336,\,0.2804)$,
		yields objective values
		$(f_{1}(\bm{x}), f_{2}(\bm{x})) \approx (0.5530, 2.0873)$.
		The inequality constraint is strictly satisfied
		($g(\bm{x}) \approx -9.5256$),
		and the second equality constraint satisfies
		$h_{2}(\bm{x}) \approx 0$,
		while the first equality constraint exhibits a moderate violation
		($h_{1}(\bm{x}) \approx -0.1011$).
		Thus, the point is close to feasibility but does not satisfy all constraints simultaneously.
		
		\item 
		The solution reported in~\citet{zhao2012quasinormal},
		$(-1.3074,\,-2.8605,\,-1.0470,\,0.4103,\,0.4475)$,
		produces
		$(f_{1}(\bm{x}), f_{2}(\bm{x})) \approx (11.3566, -9.2942)$.
		Although $h_{1}(\bm{x}) \approx 1.08\times 10^{-4}$,
		the second equality constraint has a large violation
		($h_{2}(\bm{x}) \approx -7.7391$),
		and the inequality constraint is violated
		($g(\bm{x}) \approx 1.1563 > 0$).
		Hence, this point is infeasible.
		
		\item 
		The proposed homotopy method computes the solution
		$(0.3176,\,0.5067,\,-0.2797,\,-0.1403,\,0.3192)$
		with objective values
		$(f_{1}(\bm{x}), f_{2}(\bm{x})) \approx (0.5574, 2.0584)$.
		Both equality constraints are satisfied to near machine precision
		($h_{1}(\bm{x}) \approx -1.39\times10^{-17}$ and $h_{2}(\bm{x}) \approx -2.22\times10^{-16}$),
		while the inequality constraint remains strictly feasible
		($g(\bm{x}) \approx -9.5445$).
		\del{Among the compared methods, the proposed homotopy
			framework produces the only fully feasible
			Pareto-critical solution.}
		\add{Among the three candidate solutions reported
			in this comparison, only the solution obtained
			by the proposed homotopy method satisfies all
			constraints to the reported numerical accuracy.}
	\end{itemize}
	\begin{table}[htbp]
		\centering
		\caption{Comparison of candidate solutions for Example~\ref{EX2} and objective values and constraint residuals.
			The proposed homotopy method uses
			$\bm{x}^{0}=(1,2,0,1,1)$ and $\bm{w}=(0.96,0.04)$.
			Here,
			$\|H_{\bm{\omega}^*}\|
			:=\|H(\bm{\omega}^*,\bm{\omega}^{0},t_*)\|$.}
		\label{tab:comparison_literature_example1}
		
		\begin{tabular}{lrrrrrr}
			\toprule
			Method
			& $x_1^*$ & $x_2^*$ & $x_3^*$ & $x_4^*$ & $x_5^*$ & \\
			\midrule
			\citet{shang2011}
			& 0.3077 & 0.5374 & -0.2703 & -0.1336 & 0.2804 & \\
			\citet{zhao2012quasinormal}
			& -1.3074 & -2.8605 & -1.0470 & 0.4103 & 0.4475 & \\
			Proposed Homotopy
			& 0.3176 & 0.5067 & -0.2797 & -0.1403 & 0.3192 & \\
			
			\midrule
			Method
			& $f_1(\bm{x}^*)$
			& $f_2(\bm{x}^*)$
			& $h_1(\bm{x}^*)$
			& $h_2(\bm{x}^*)$
			& $g(\bm{x}^*)$
			& $\|H_{\bm{\omega}^*}\|$ \\
			\midrule
			\citet{shang2011}
			& 0.5530
			& 2.0873
			& -0.1011
			& $\add{0.0000}$
			& -9.5256
			& $6.6\times10^{-16}$ \\
			\citet{zhao2012quasinormal}
			& 11.3566
			& -9.2942
			& $1.08\times10^{-4}$
			& -7.7391
			& 1.1563
			& $1.0\times10^{-12}$ \\
			Proposed Homotopy
			& 0.557392
			& 2.0584
			& $-1.39\times10^{-17}$
			& $-2.22\times10^{-16}$
			& -9.5445
			& $7.1\times10^{-12}$ \\
			\bottomrule
		\end{tabular}
	\end{table}

	\begin{table}[h]
		\centering
		\caption{\del{Numerical results for Example \ref{EX2} with proposed homotopy algorithm \ref{alg:homotopy} and $\textbf{w} = (0.4, 0.6)$, add feasible Boolean, where $\|H_{\textbf{w}^*}\| = \|H (\textbf{w}^*, \textbf{w}^0, t_*)\|.$} \add{Sensitivity of the proposed homotopy method to the initial
				point for Example~\ref{EX2} with $\bm{w}=(0.4,0.6)$.
				Both feasible and infeasible initial points are considered.
				Here,
				$\|H_{\textbf{w}^*}\|
				:=\|H(\textbf{w}^*,\textbf{w}^0,t_*)\|$.}}
		\begin{tabular}{c c c c c}
			\toprule[1.5pt] \\
			$\bm{x}_0$ & $\bm{x}^0$ \add{Feasible?} & $\bm{x}^*$ & $f(\bm{x}^*)$ & $\|H_{\textbf{w}^*}\|$ \\
			\\
			\midrule
			$(1.0, 2.0, 0.0, 1.0, 1.0)$ & No & $ (-0.1390, -0.0518, -0.5309, -0.4189, 1.5023)$ & $(2.7363, -0.4147)$ & $2.8\times 10^{-5}$ \\ 
			\\
			$(0.5, 0.5, 0.5, 0.5, 0.5)$ & No & $(-0.1375, -0.0488,  -0.5281,  -0.4152, 1.4993)$ & $( 2.7207, -0.4042)$ & $5.4\times 10^{-5}$ \\ 
			\\
			$(-2.0, 0.0, 0.0, 0.0, 4.0)$ & Yes & $ (-0.1370, -0.0507, -0.5311, -0.4180,  1.4983)$ & $(2.7231, -0.4058)$ & $3.7\times 10^{-5}$ \\ 
			\\
			$(0.4, 0.8, 0.0, 0.0, 0.0)$ & Yes & $ (-0.1389, -0.0513, -0.5306, -0.4179, 1.5020)$ & $(2.7341, -0.4132)$ & $5.0\times 10^{-5}$ \\ 
			\\
			\toprule[1.5pt]
		\end{tabular}
		\label{tab:homotopyEX2}
	\end{table}

	\medskip\noindent \textbf{Sensitivity to Initialization}
	
	Table\,\ref{tab:homotopyEX2} reports the numerical results for
	Example\,\ref{EX2} obtained using the proposed homotopy algorithm with
	$\textbf{w}=(0.4,0.6)$, starting from both feasible and infeasible
	initial guesses. The feasibility of each initial point $\bm{x}^0$ is
	explicitly indicated. Regardless of whether $\bm{x}^0$ is feasible, the
	homotopy method consistently converges to a Pareto-critical point with
	objective values clustered around $f(\bm{x}^*)\approx(2.7,-0.41)$.
	\add{The computed solutions are close but not identical, indicating some
		dependence on the initialization.}
	Moreover, the small values of the KKT residual
	$\lVert H_{\bm{x}^*}\rVert$ confirm that the computed solutions satisfy
	the optimality conditions to high accuracy. These results
	\del{demonstrate the robustness of the proposed homotopy approach with
		respect to the choice of initial guess and}
	\add{indicate limited sensitivity to the initial guesses considered here,
		while demonstrating}
	its ability to drive infeasible starting points toward KKT-consistent
	Pareto-critical points.
	
	
	This example highlights one of the key advantages of continuation-based methods: the ability to systematically trace disconnected and nonconvex Pareto structures through branch-following mechanisms embedded in the homotopy formulation.
	
	\begin{example}[\textbf{Degenerate and Disconnected Four-Objective Problem}]
		
		Consider the decision vector \(\bm{x}=(x_1,x_2,x_3,x_4)^T\). We study the constrained four-objective optimization problem
		\begin{equation}
			\label{eq:degenerate_disconnected_mop_eq}
			\begin{aligned}
				\min_{\bm{x}}\quad
				&F(\bm{x})=
				\bigl(f_1(\bm{x}),f_2(\bm{x}),f_3(\bm{x}),f_4(\bm{x})\bigr)\\
				\text{s.t.}\quad
				&x_4=0,\\
				&x_3^2-x_2^2=0,\\
				&(x_1^2-1)^2+x_2^2-0.08\le 0,\\
				&-2\le x_1\le 2,\qquad -1\le x_2,x_3,x_4\le 1,
			\end{aligned}
		\end{equation}
		where
		\begin{equation}
			\label{eq:degenerate_disconnected_objectives_eq}
			\begin{aligned}
				f_1(\bm{x})&=(x_1+1)^2+x_2^2+0.05\sin(8x_1),\\
				f_2(\bm{x})&=(x_1-1)^2+x_2^2+0.05\sin(8x_1),\\
				f_3(\bm{x})&=x_2^2+0.1(1-\cos(6x_1)),\\
				f_4(\bm{x})&=(x_1^2-1)^2+x_2^2.
			\end{aligned}
		\end{equation}
	\end{example}
	
	This example combines several challenging features for continuation-based many-objective optimization. The problem is many-objective because it contains four competing objectives. It is nonconvex due to the oscillatory sinusoidal terms and the quartic nonlinear constraint defining the feasible set. The inequality constraint \((x_1^2-1)^2+x_2^2\le0.08\) creates two disconnected feasible components centered near \(x_1=-1\) and \(x_1=1\). The equality constraint \(x_3^2-x_2^2=0\) \del{introduces a branched nonlinear manifold, while the constraint \(x_4=0\) fixes a redundant variable} \add{defines two intersecting branches, $x_3=x_2$ and $x_3=-x_2$, with a singular intersection at $x_2=x_3=0$, while the constraint $x_4=0$
		fixes one decision variable}. The problem is also degenerate because the objective functions do not explicitly depend on \(x_3\) or \(x_4\). Consequently, multiple feasible points may map to the same objective vector, and the Pareto-critical set in decision space may possess a higher-dimensional or branched structure relative to its image in objective space.
	
	\add{At the intersection $x_2=x_3=0$, the
		gradient of the nonlinear equality constraint
		vanishes, so the linear independence constraint
		qualification (Assumption A2) fails. Consequently, this example
		also tests the numerical continuation method
		outside the regularity assumptions used in the
		global convergence theorem.}
	
	\medskip \noindent \textbf{Continuation Setup and Branch Initialization}
	
	For this example, the decision dimension is \(n=4\), the number of objectives is \(p=4\), the number of equality constraints is \(s=2\), and the number of inequality constraints is \(m=9\), including the box constraints. The homotopy state vector is initialized as $\bm{\omega}^0=
	[\bm{x}^0,\bm{w}^0,\bm{u}^0,\bm{v}^0,t_0]^T,
	\qquad t_0=1,$ equality multipliers \(\bm{v}^0=\mathbf{0}\in\mathbb{R}^2\), and inequality multipliers \(\bm{u}^0=[0.05\,\mathbf{1}]\in\mathbb{R}^9\).
	
	The objective-weight vectors are selected from a positive discretization of the four-objective simplex using discretization level of 10 and lower bound \(0.03\). Thus each weight vector satisfies
	\[
	\textbf{w}^0=(\text{w}_1,\text{w}_2,\text{w}_3,\text{w}_4),\qquad
	\sum_{i=1}^4 \text{w}_i=1,\qquad
	\text{w}_i\ge0.03.
	\]
	
	To explore disconnected branches, the continuation is restarted from multiple feasible initial points near both disconnected feasible components. The restart points are centered near \(x_1=\pm1\), with
	\[
	x_2\in\{0.02,0.04,0.06,0.08,0.10,0.12,0.15,-0.04,-0.08,-0.12,-0.15\}.
	\]
	The equality manifold \(x_3^2-x_2^2=0\) is initialized using both branches \(x_3=x_2\) and \(x_3=-x_2\), with \(x_4=0\). Consequently, the restart points take the form
	$\bm{x}=(\pm1,x_2,\pm x_2,0)^T,$
	subject to the strict interior feasibility condition
	$(x_1^2-1)^2+x_2^2<0.08.$
	The continuation paths are advanced using a pseudo-arclength predictor--corrector strategy. The initial continuation step size is \(\alpha_0=0.06\), with bounds \(\alpha_{\min}=10^{-5}\) and \(\alpha_{\max}=0.10\). The stopping tolerance for the homotopy parameter is \(\epsilon_t=10^{-4}\), and the maximum number of continuation steps per path is \(k_{\max}=900\).
	
	At each continuation step, the tangent direction is computed from the null space of the homotopy Jacobian and oriented consistently using the previous tangent vector. Newton-type corrector iterations are then applied to the augmented pseudo-arclength system with maximum Newton iterations \(15\), Newton tolerance \(10^{-9}\), and a damped line search. \add{
		A trial step is accepted only if the corrected point
		is finite, the augmented pseudo-arclength residual
		satisfies the prescribed acceptance criterion, and
		the homotopy parameter remains within the admissible
		range while allowing small deviations from strict
		progress toward $t=0$}. \del{The step size is increased following rapid Newton convergence and reduced following slow convergence or failed corrections.} \add{
		The step size is increased following rapid Newton
		convergence and reduced following slow convergence
		or failed corrections. Additional step-size
		restrictions are imposed as the homotopy parameter
		approaches zero to improve endpoint resolution.}
	
	\add{In total, 44 feasible initial points and 84
		positive objective-weight vectors are used,
		giving 3,696 independent homotopy continuation
		attempts.}

	\medskip \noindent \textbf{Pareto-Critical Structure and Continuation Results}
	
	The following figures illustrate the geometry of the computed Pareto-critical solution set, the disconnected feasible components, the equality-constraint structure, and the objective-space behavior of the proposed homotopy continuation framework.
	\begin{figure}[H]
		\centering
		\includegraphics[width=0.45\textwidth]{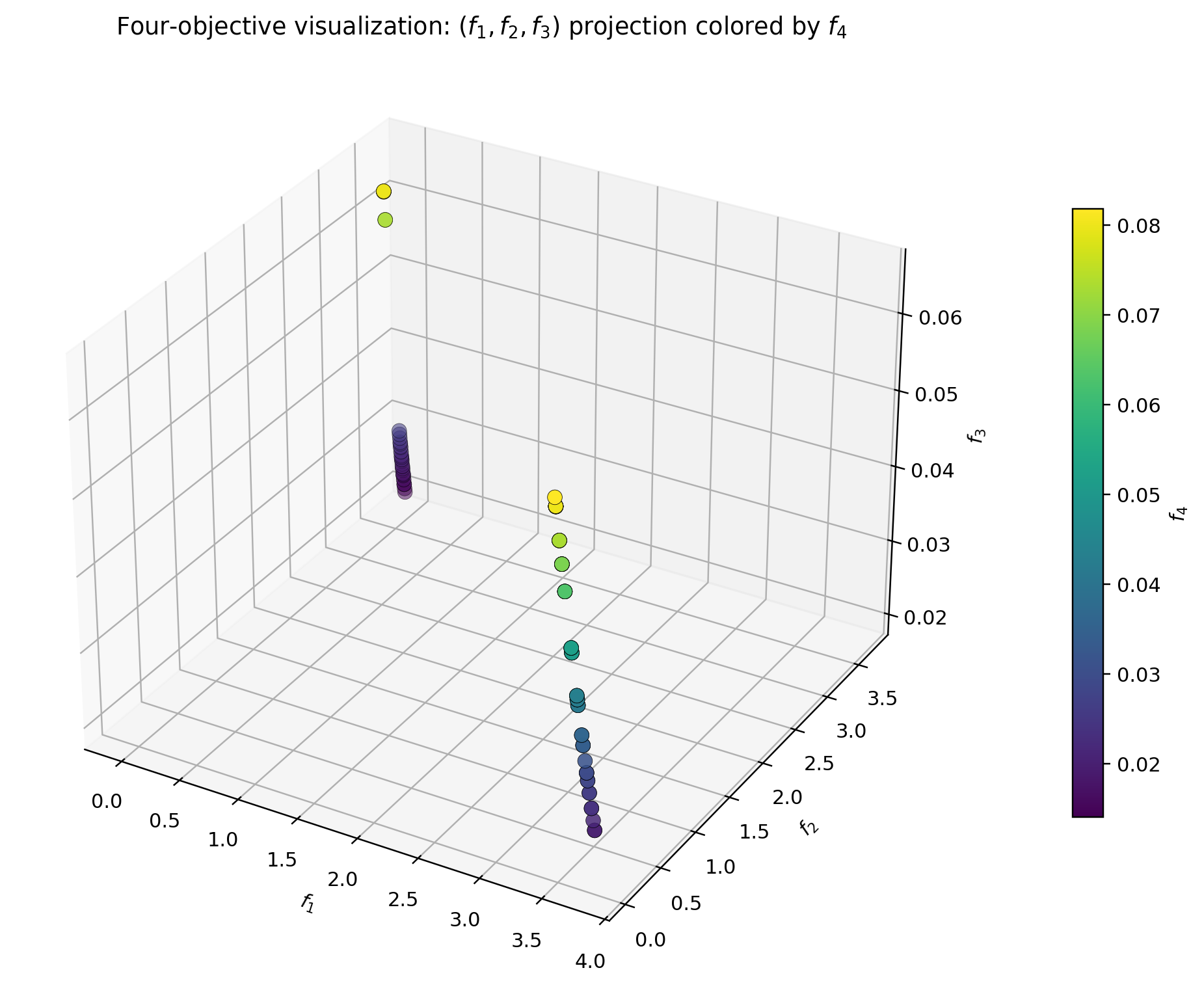}
		\includegraphics[width=0.45\textwidth]{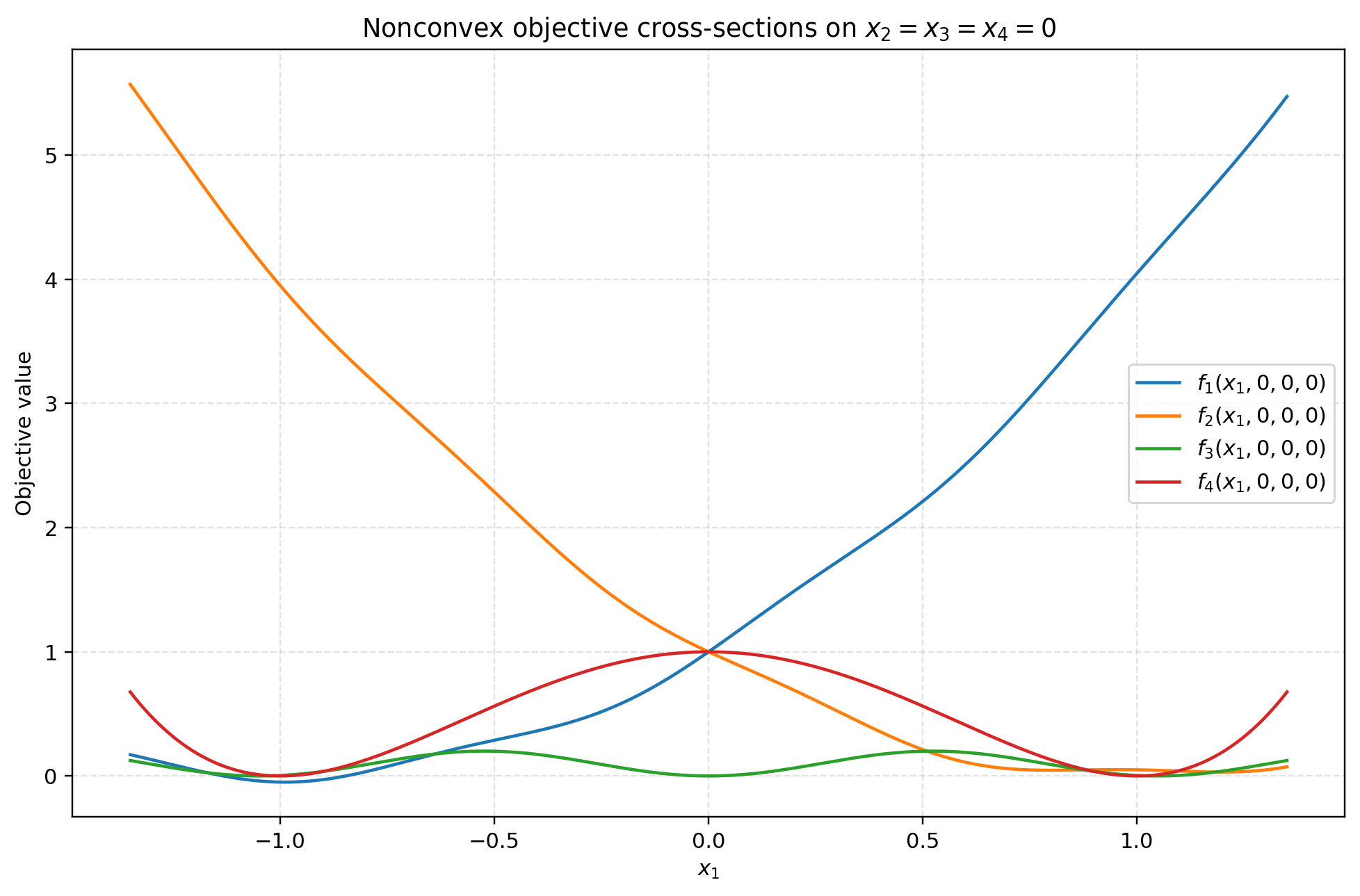}
		\caption{
			Left: Three-dimensional projection of the computed Pareto-critical set in the $(f_1,f_2,f_3)$ objective space, with $f_4$ encoded by color. Right: Cross-sections of the four objective functions along the reduced manifold $x_2=x_3=x_4=0$.
		}
		\label{fig:manyobj_3d_projection_colored}
	\end{figure}
	
	The left panel of Fig.~\ref{fig:manyobj_3d_projection_colored} illustrates the geometric structure of the computed Pareto-critical set in objective space. The disconnected vertical clusters correspond to continuation trajectories associated with the two disconnected feasible components centered near \(x_1\approx -1\) and \(x_1\approx 1\). The narrow branch-like structures indicate degeneracy in the objective mapping, where multiple feasible decision vectors collapse onto low-dimensional Pareto-critical structures in objective space. The color variation associated with \(f_4\) further reveals how the fourth objective changes along the continuation trajectories. 
	
	The right panelof Fig.~\ref{fig:manyobj_3d_projection_colored}  displays one-dimensional cross-sections of the objective functions along the reduced manifold \(x_2=x_3=x_4=0\). The oscillatory sinusoidal perturbations introduce nonconvexity and multiple competing local trade-off regions, producing a highly nonlinear many-objective optimization landscape. These competing objective structures contribute to the disconnected and branched Pareto-critical behavior observed in the continuation results.

	Fig.~\ref{fig:manyobj_pairwise_objective} illustrates the pairwise trade-off structure among the four competing objectives. The disconnected vertical branches visible in several projections reflect the presence of disconnected feasible components and multiple continuation branches generated by the homotopy framework. The color distributions indicate that the omitted objectives vary nonuniformly along the Pareto-critical trajectories, demonstrating the coupled many-objective nature of the problem. 
	
	The projections involving \(f_1\) and \(f_2\) exhibit strong separation between disconnected branches associated with the two feasible regions centered near \(x_1\approx -1\) and \(x_1\approx 1\). In contrast, the \((f_3,f_4)\) projection reveals a nearly one-dimensional structure, indicating degeneracy in the objective mapping and strong correlation between these objectives along the computed Pareto-critical set. Overall, the pairwise projections confirm that the proposed continuation framework successfully captures disconnected, nonlinear, and degenerate Pareto-critical structures in objective space.
	
	Fig.~\ref{fig:manyobj_decision_space_3d} illustrates the geometric structure of the computed Pareto-critical set in decision space and its relationship to the disconnected feasible components. The left panel shows that the continuation trajectories collapse near a low-dimensional branched manifold in the \((x_1,x_2,x_3)\) subspace, reflecting the degeneracy introduced by the equality constraint \(x_3^2-x_2^2=0\). The two visible branches correspond to the manifolds \(x_3=x_2\) and \(x_3=-x_2\), while the concentration of solutions near \(x_2=x_3=0\) indicates partial collapse of the Pareto-critical structure in decision space.
	
	The right panel displays the disconnected feasible components induced by the nonlinear inequality constraint \((x_1^2-1)^2+x_2^2\le0.08\). The continuation framework successfully traces Pareto-critical branches associated with both disconnected regions near \(x_1\approx -1\) and \(x_1\approx 1\). These results demonstrate the capability of the proposed continuation strategy to explore disconnected feasible regions and branched Pareto-critical structures within a unified homotopy framework.
	\begin{figure}[H]
		\centering
		\includegraphics[width=0.95\textwidth]{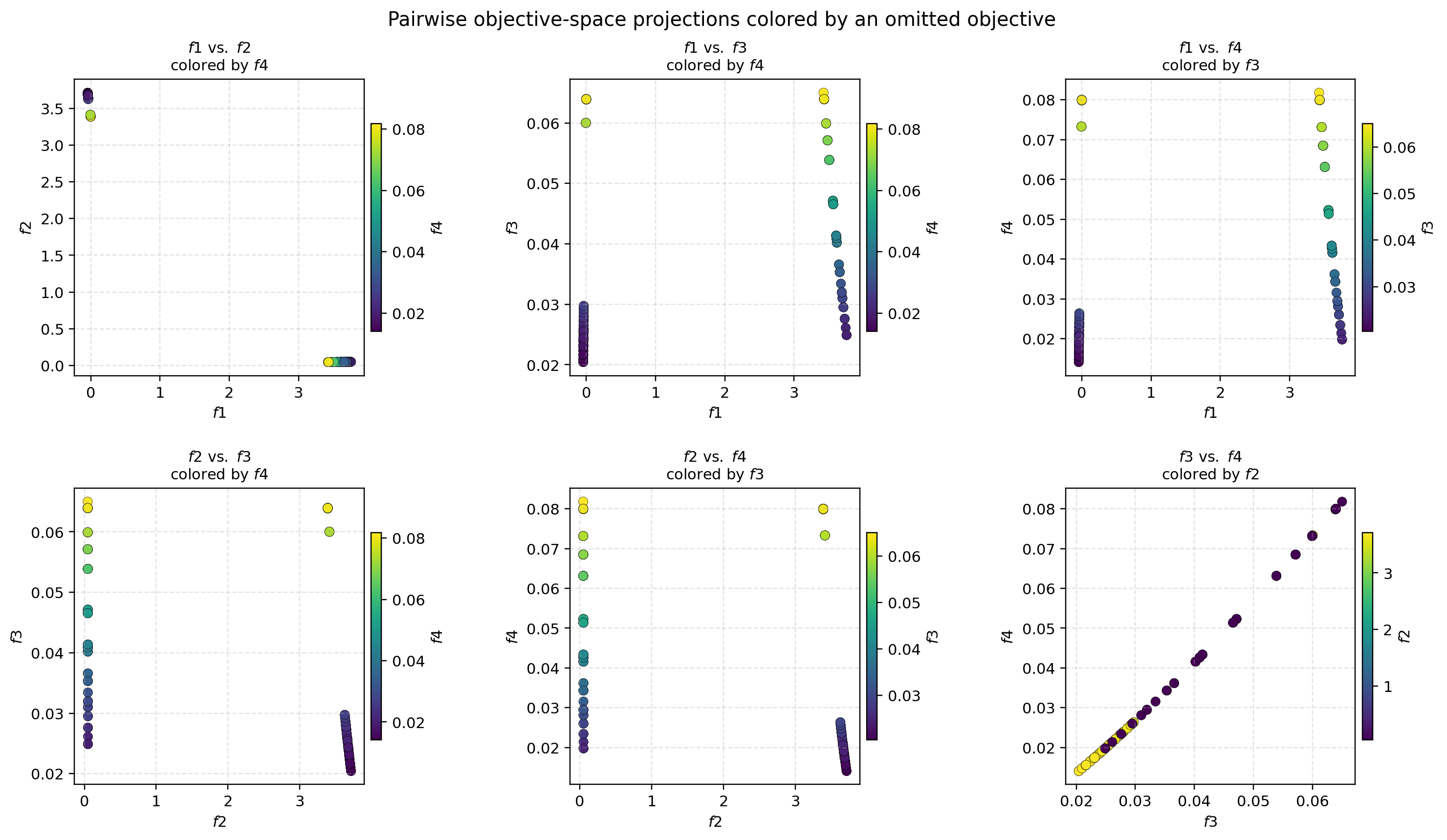}
		\caption{
			Pairwise objective-space projections of the computed Pareto-critical set. Each subplot displays a two-objective projection, while the omitted objective is represented by color.
		}
		\label{fig:manyobj_pairwise_objective}
	\end{figure}
	\begin{figure}[H]
		\centering
		\begin{minipage}[c]{0.35\textwidth}
			\centering
			\includegraphics[
			width=\linewidth
			]{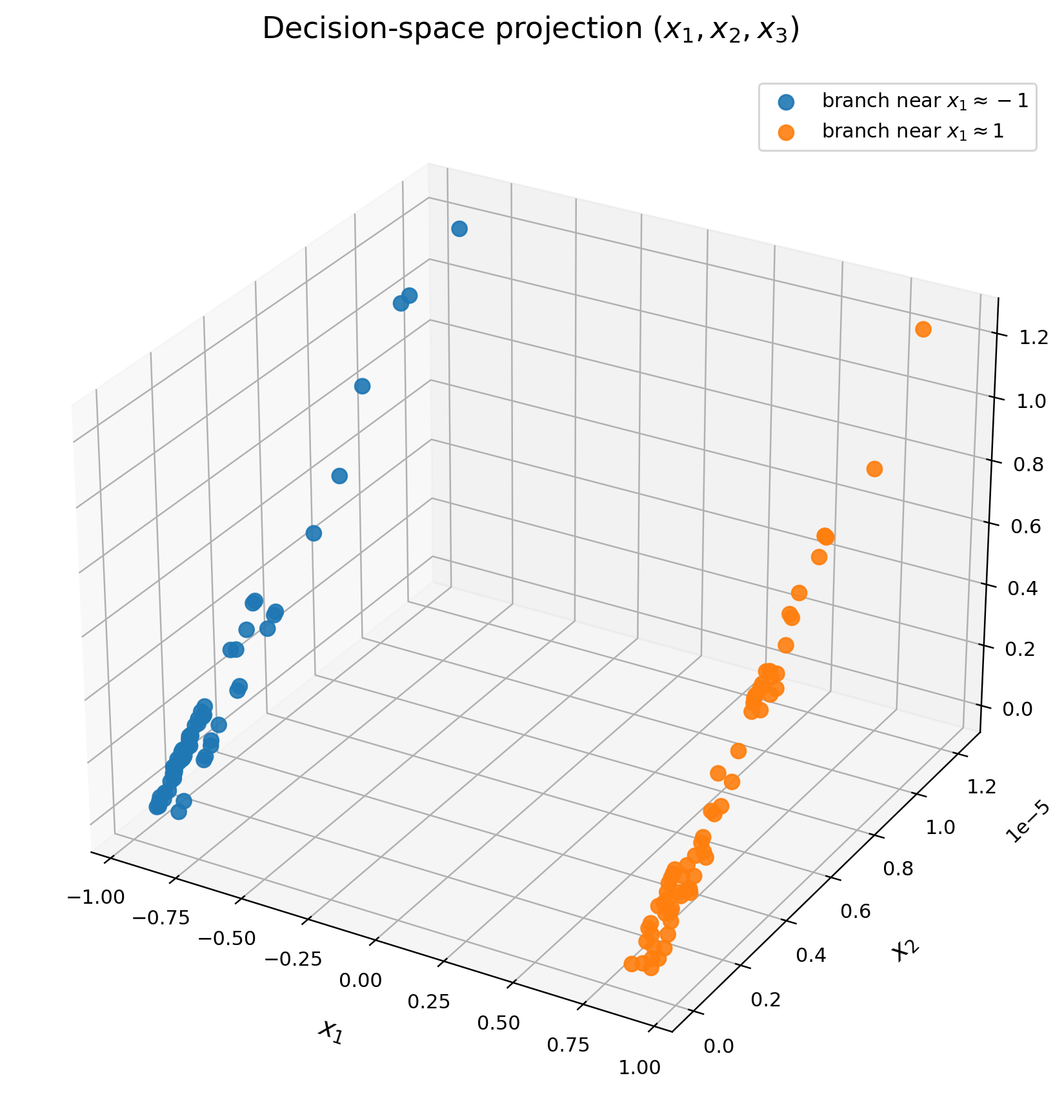}
		\end{minipage}
		\hfill
		\begin{minipage}[c]{0.48\textwidth}
			\centering
			\includegraphics[
			width=\linewidth,
			trim=0 0 0 0,
			clip
			]{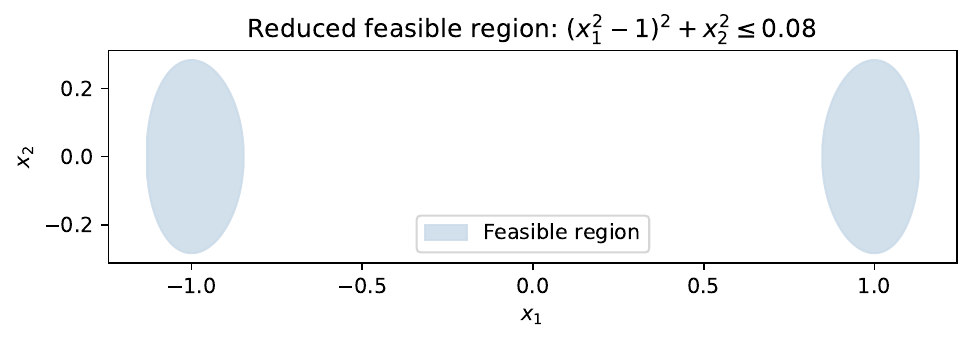}
		\end{minipage}
		\caption{
			Left: Three-dimensional decision-space projection
			$(x_1,x_2,x_3)$ of the computed Pareto-critical set.
			\add{Right: Projection of the feasible region onto the
				$(x_1,x_2)$-plane, showing two disconnected feasible
				components near} $x_1\approx -1$ and $x_1\approx 1$.
		}
		\label{fig:manyobj_decision_space_3d}
	\end{figure}
	\begin{figure}[H]
		\centering
		\includegraphics[width=0.4\textwidth]{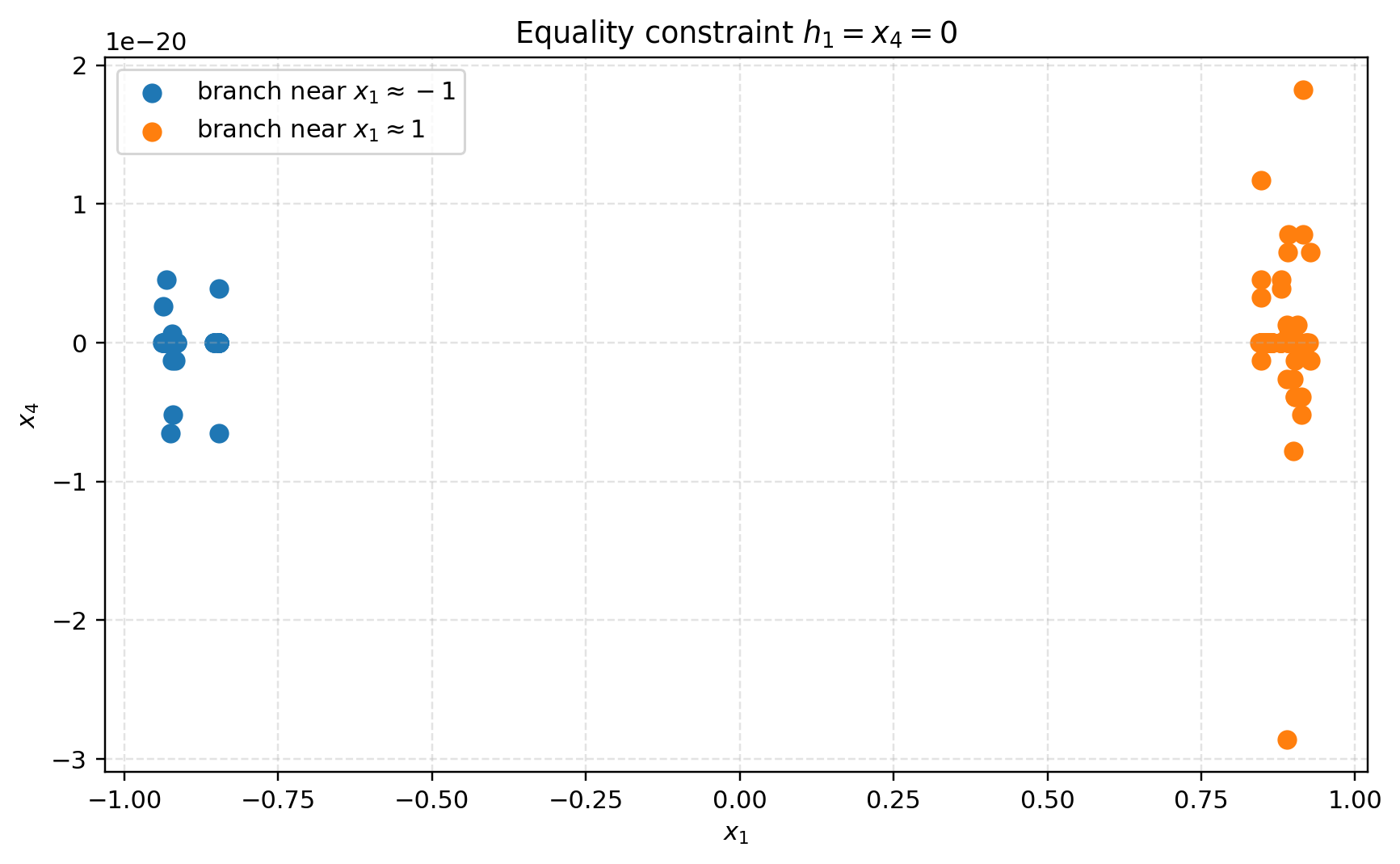}
		\includegraphics[width=0.45\textwidth]{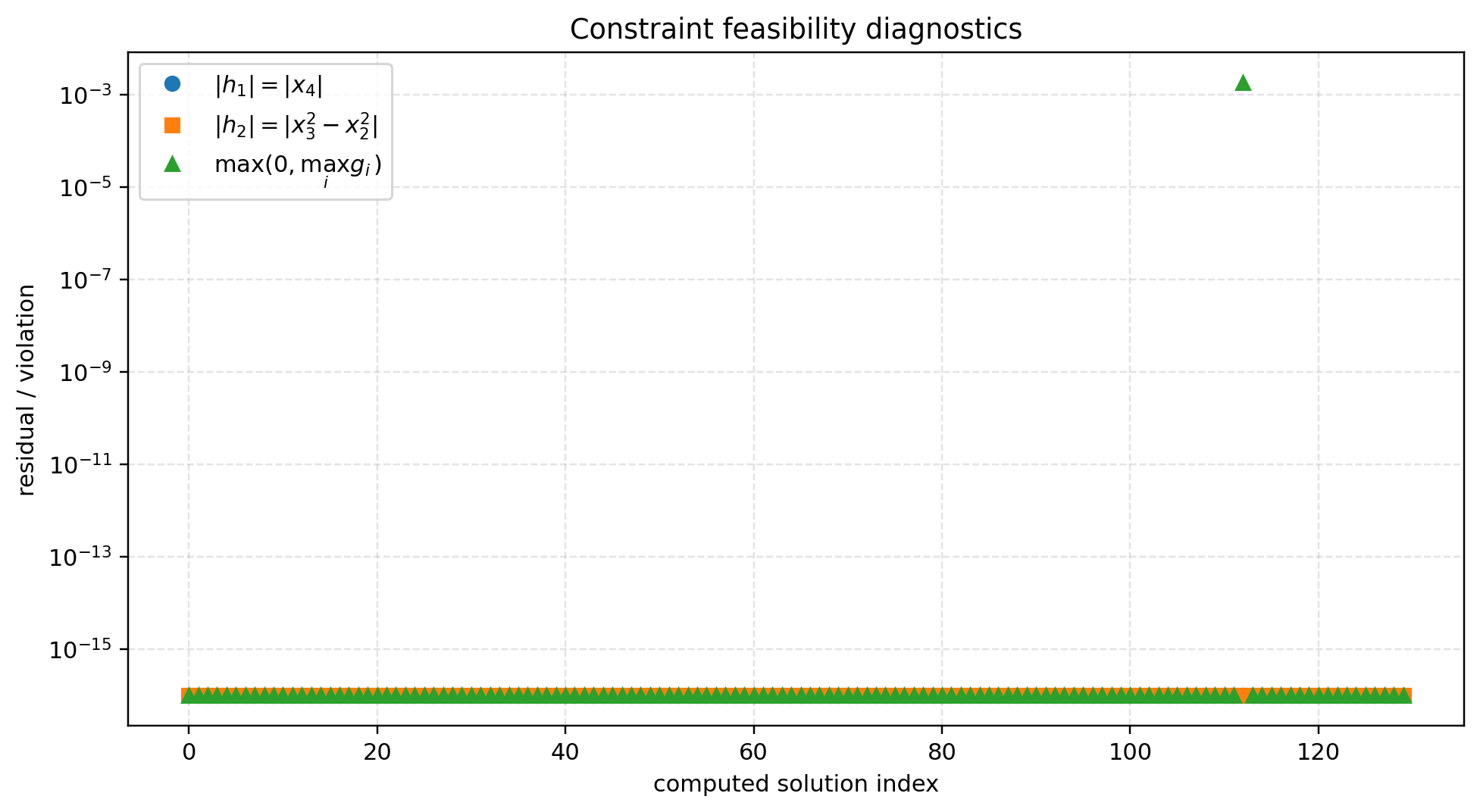}
		\caption{
			Left: Equality-constraint structure for \(h_1=x_4=0\). Right: Constraint feasibility diagnostics for the computed Pareto-critical solutions.
		}
		\label{fig:manyobj_feasibility}
	\end{figure}
	Fig.~\ref{fig:manyobj_feasibility} verifies the feasibility of the computed Pareto-critical solutions. The left panel shows that all computed points satisfy the equality constraint \(x_4=0\) \del{to near machine precision}\add{within the reported numerical feasibility tolerance}, while remaining associated with the two disconnected feasible components near \(x_1\approx -1\) and \(x_1\approx 1\). This confirms that the continuation trajectories remain on the prescribed equality manifold.
	
	The right panel reports the equality residuals and inequality violations along the computed continuation solutions. Both equality residuals remain close to machine precision, and no positive inequality violation is observed. These diagnostics confirm that the computed Pareto-critical points are not merely approximate objective-space samples, but feasible KKT-consistent solutions of the constrained many-objective problem.
	
	
	\medskip \noindent \textbf{Quality Indicators and Feasibility Diagnostics} 
	
	\begin{table}[h]
		\centering
		\caption{\add{Original homotopy results for Example~2,
				aggregated over 3,696 continuation attempts.
				Quality indicators are computed relative to the
				common sampled reference set.}}
		\label{tab:example2_homotopy_baseline}
		\begin{tabular}{lr}
			\toprule
			Indicator & Original homotopy \\
			\midrule
			Feasible nondominated solutions & 129 \\
			Both disconnected components represented & Yes \\
			Normalized HV & 0.336810 \\
			GD & 0.007367 \\
			IGD & 0.246221 \\
			\bottomrule
		\end{tabular}
	\end{table}
	
	\del{The quality-indicator table below reports an averaged Hausdorff distance identical to IGD without a separately verified computation.}
	\del{The equality residuals and inequality violations remain near machine precision throughout the continuation process.}
	\add{
		The original homotopy implementation was evaluated
		using multiple initializations. The resulting
		solutions were pooled and filtered to obtain
		129 feasible nondominated solutions spanning
		both disconnected components.
		As shown in Table~\ref{tab:example2_homotopy_baseline},
		the original homotopy run produced 129 feasible nondominated
		solutions covering both disconnected components. The
		reported minimum aggregated constraint violation was
		approximately $8.08\times10^{-28}$; this value is a
		minimum over the reported solutions, not an upper bound
		on all constraint residuals. The small GD ($0.007367$)
		and comparatively larger IGD ($0.246221$) indicate that
		the homotopy solutions lie close to portions of the
		sampled reference front but provide less extensive
		coverage than the reduced-space evolutionary solutions
		discussed below. Feasibility and KKT stationarity are
		distinct diagnostics; the latter cannot be inferred
		from constraint satisfaction alone.}
	
	
	\medskip \noindent \textbf{Comparison with Constrained Evolutionary Algorithms}
	
	
	Because the degenerate disconnected many-objective problem is specifically designed to test approximation quality and coverage of complicated Pareto-critical structures, additional set-based quality indicators are included for this example.
	\add{We compare the original homotopy baseline with NSGA-II, NSGA-III, and C-TAEA under four complementary protocols. Each evolutionary configuration uses 20 independent runs, population size 84, and 400 generations (33,600 reported evaluations per run); homotopy is the same single original run throughout. The original strict test uses $\tau_h=10^{-7}$ and $\tau_g=10^{-7}$.}

	\begin{table}[htbp]
		\centering
		\caption{\add{Experimental protocols for Example~2.
				All evolutionary algorithms use 20 independent runs,
				population size 84, and 400 generations. The inequality
				tolerance is fixed at $\tau_g=10^{-7}$.}}
		\label{tab:ex2_protocols}
		\small
		\begin{tabular}{@{}lllll@{}}
			\toprule
			Protocol & Equality tolerance & EA computation
			& Purpose & Results \\
			\midrule
			Original & $10^{-7}$ & Original runs
			& Strict feasibility baseline
			& Tables \ref{tab:example2_homotopy_baseline}, \ref{tab:example2_strict_feasibility} \\
			A1 & $10^{-4}$ & Reassess original runs
			& Effect of relaxed feasibility assessment
			& Table\,\ref{tab:ex2_updated_indicators}; Fig.\,\ref{fig:ex2_hv_igd_boxplots} \\
			A2 & $10^{-4}$ & New runs
			& Effect of relaxed tolerance
			& Table \ref{tab:ex2_updated_indicators}; Figs.\,\ref{fig:ex2_a2_decision},\ref{fig:ex2_hv_igd_boxplots} \\
			B1 & $10^{-7}$--$10^{-3}$ & Reassess original runs
			& Tolerance sensitivity
			& Fig.\,\ref{fig:ex2_b1_tolerance} \\
			B2 & Exact elimination & New reduced-space runs
			& Effect of equality elimination
			& Table\,\ref{tab:ex2_updated_indicators}; Figs.\,\ref{fig:ex2_b2_decision}--\ref{fig:ex2_objective_comparison} \\
			\bottomrule
		\end{tabular}
	\end{table}

	\add{To distinguish the effects of feasibility tolerance
		and equality-constraint handling, we consider the five
		experimental settings summarized in
		Table~\ref{tab:ex2_protocols}. In the original experiment,
		the evolutionary algorithms use the strict equality
		tolerance $\tau_h=10^{-7}$. Protocol A1 retains these
		original final populations but reassesses their feasibility
		using the relaxed tolerance $\tau_h=10^{-4}$, without
		rerunning the optimization. In contrast, A2 reruns the
		evolutionary algorithms with $\tau_h=10^{-4}$ incorporated
		into constraint handling during optimization. Protocol B1
		examines sensitivity to the equality tolerance by
		reassessing the original final populations over a range
		of tolerances, while B2 eliminates the equality constraints
		analytically and performs new optimization in the reduced
		decision space. The original homotopy results are retained
		without modification in all comparisons.}
	
	\add{We evaluate feasibility separately from the quality of the objective-space approximation. For a computed point $\bm{x}$, define the equality residual $r_h(\bm{x})=\max\{|x_4|,|x_3^2-x_2^2|\}$ and the inequality residual $r_g(\bm{x})$ as the maximum positive violation of the nonlinear inequality and the box constraints. A point is classified as feasible when $r_h\le\tau_h$ and $r_g\le10^{-7}$. We report normalized hypervolume (HV), generational distance (GD), and inverted generational distance (IGD) using the common sampled reference set and the same objective normalization across all studies. The reference set is a numerical approximation, not an analytically certified Pareto front.} 
	
	\add{As summarized in Table~\ref{tab:ex2_protocols},
		A1 and B1 reassess the original evolutionary final
		populations without additional optimization, whereas
		A2 and B2 involve new evolutionary runs. In particular,
		A1 changes only the feasibility assessment from
		$\tau_h=10^{-7}$ to $\tau_h=10^{-4}$, while A2 uses
		$\tau_h=10^{-4}$ during both optimization and subsequent
		evaluation. B1 examines the sensitivity of the original
		results to the equality tolerance, and B2 enforces the
		equality constraints analytically before optimization.
		The original homotopy baseline is unchanged throughout.}

	\begin{table}[h]
		\centering
		\caption{Feasibility on the original four-variable formulation under the strict tolerance $\tau_h=\tau_g=10^{-7}$. The best violation is the smallest reported aggregate violation over each method's runs; times are mean wall-clock seconds for evolutionary algorithms.}
		\label{tab:example2_strict_feasibility}
		\small
		\begin{tabular}{lrrrr}
			\toprule
			Method & Feasible runs & Feasible points & Best violation & Time (s) \\
			\midrule
			Homotopy & $1/1$ & 129 & $8.08\times10^{-28}$ & 1836.9 \\
			NSGA-II & $0/20$ & 0 & $1.75\times10^{-6}$ & 9.7 \\
			NSGA-III & $0/20$ & 0 & $2.10\times10^{-6}$ & 13.3 \\
			C-TAEA & $0/20$ & 0 & $9.13\times10^{-7}$ & 25.4 \\
			\bottomrule
		\end{tabular}
	\end{table}

	\add{As shown in Table~\ref{tab:example2_strict_feasibility},
		none of the 20 runs of NSGA-II, NSGA-III, or C-TAEA
		produced a feasible solution under the original strict
		tolerance of $10^{-7}$, whereas the homotopy run
		produced 129 feasible nondominated solutions. Consequently,
		HV, GD, and IGD are undefined for the strictly feasible
		evolutionary final sets in this experiment. These results
		demonstrate the ability of the homotopy implementation
		to satisfy the prescribed constraints more accurately
		in this benchmark, but do not establish a general
		advantage in Pareto-front coverage or computational
		efficiency.}
	
	\begin{figure}[h]
		\centering
		\includegraphics[width=.53\linewidth]{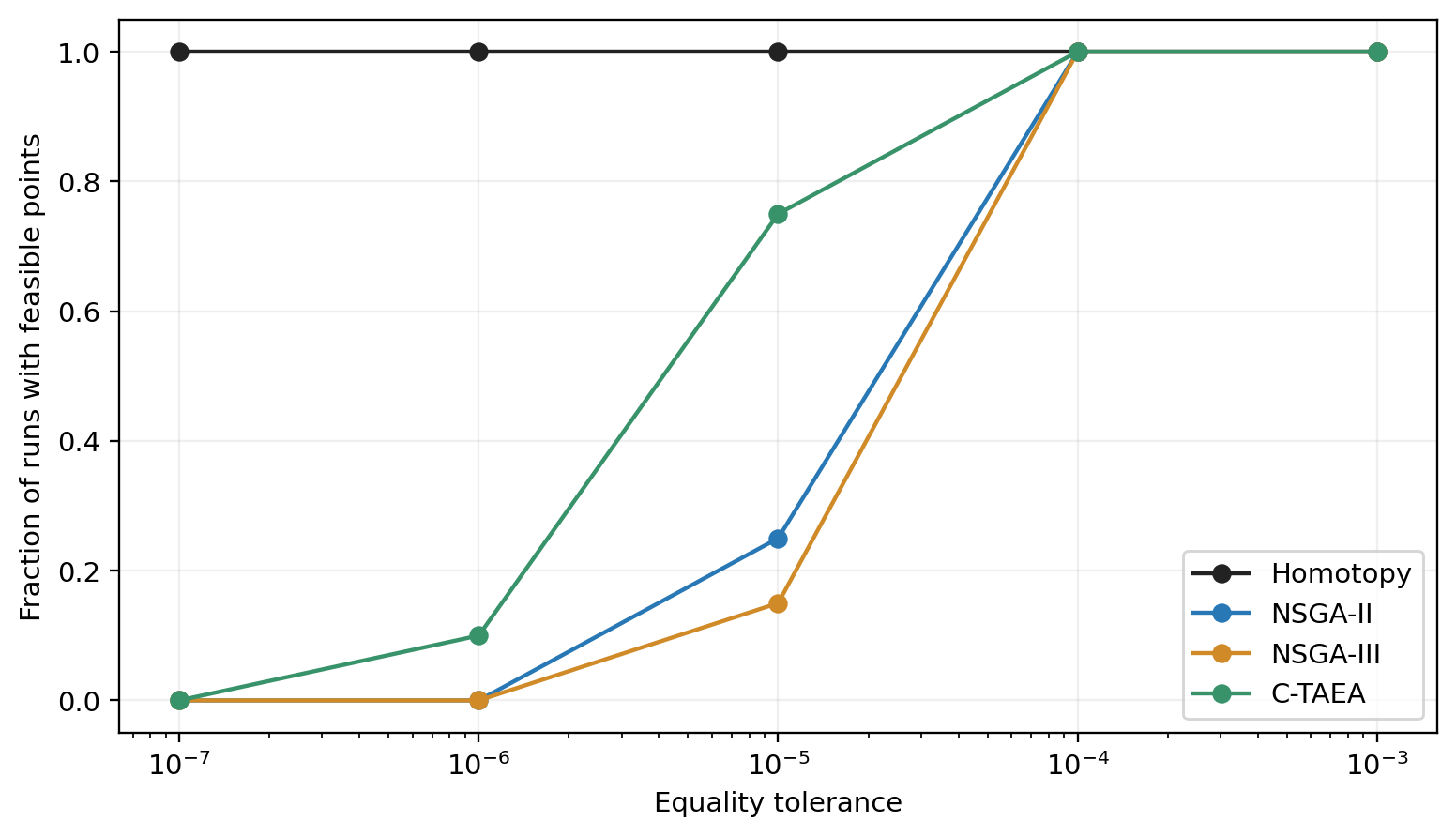}
		\caption{Study B1: fraction of original evolutionary runs containing at least one feasible final-population point as the equality tolerance is varied retrospectively. The inequality tolerance remains $10^{-7}$; no new optimization is performed.}
		\label{fig:ex2_b1_tolerance}
	\end{figure}

	\begin{table}[h]
		\centering
		\caption{Example~2 objective-space indicators. Evolutionary entries are means over 20 runs; the homotopy baseline is one run. A1 is retrospective reassessment, A2 is reoptimization with relaxed equality handling, and B2 uses analytical equality elimination. All indicators use the same sampled reference set and normalization.}
		\label{tab:ex2_updated_indicators}
		\small
		\begin{tabular}{llrrr}
			\toprule
			Study & Method & HV $\uparrow$ & IGD $\downarrow$ & GD $\downarrow$ \\
			\midrule
			A1 & NSGA-II & 0.404265 & 0.398642 & 0.015433 \\
			& NSGA-III & 0.372807 & 0.490843 & 0.006035 \\
			& C-TAEA & 0.566873 & 0.018493 & 0.005796 \\
			\midrule
			A2 & NSGA-II  & 0.390470 & 0.434610 & 0.015508 \\
			& NSGA-III & 0.365717 & 0.489956 & 0.014170 \\
			& C-TAEA   & 0.553286 & 0.037696 & 0.005774 \\
			\midrule
			B2 & NSGA-II & 0.563160 & 0.018374 & 0.007405 \\
			& NSGA-III & 0.561801 & 0.026321 & 0.006474 \\
			& C-TAEA & 0.567778 & 0.015299 & 0.005889 \\
			\midrule
			Original & Homotopy & 0.336810 & 0.246221 & 0.007367 \\
			\bottomrule
		\end{tabular}
	\end{table}
	
	\add{The tolerance-sensitivity results in
		Fig.~\ref{fig:ex2_b1_tolerance} demonstrate that the
		feasibility conclusions for the original evolutionary
		populations depend strongly on the equality threshold.
		At $\tau_h=10^{-5}$, at least one feasible final-population
		point is found in 5/20 NSGA-II runs, 3/20 NSGA-III runs,
		and 15/20 C-TAEA runs. At $\tau_h=10^{-4}$, all three
		methods achieve 20/20 feasible runs. The original
		homotopy baseline remains feasible throughout the
		tested tolerance range. Since B1 only reassesses fixed
		populations, these changes reflect feasibility
		classification rather than improvements obtained
		through additional optimization.}
	
	\begin{figure}[h]
		\centering
		\includegraphics[width=.98\linewidth]{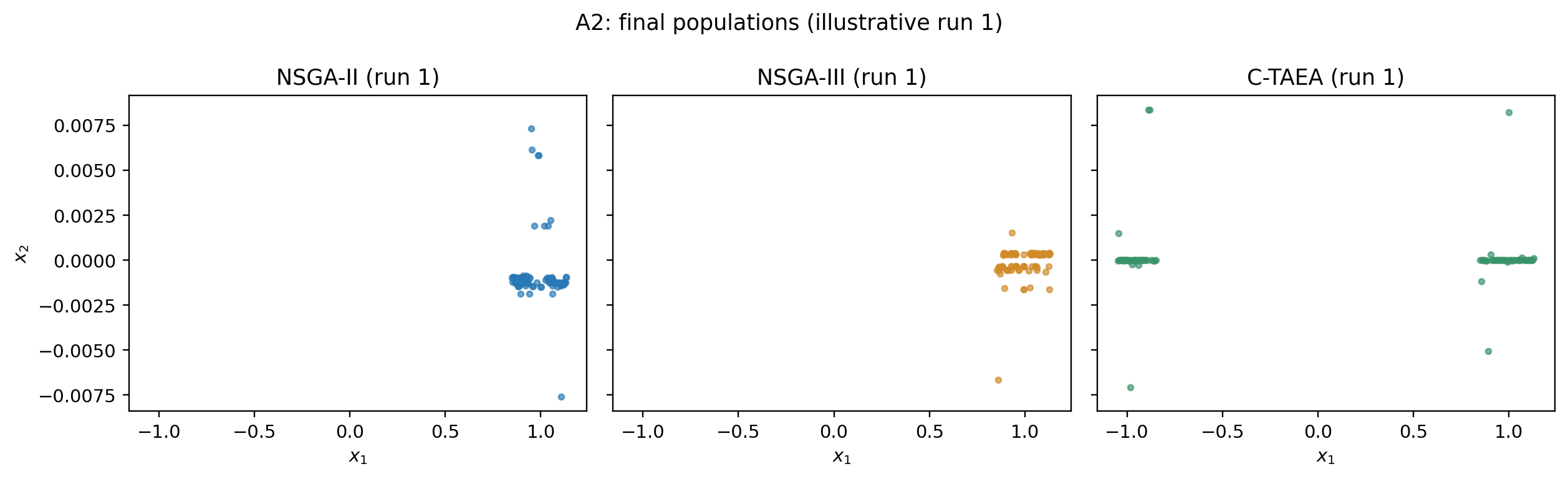}
		\caption{A2 final-population decision-space distributions for illustrative run 1 of each evolutionary algorithm. These plots show final populations, not exclusively feasible nondominated points.}
		\label{fig:ex2_a2_decision}
	\end{figure}

	\begin{figure}[t]
		\centering
		\includegraphics[width=.98\linewidth]{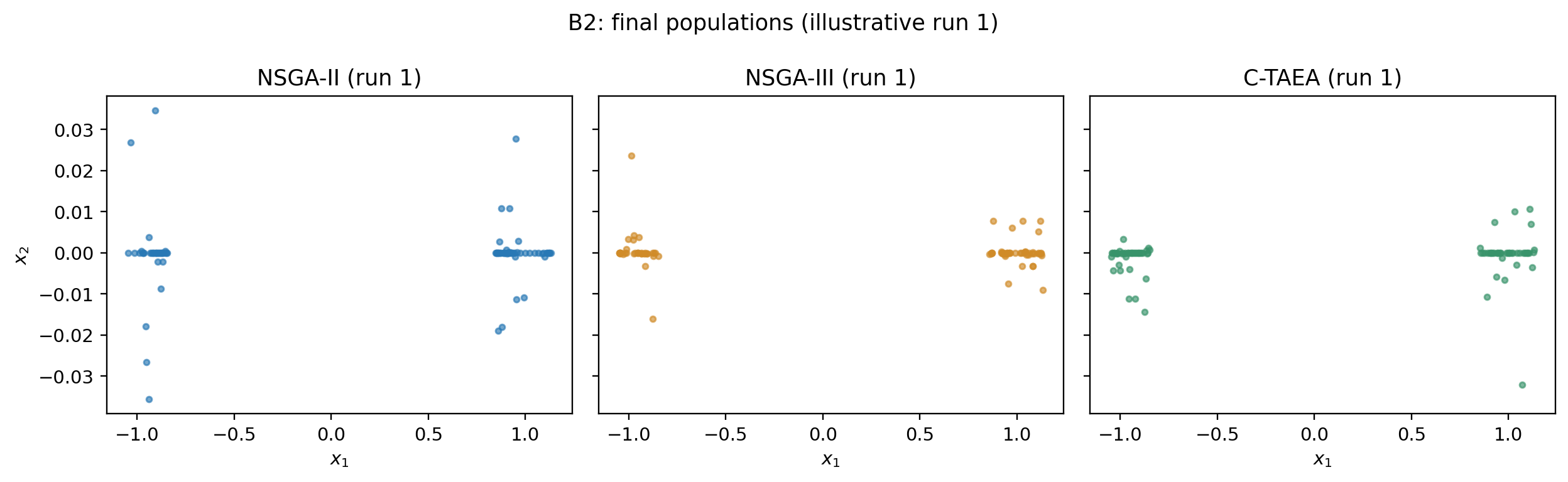}
		\caption{B2 final-population decision-space distributions for illustrative run 1 following analytical equality elimination. All three displayed populations reach both disconnected components.}
		\label{fig:ex2_b2_decision}
	\end{figure}

	\begin{figure}[h]
		\centering
		\includegraphics[width=.48\linewidth]{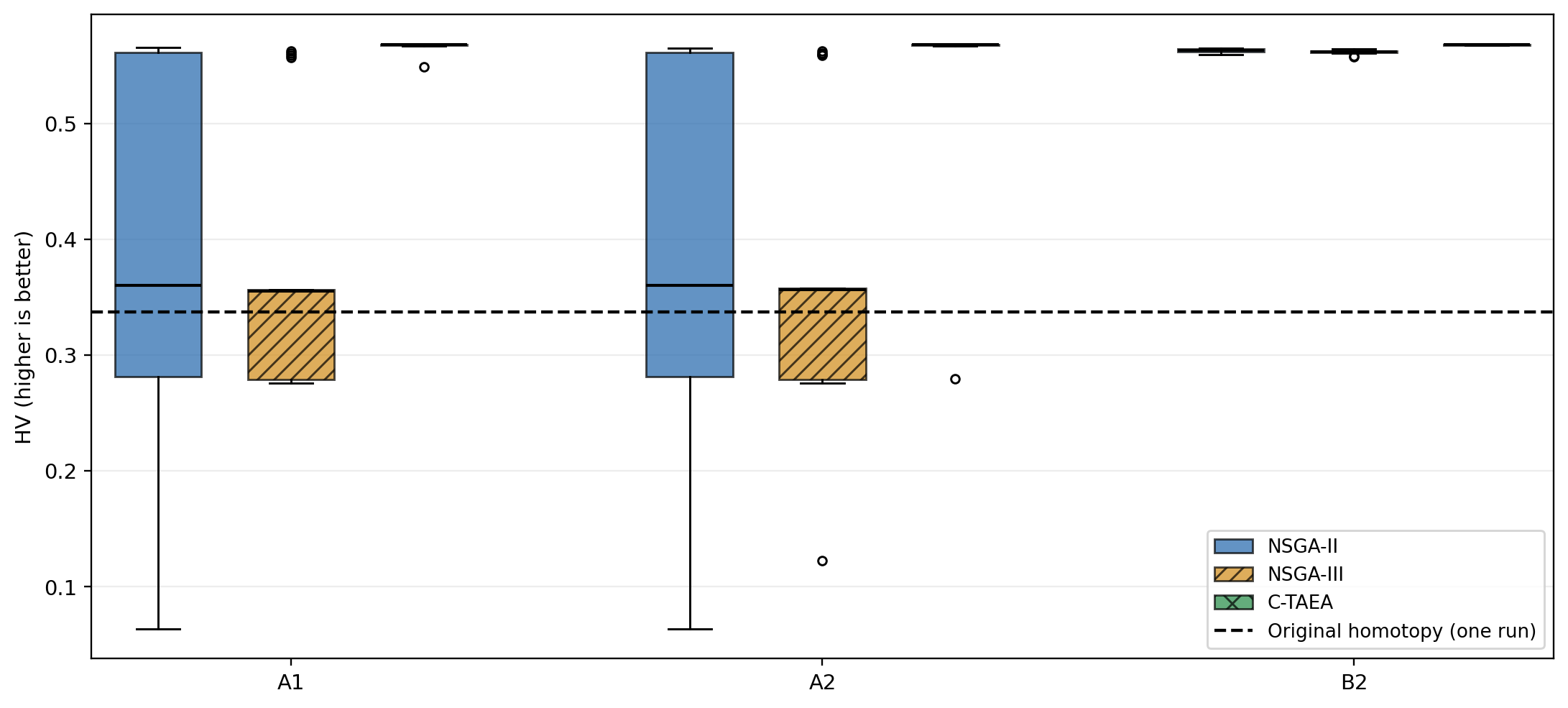}
		\includegraphics[width=.48\linewidth]{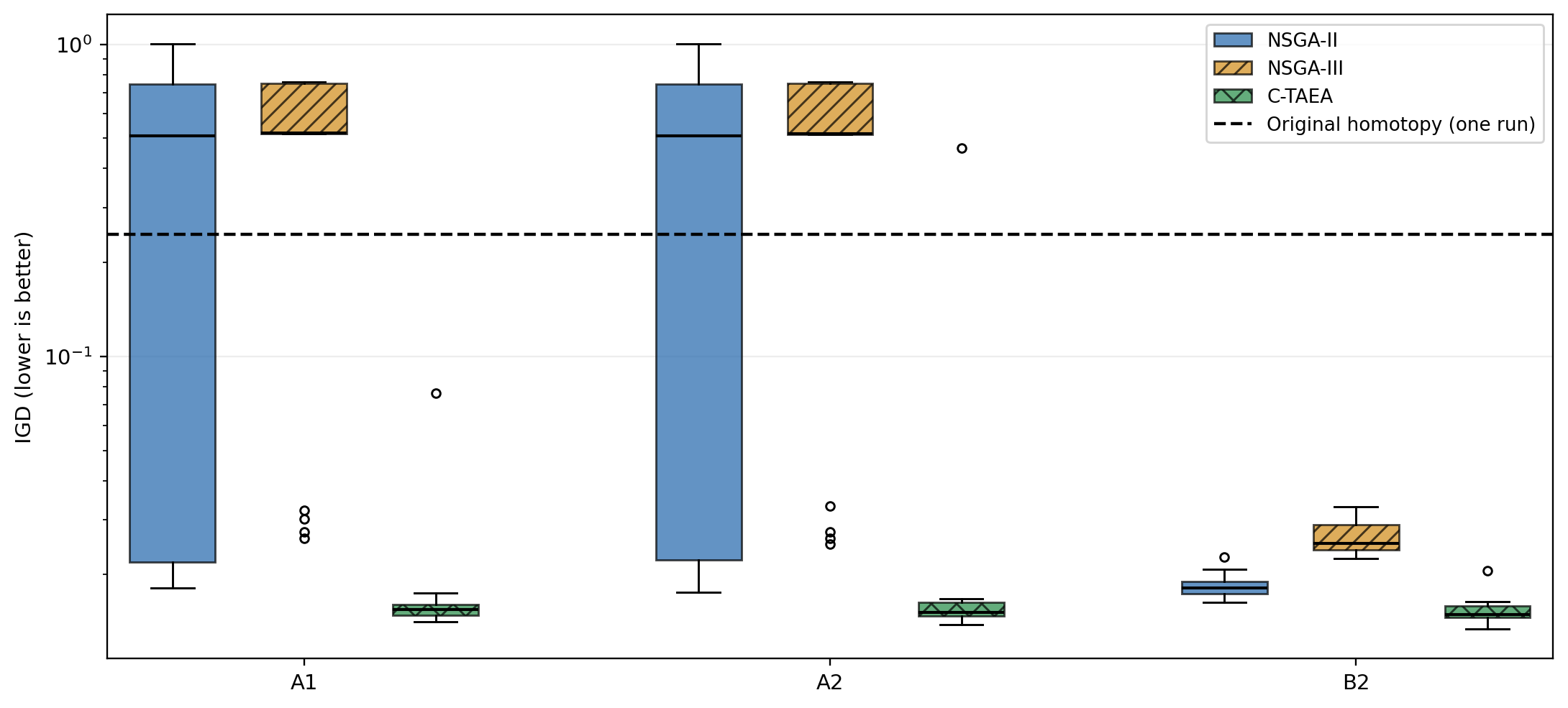}
		\caption{Left: Normalized HV across 20 evolutionary runs for A1, A2, and B2. The horizontal dashed line is the original single-run homotopy baseline. Color, box hatching, and outlines distinguish the methods in grayscale reproduction. Right: IGD across 20 evolutionary runs on a logarithmic scale, with the original single-run homotopy baseline. Smaller values indicate better coverage of the common sampled reference set.}
		\label{fig:ex2_hv_igd_boxplots}
	\end{figure}
	
	\add{The objective-space indicators in
		Table~\ref{tab:ex2_updated_indicators} and their
		distributions in Fig.~\ref{fig:ex2_hv_igd_boxplots}
		show substantial differences among the evolutionary
		methods and experimental protocols. In A1, C-TAEA
		covers both disconnected feasible components in all
		20 retrospectively feasible runs, compared with
		8 runs for NSGA-II and 4 for NSGA-III. Under A2
		reoptimization, the corresponding counts are 20, 7,
		and 4. The illustrative A2 final populations in
		Fig.~\ref{fig:ex2_a2_decision} further show differences
		in the distribution of solutions between the two
		components. In B2, all three evolutionary algorithms
		cover both components in all 20 runs, as also
		illustrated in Fig.~\ref{fig:ex2_b2_decision}.
		Their mean HV values range from 0.561801 to
		0.567778, and their mean IGD values range from
		0.015299 to 0.026321, compared with 0.336810
		and 0.246221, respectively, for the original
		homotopy baseline. Thus, analytical equality
		elimination leads to substantially broader coverage
		of the sampled reference front under the reported
		experimental settings.}

	\add{The improvement observed in B2 must be interpreted
		in light of the change in problem representation.
		The equality constraints admit the parameterization
		$x_4=0$ and $x_3=\pm x_2$. Because the objective
		functions and the remaining inequality depend only
		on $(x_1,x_2)$, B2 optimizes directly over these two
		coordinates and reconstructs equality-feasible
		points. Consequently, the B2 results in
		Table~\ref{tab:ex2_updated_indicators} and
		Fig.~\ref{fig:ex2_b2_decision} do not constitute
		a like-for-like comparison with the original
		four-variable homotopy implementation. In particular,
		exact equality satisfaction in B2 follows by
		construction rather than numerical convergence
		to the equality manifold.}

	\begin{figure}[h]
		\centering
		\includegraphics[width=.95\linewidth]{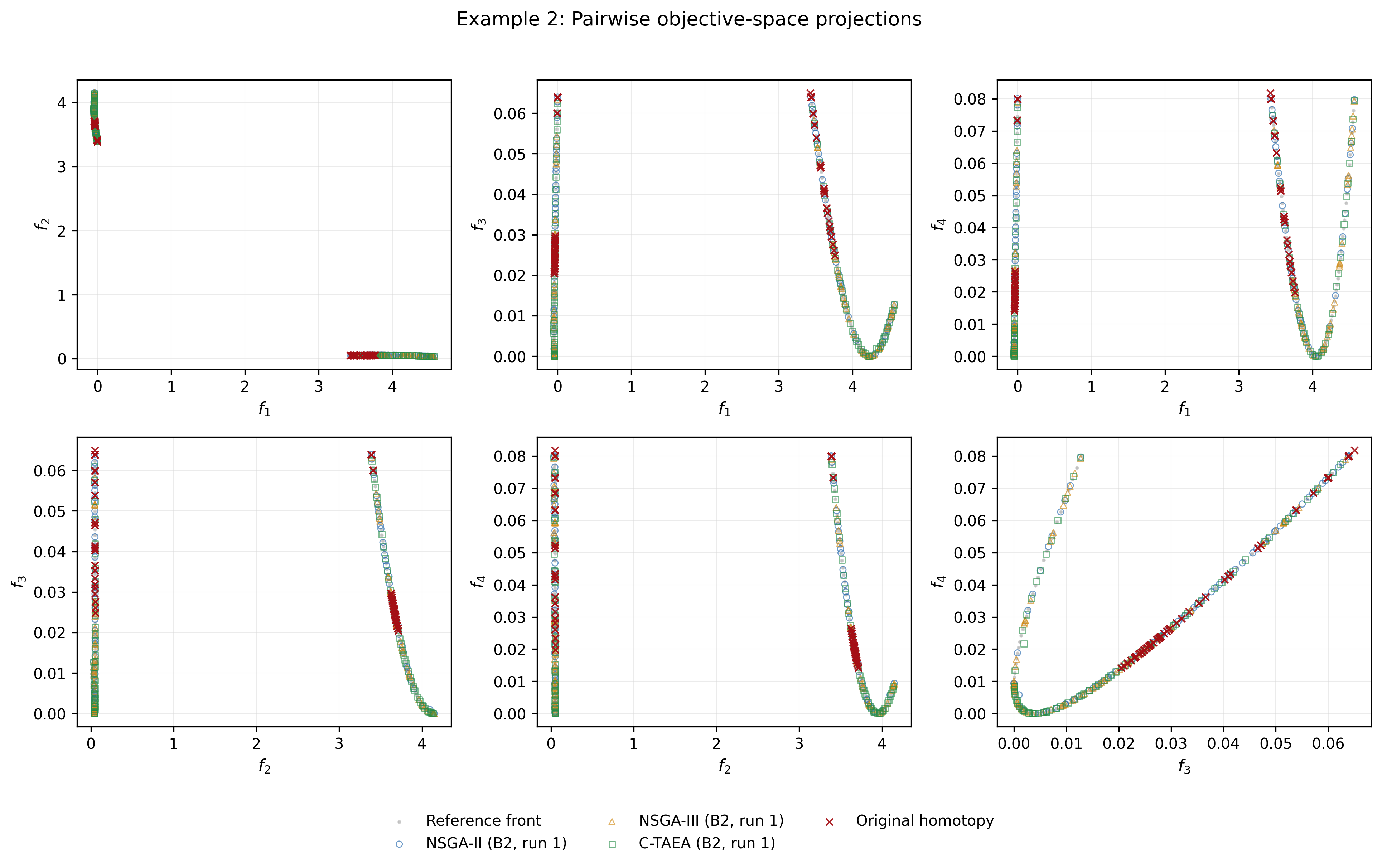}
		\caption{Six pairwise objective projections of the sampled reference set, original homotopy solutions, and illustrative run 1 of each B2 evolutionary method. Distinct marker shapes supplement color for grayscale reproduction. Pairwise projections are visual diagnostics and do not establish four-dimensional dominance.}
		\label{fig:ex2_objective_comparison}
	\end{figure}
	
	\add{The six pairwise objective-space projections
		in Fig.~\ref{fig:ex2_objective_comparison}
		provide a visual comparison of the sampled reference
		set, the original homotopy solutions, and illustrative
		B2 evolutionary populations. The homotopy solutions
		follow portions of the sampled reference geometry,
		whereas the B2 populations span more of the plotted
		trade-off curves. Consistently, Table~
		\ref{tab:ex2_updated_indicators} shows that the
		original homotopy GD ($0.007367$) is comparable
		in magnitude to the B2 GD values
		($0.005889$--$0.007405$), while its IGD
		($0.246221$) is considerably larger than the
		B2 values ($0.015299$--$0.026321$).
		This suggests that the principal difference is
		reference-front coverage rather than pointwise
		proximity alone. Since the displayed plots are
		two-dimensional projections of four-dimensional
		objective vectors, the quantitative indicators
		provide the primary evidence for this comparison.}
	
	\add{Finally, the original runtime comparison in
		Table~\ref{tab:example2_strict_feasibility}
		shows that the homotopy computation is substantially
		more expensive than the reported evolutionary runs.
		However, the methods use different computational
		budgets, and B2 additionally benefits from analytical
		equality elimination. Taken together, the strict
		feasibility results in
		Table~\ref{tab:example2_strict_feasibility},
		the tolerance study in Fig.~\ref{fig:ex2_b1_tolerance},
		and the approximation-quality results in
		Table~\ref{tab:ex2_updated_indicators}
		support a qualified conclusion: the original
		homotopy implementation satisfies stringent
		feasibility requirements on this degenerate
		equality-constrained example, whereas evolutionary
		optimization benefits substantially from relaxed
		equality handling and especially from analytical
		elimination. These findings do not establish
		universal superiority of either algorithmic family.}

	
	\begin{figure}[H]
		\centering
		\includegraphics[scale=0.25]{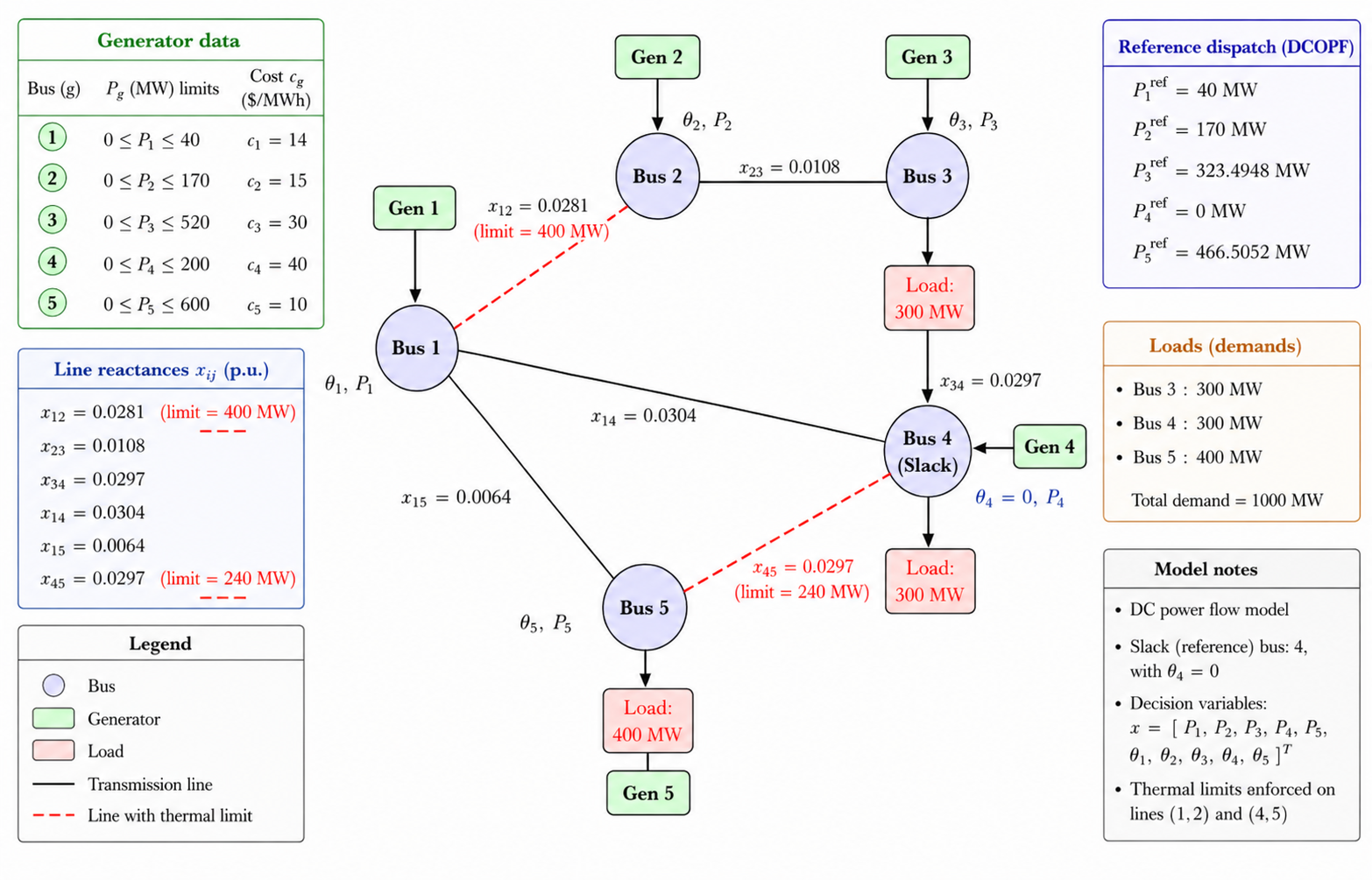}
		\caption{IEEE 5-bus test system with five generators, demands of 300 MW at bus 3, 300 MW at bus 4, and 400 MW at bus 5, transmission-line reactances, and thermal limits on lines $(1,2)$ and $(4,5)$. Bus 4 is selected as the reference (slack) bus with $\theta_4=0$.}
		\label{fig:ieee5bus}
	\end{figure}
	\begin{example}[\textbf{Three-Objective IEEE 5-Bus Optimal Power Flow Problem}]
		\label{ieee5bus}
		\del{This example considers the IEEE 5-bus transmission
			network for a multiobjective optimal power flow (MOPF)
			problem.}
		\add{This example considers a five-bus transmission
			network using network data from MATPOWER
			\citet{zimmerman2011matpower}. The three-objective
			optimal power flow (MOPF) formulation is constructed
			to investigate generation cost, dispatch tracking,
			and transmission-line stress.}
		The system consists of five buses connected by six transmission lines, where each bus is associated with a voltage phase angle \(\theta_i\), and each generator bus has a real-power generation variable \(P_i\) measured in megawatts (MW). The decision vector is
		\[
		\bm{x}=
		\begin{bmatrix}
			P_1 & P_2 & P_3 & P_4 & P_5 &
			\theta_1 & \theta_2 & \theta_3 & \theta_4 & \theta_5
		\end{bmatrix}^{T}.
		\]
		The network includes five dispatchable
		generators together with fixed loads of
		300 MW at \del{buses 2 and 3 and 400 MW at bus 4.}
		\add{buses 3 and 4 and 400 MW at bus 5.} 
		Figure~\ref{fig:ieee5bus} provides a compact but representative benchmark for studying the interaction among economic dispatch, preferred operating points, and transmission congestion. The resulting MOO problem is
		\begin{equation}
			\label{eq:MOPF_compact}
			\begin{aligned}
				\min_{\bm{x}}\quad
				&F(\bm{x})=\bigl(f_1(\bm{x}),f_2(\bm{x}),f_3(\bm{x})\bigr)\\[2mm]
				\text{s.t.}\quad
				&\add{P_1}
				-100\left(
				224.7319\theta_1
				-35.5872\theta_2
				-32.8947\theta_4
				-156.25\theta_5
				\right)=0,\\
				&\add{P_2}
				-100\left(
				-35.5872\theta_1
				+128.1798\theta_2
				-92.5926\theta_3
				\right)=0,\\
				&\add{P_3-300}
				-100\left(
				-92.5926\theta_2
				+126.2626\theta_3
				-33.67\theta_4
				\right)=0,\\
				&\add{P_4-300}
				-100\left(
				-32.8947\theta_1
				-33.67\theta_3
				+100.2348\theta_4
				-33.67\theta_5
				\right)=0,\\
				&\add{P_5-400}
				-100\left(
				-156.25\theta_1
				-33.67\theta_4
				+189.92\theta_5
				\right)=0,\\
				&-\pi\le \theta_i\le \pi,\qquad i=1,\dots,5,\\
				&\theta_4=0,\\
				&-400\le
				100(35.5872\,\theta_1-35.5872\,\theta_2)
				\le 400,\\
				&-240\le
				100(33.67\,\theta_4-33.67\,\theta_5)
				\le 240,\\
				&0\le P_1\le 40,\quad
				0\le P_2\le 170,\quad
				0\le P_3\le 520,\\
				&0\le P_4\le 200,\quad
				0\le P_5\le 600.
			\end{aligned}
		\end{equation}
		
		The objectives represent generation cost, dispatch tracking, and transmission-line stress:
		\begin{equation}
			\label{eq:objectives_vector}
			\begin{aligned}
				f_1(\bm{x})
				&=14P_1+15P_2+30P_3+40P_4+10P_5,\\
				f_2(\bm{x})
				&=(P_1-40)^2+(P_2-170)^2+(P_3-323.4948)^2\\
				&\quad +(P_4-0)^2+(P_5-466.5052)^2,\\
				f_3(\bm{x})
				&=\left(\frac{100(35.5872\,\theta_1-35.5872\,\theta_2)}{400}\right)^2
				+\left(\frac{100(33.67\,\theta_4-33.67\,\theta_5)}{240}\right)^2.
			\end{aligned}
		\end{equation}
		The first objective minimizes generation cost using the MATPOWER \texttt{GENCOST} coefficients~\citet{zimmerman2011matpower}. The second objective penalizes deviations from a preferred DC optimal power flow dispatch,
		\[
		(P_1^{\mathrm{ref}},P_2^{\mathrm{ref}},P_3^{\mathrm{ref}},P_4^{\mathrm{ref}},P_5^{\mathrm{ref}})
		=(40,170,323.4948,0,466.5052).
		\]
		The third objective penalizes normalized power flows on the thermally constrained lines \((1,2)\) and \((4,5)\), thereby acting as a convex proxy for network congestion.
	\end{example}

	The formulation is based on the DC power flow approximation, where real-power flows are modeled using voltage phase-angle differences and line reactances:
	\[
	P_{ij}=\add{100}\frac{\theta_i-\theta_j}{x_{ij}}.
	\]
	Bus 4 is selected as the slack bus with \(\theta_4=0\), and the first five equality constraints enforce power balance at each bus. The coefficients appearing in the balance equations correspond to the scaled susceptance matrix derived from the line reactances shown in Fig.~\ref{fig:ieee5bus}. Two transmission lines are subject to explicit thermal limits: line \((1,2)\) is limited to 400 MW, while line \((4,5)\) is limited to 240 MW.

	\add{We compute approximate solutions of the three-objective IEEE 5-bus problem using the homotopy Algorithm~2 with 49 prescribed
		objective-weight vectors. For each weight vector, a separate predictor--corrector continuation is performed from the same initial decision vector, followed by a fixed-weight Newton correction
		at $t=0$ when possible. The resulting solutions
		are aggregated over the 49 continuation attempts
		to construct the trade-off set shown in the
		following figures and tables. For comparison,
		the weighted-sum method uses the same 49 weight
		vectors and solves the corresponding scalarized
		optimization problems using SLSQP. We also
		consider NSGA-III and MOEA/D for a separate
		assessment of objective-space coverage; these
		evolutionary methods are not restricted to the
		49 prescribed weight vectors.}

	Below we present the numerical results for the multi-objective optimal power flow problem on the IEEE 5-bus system using the proposed homotopy method. Since the problem involves three competing objectives: generation cost, dispatch tracking, and network stress, the most informative results are those that illustrate the tradeoffs among these objectives and the corresponding changes in the physical operating point of the system. In particular, how the Pareto-critical points are distributed in objective space, how the generator dispatch changes across the Pareto front, how the constrained transmission lines behave, and how reliably the proposed homotopy method traces the solution set are presented.

	\medskip\noindent\textbf{\add{Homotopy Solutions and Objective Tradeoffs}}
	
	\begin{figure}[htbp]
		\centering
		\begin{subfigure}[t]{0.48\textwidth}
			\centering
			\includegraphics[width=\linewidth]
			{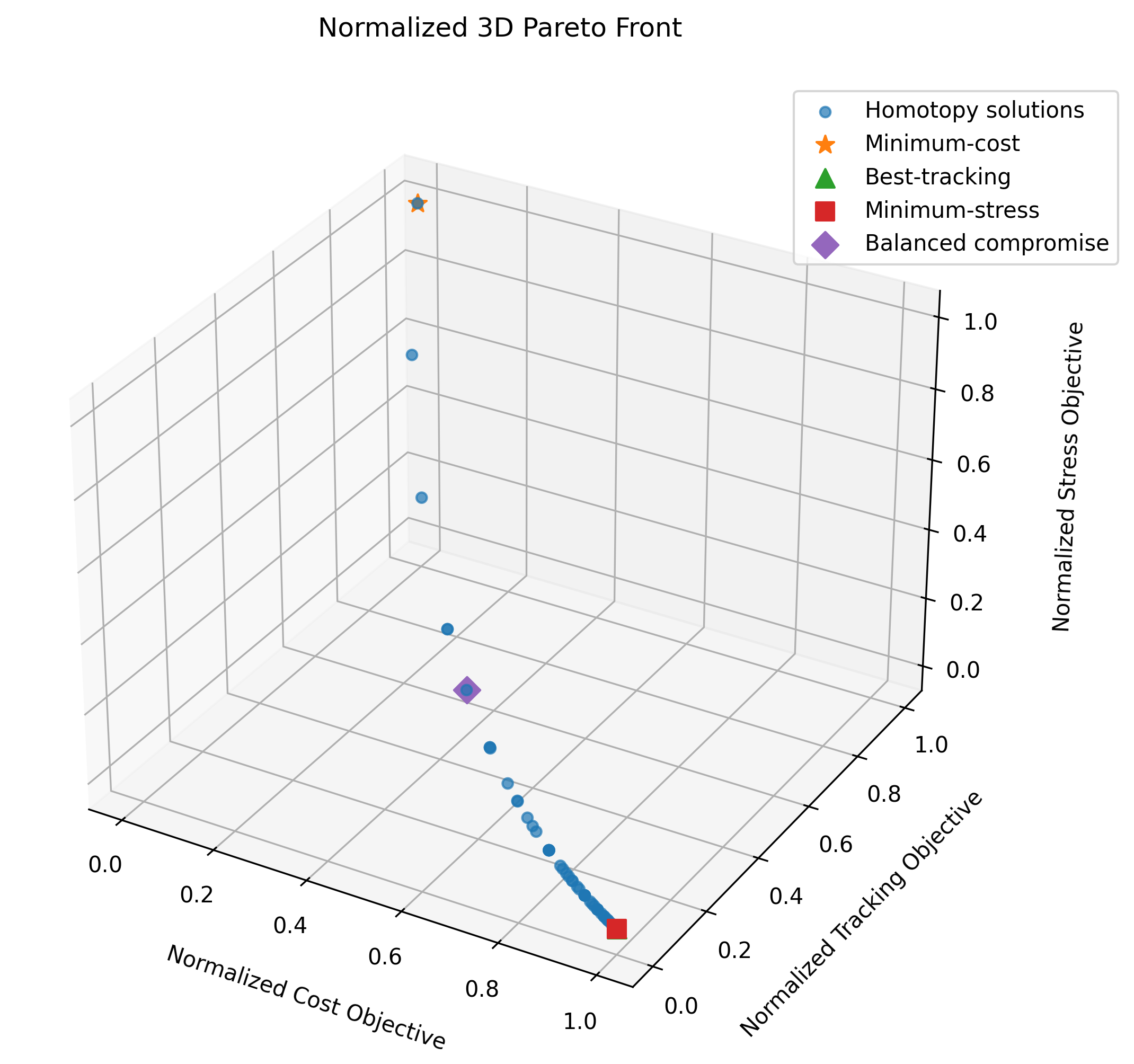}
			\caption{Three-dimensional Pareto front.}
			\label{fig:pareto3d}
		\end{subfigure}
		\hfill
		\begin{subfigure}[t]{0.48\textwidth}
			\centering
			\includegraphics[width=\linewidth]
			{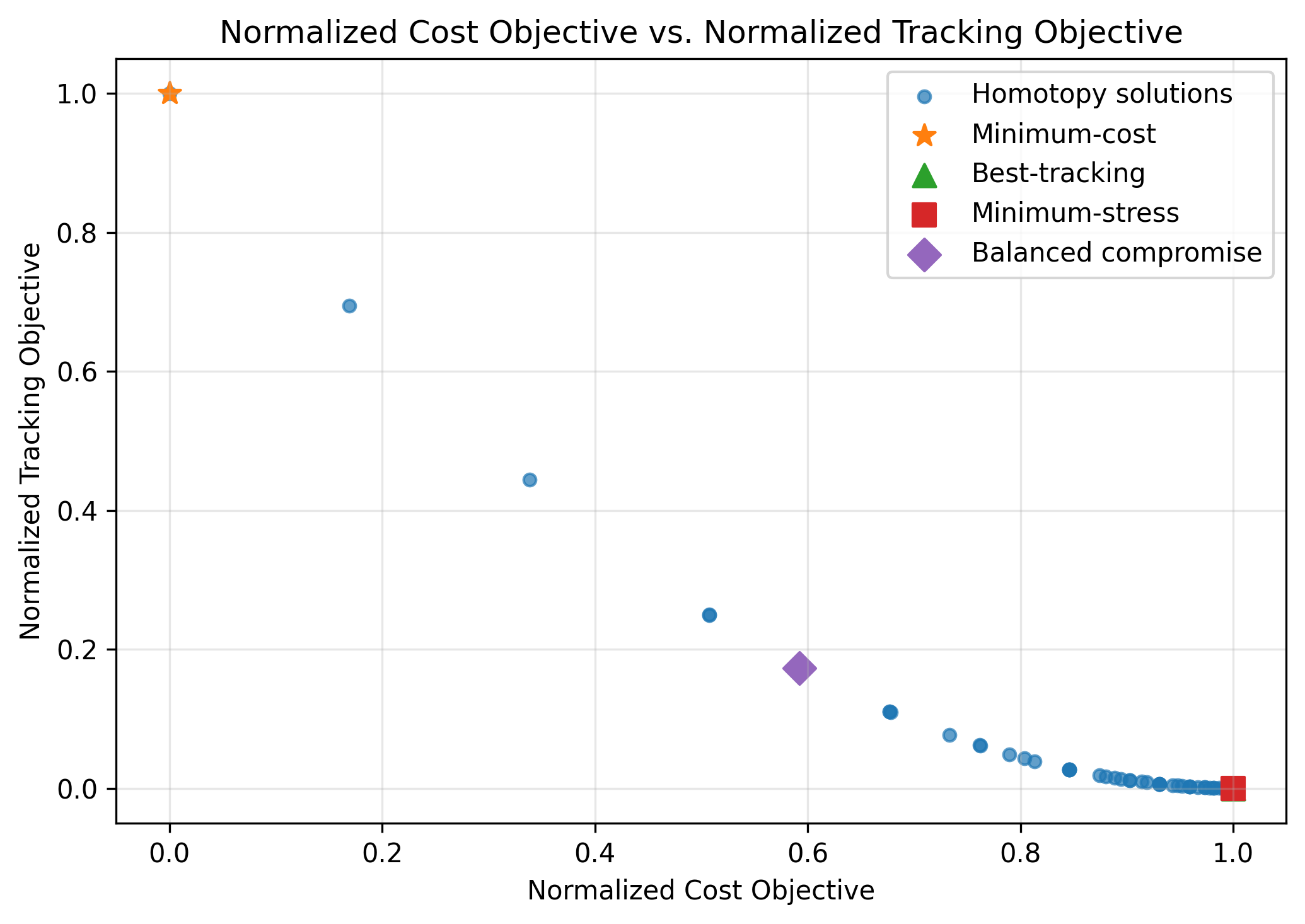}
			\caption{Cost versus dispatch tracking.}
			\label{fig:ieee5_hom_f1f2}
		\end{subfigure}
		\vspace{0.3cm}
		\begin{subfigure}[t]{0.48\textwidth}
			\centering
			\includegraphics[width=\linewidth]
			{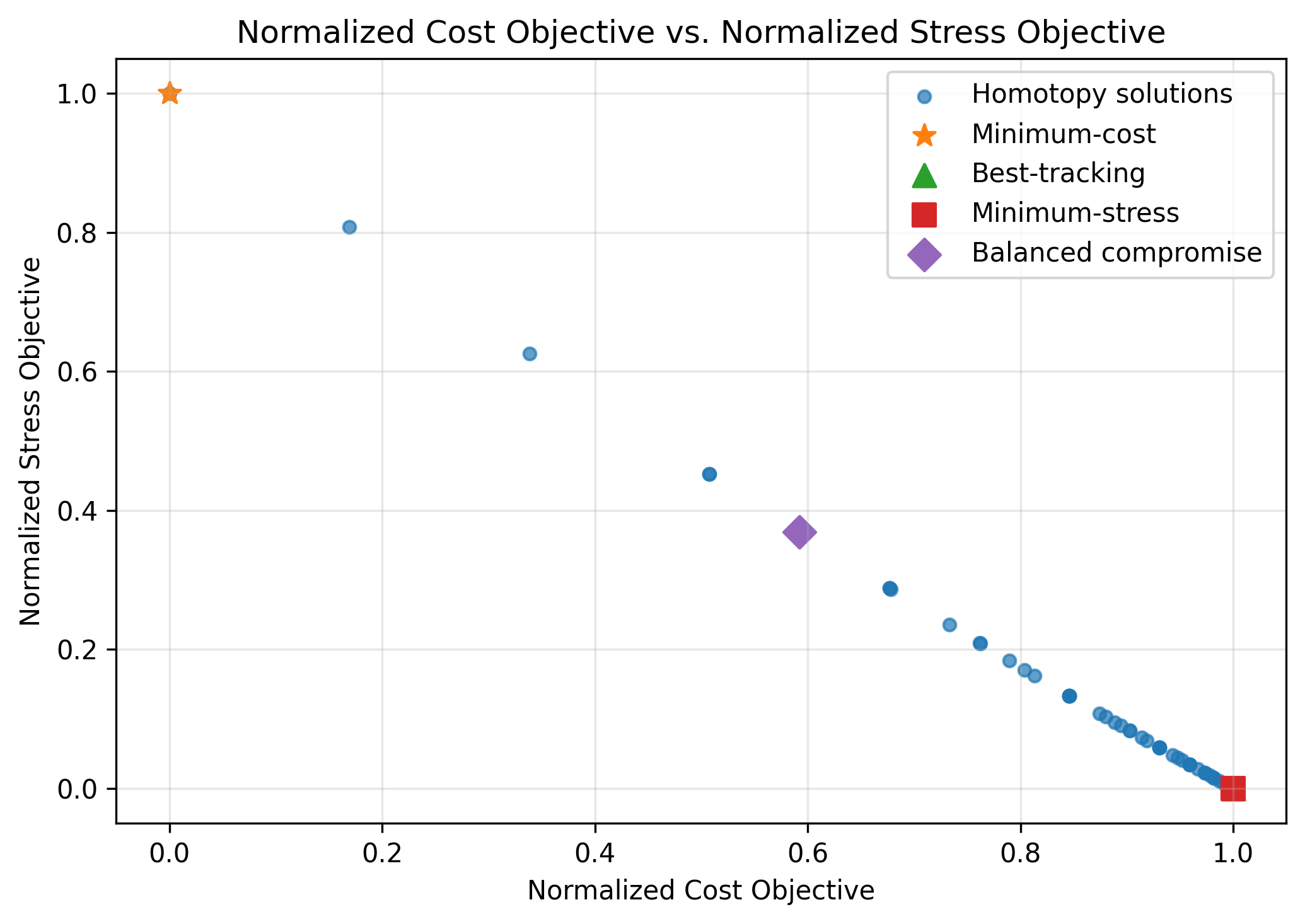}
			\caption{Cost versus network stress.}
			\label{fig:ieee5_hom_f1f3}
		\end{subfigure}
		\hfill
		\begin{subfigure}[t]{0.48\textwidth}
			\centering
			\includegraphics[width=\linewidth]
			{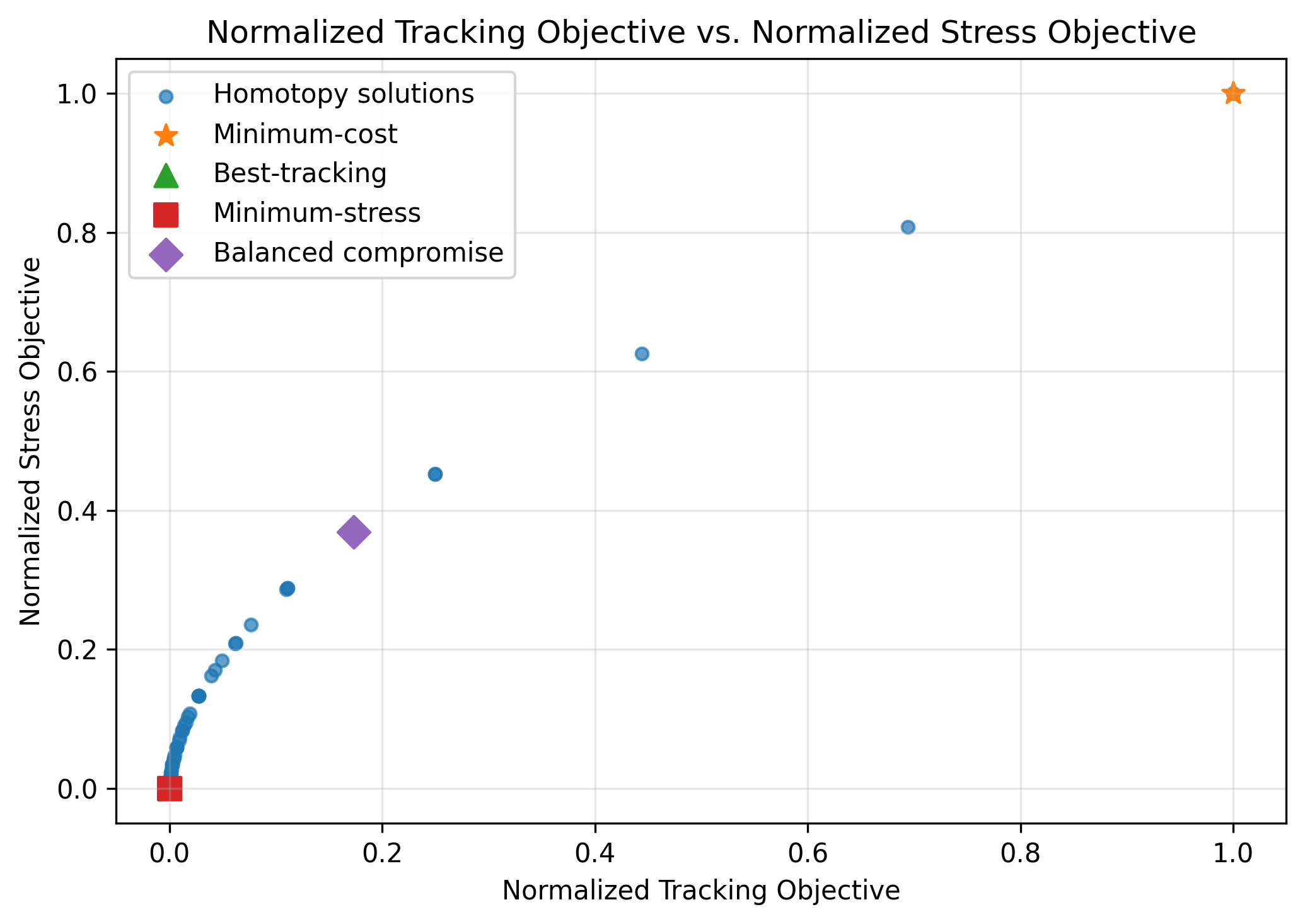}
			\caption{Dispatch tracking versus network stress.}
			\label{fig:ieee5_hom_f2f3}
		\end{subfigure}
		\caption{\add{
				Normalized objective-space visualization for the IEEE 5-bus
				three-objective optimal power flow problem using the
				49 corrected homotopy endpoints.
				(a) Three-dimensional representation of the computed
				Pareto-critical solutions;
				(b)--(d) pairwise projections illustrating the trade-offs
				among generation cost, dispatch tracking, and network stress.
				The highlighted points indicate the minimum-cost solution,
				best-tracking solution, minimum-stress solution, and
				normalized ideal-point compromise.
		}}
		\label{fig:ieee5_homotopy_pareto}
	\end{figure}
	
	To illustrate the trade-offs among the three objectives,
	four representative operating points are selected from
	the \add{computed homotopy solutions:}
	\begin{itemize}
		\item the minimum-cost solution,
		\item the minimum-stress solution,
		\item the best-tracking solution,
		\item a balanced compromise solution.
	\end{itemize}
	
	Figure~\ref{fig:ieee5_homotopy_pareto} illustrates the trade-off structure among the three competing objectives. 
	\del{The computed Pareto front appears smooth
		and continuous, indicating stable interactions
		among the objectives and consistent numerical
		behavior of the proposed homotopy framework.}
	\add{The 49 corrected homotopy endpoints form a
		well-defined, nonlinear trade-off curve in
		normalized objective space. The pairwise projections
		provide complementary views of the relationships
		among generation cost, dispatch tracking, and
		network stress over the sampled solutions.}
	\del{The geometry of the surface reveals the compromise relationships between generation cost, dispatch tracking performance, and network stress. }
	\add{The three-dimensional representation and its
		pairwise projections reveal the competing effects
		of these objectives. In particular, solutions with
		lower generation cost generally exhibit higher
		dispatch-tracking error and greater network stress,
		whereas solutions with higher generation cost
		achieve improvements in the other two objectives.}
	\del{In particular, the knee region corresponds to operating points the apparent knee region corresponds to operating points where small improvements in one objective produce relatively large degradations in the others, making these solutions especially attractive from a practical decision-making perspective.}
	\add{The nonlinear shape of the trade-off curve
		also indicates that the relative cost of improving
		one objective varies across the computed solution
		set. This motivates the selection of a balanced
		compromise solution rather than relying exclusively
		on any single-objective optimum.}

	\del{Strong agreement is observed between the homotopy solutions
		and the Weighted Sum solutions across all projections,
		providing additional numerical validation of the proposed method.}
	\del{For the $(f_1,f_2)$ projection, a pronounced nonlinear trade-off is observed between generation cost and dispatch tracking performance. 
		In particular, reducing the generation cost below approximately $1.72\times10^4$ leads to a rapid increase in the dispatch-tracking error, indicating a strong conflict between economic operation and reference-tracking objectives.}
	\del{The $(f_1,f_3)$ projection exhibits an approximately monotone and nearly linear trade-off between generation cost and network stress. 
		This behavior suggests that economically favorable operating points tend to increase transmission loading and overall network stress.}
	\del{Finally, the $(f_2,f_3)$ projection reveals a positive correlation between dispatch tracking performance and network stress. 
		Solutions with improved dispatch tracking generally correspond to lower network stress levels, indicating partial compatibility between these two objectives over the computed Pareto set.}
	
	\begin{table}[h]
		\caption{\add{Representative Pareto solutions with original objective values.}}
		\label{tab:representative-solutions}
		\small
		\begin{tabular}{lrrrrrr}
			\toprule
			Solution & w$_1$ & w$_2$ & w$_3$ & $f_1$ & $f_2$ & $f_3$ \\
			\midrule
			Minimum-cost & 0.450 & 0.075 & 0.475  & \textbf{16879.919078} & 1799.861537 & 0.182762 \\
			Best-tracking & 0.075 & 0.825 & 0.100  & 17470.805411 & \textbf{0.413194} & \textbf{0.164202} \\
			Minimum-stress & 0.075 & 0.825 & 0.100  & 17470.805411 & \textbf{0.413194} & \textbf{0.164202} \\
			Balanced compromise & 0.375 & 0.150 & 0.475  & 17229.905652 & 312.475869 & 0.171051 \\
			\bottomrule
		\end{tabular}
	\end{table}

	The corresponding objective values are reported in
	Table~\ref{tab:representative-solutions}.
	The pairwise projections provide a clearer interpretation
	of the interactions between individual objective pairs.
	
	\add{Across the 49 sampled weight vectors, the computed
		objective values satisfy
		$16879.9191 \le f_1 \le 17470.8054$,
		$0.4132 \le f_2 \le 1799.8615$, and
		$0.164202 \le f_3 \le 0.182762$.
		The $f_1$--$f_2$ projection exhibits a pronounced
		nonlinear trade-off. Over the sampled region, lower
		generation cost is associated with greater dispatch
		deviation and higher network stress. These solutions
		represent the portion of the objective-space trade-off
		structure reached by the selected weight vectors,
		rather than a complete characterization of the
		three-objective Pareto front.}
	
	The minimum-cost solution achieves the smallest
	generation cost objective \(f_1\)
	\add{among the 49 sampled homotopy endpoints},
	but exhibits significantly larger dispatch-tracking
	and stress objectives. This behavior indicates that
	\del{aggressively minimizing generation cost leads to
		deterioration in both tracking performance and network
		stress characteristics}
	\add{the lower-cost operating points in the sampled
		solution set involve a trade-off with dispatch-tracking
		accuracy and network stress}.
	
	In contrast, the best-tracking and minimum-stress
	solutions
	\del{are located very close to one another on the
		Pareto front, suggesting a strong correlation between
		the tracking objective \(f_2\) and the stress objective
		\(f_3\) for this test system}
	\add{coincide to the reported numerical precision.
		Thus, among the 49 sampled endpoints, the same
		operating point minimizes both \(f_2\) and \(f_3\).}
	\del{Improving dispatch tracking simultaneously reduces
		transmission-line stress, although at the expense
		of increased generation cost.}
	
	\add{Along the sampled homotopy solutions, improved
		dispatch tracking is accompanied by reduced network
		stress, although both improvements are generally
		associated with increased generation cost.}
	
	
	\begin{table}[h]
		\caption{\add{Generator dispatches (MW) for four
				representative solutions selected from the
				49 corrected homotopy endpoints of the IEEE
				5-bus system.}}
		\label{tab:dispatches}
		\small
		\begin{tabular}{lrrrrr}
			\toprule
			Solution & $P_1$ & $P_2$ & $P_3$ & $P_4$ & $P_5$ \\
			\midrule
			Minimum-cost & 40.0000 & 170.0000 & 293.4960 & 0.0000 & 496.5040 \\
			Best-tracking & 40.0000 & 170.0000 & 323.0403 & 0.0000 & 466.9597 \\
			Minimum-stress & 40.0000 & 170.0000 & 323.0403 & 0.0000 & 466.9597 \\
			Balanced compromise & 40.0000 & 170.0000 & 310.9953 & 0.0000 & 479.0047 \\
			\bottomrule
		\end{tabular}
	\end{table}
	
	\begin{figure}[h]
		\centering
		\includegraphics[width=.58\textwidth]{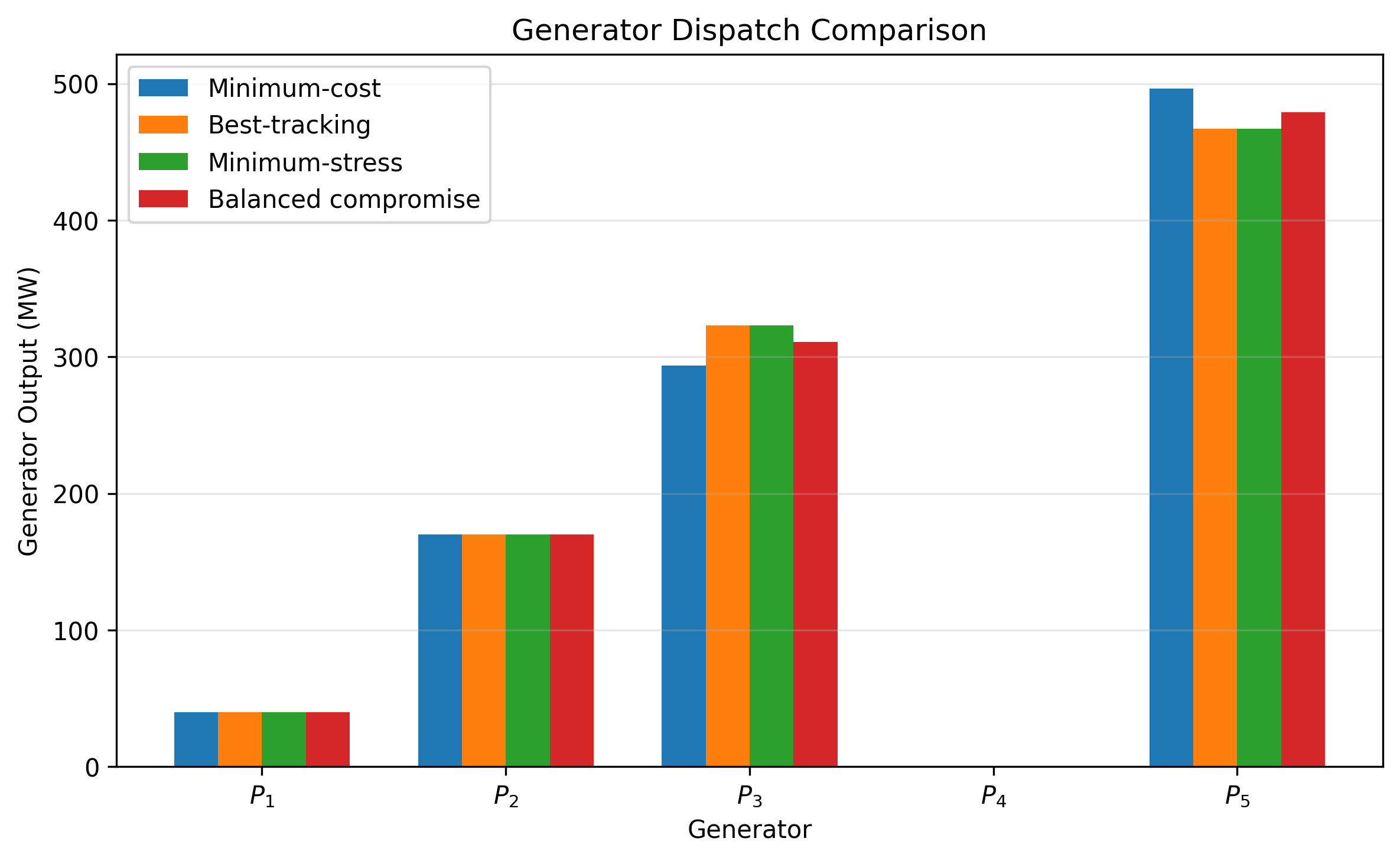}
		\caption{\add{Comparison of generator dispatches
				for four representative homotopy solutions.
				The variation occurs primarily between
				generators $P_3$ and $P_5$, while the
				best-tracking and minimum-stress selections
				coincide to numerical precision.}}
		\label{fig:ieee5_dispatch}
	\end{figure}

	\del{The balanced compromise solution lies near the knee
		region of the Pareto front, where substantial improvements
		in tracking and stress performance can still be achieved
		with relatively moderate increases in generation cost.
		Beyond this region, additional improvements in $f_2$
		and $f_3$ require disproportionately larger increases
		in $f_1$, indicating diminishing returns in the
		tradeoff structure.}
	\add{The balanced compromise solution provides an
		intermediate operating point between the minimum-cost
		and best-tracking solutions. Its dispatch reflects
		a reduction in tracking deviation and network stress
		relative to the minimum-cost solution, accompanied
		by an increase in generation cost.}
	\add{The generator dispatches associated with the
		representative Pareto solutions are reported in}
	Table~\ref{tab:dispatches} and
	Figure~\ref{fig:ieee5_dispatch}.
	\del{For this test system, the minimum-stress and
		best-tracking solutions coincide numerically,
		indicating a strong correlation between the
		dispatch-tracking and network-stress objectives
		over the computed Pareto set.
		These dispatch values provide additional physical
		interpretation of how the competing objectives
		redistribute generation across the network.}
	\add{For all four representative solutions,
		generators $P_1$ and $P_2$ operate at their
		upper generation limits of $40$ MW and $170$ MW,
		respectively, while $P_4$ remains zero to
		numerical precision. The changes in generation
		dispatch occur between $P_3$ and $P_5$, whose
		combined output remains $790$ MW to satisfy
		the power-balance requirement.}
	\add{In particular, moving from the minimum-cost
		solution to the best-tracking solution increases
		$P_3$ from $293.4960$ MW to $323.0403$ MW,
		while decreasing $P_5$ from $496.5040$ MW
		to $466.9597$ MW. The best-tracking and
		minimum-stress solutions coincide to the
		reported numerical precision, indicating that
		the same sampled operating point attains the
		smallest tracking deviation and network stress
		among the 49 homotopy endpoints. These results
		illustrate how the trade-offs among the three
		objectives are reflected in the physical
		generator dispatch.}
	
	\medskip\noindent\textbf{\add{Independent Scalarization Validation and Feasibility}}

	\begin{table}[htbp]
		\centering
		\caption{\add{Maximum residuals and cross-method objective differences for the 49-weight experiment.}}
		\label{tab:constraints_feas_opf}
		\begin{tabular}{lrr}
			\toprule
			Quantity & Homotopy & Weighted sum \\
			\midrule
			Maximum power-balance residual (MW) & $1.71\times10^{-13}$ & $1.56\times10^{-10}$ \\
			Maximum inequality-constraint violation & $4.64\times10^{-11}$ & $0$ \\
			\midrule
			\multicolumn{3}{l}{Maximum absolute objective differences at matching weights:}\\
			$|\Delta f_1|$ & \multicolumn{2}{c}{$1.59371\times10^{-3}$} \\
			$|\Delta f_2|$ & \multicolumn{2}{c}{$6.37443\times10^{-3}$} \\
			$|\Delta f_3|$ & \multicolumn{2}{c}{$5.2647\times10^{-8}$} \\
			\bottomrule
		\end{tabular}
	\end{table}

	\del{The homotopy solutions exhibit strong agreement with the
		Weighted Sum solutions across both the three-dimensional
		Pareto surface and the pairwise objective projections,
		providing additional numerical validation of the proposed
		continuation framework. Moreover, the homotopy method
		generates a denser and more uniform distribution of
		Pareto-critical solutions, particularly near the knee
		region where the trade-offs among the competing objectives
		are most significant. The resulting Pareto structure
		appears smooth and predominantly convex, indicating
		stable and well-behaved interactions among the objectives
		for this benchmark multiobjective optimal power flow
		problem.}
	\add{Table~\ref{tab:constraints_feas_opf} summarizes the
		constraint-feasibility checks and objective-value agreement
		between the proposed homotopy method and an independent
		weighted-sum implementation using SLSQP.
		All 49 homotopy runs reached corrected endpoints and
		satisfied the reported feasibility tolerances.
		The maximum terminal homotopy-system residual norm was
		$8.61\times10^{-10}$.}
	
	\add{The maximum power-balance residuals were
		$1.71\times10^{-13}$ MW for homotopy and
		$1.56\times10^{-10}$ MW for weighted-sum SLSQP.
		The maximum reported inequality violations were
		$4.64\times10^{-11}$ and zero, respectively.}
	
	\add{At matching weight vectors, the maximum absolute
		objective-value differences were
		$1.59371\times10^{-3}$ for $f_1$,
		$6.37443\times10^{-3}$ for $f_2$, and
		$5.2647\times10^{-8}$ for $f_3$.
		These differences are small relative to the corresponding
		objective scales, indicating close numerical agreement
		between the two methods for the sampled scalarizations.
		This agreement supports the accuracy and feasibility
		of the computed fixed-weight solutions, although
		small residuals and objective differences alone
		do not establish global optimality or characterize
		the complete Pareto front.}

	\del{Moreover, the homotopy method generates a denser and more uniform distribution of Pareto-critical solutions, particularly near the knee region where the trade-offs among the competing objectives are most significant. The resulting Pareto structure appears smooth and predominantly convex, indicating stable and well-behaved interactions among the objectives for this benchmark multiobjective optimal power flow problem.}
	
	\medskip
	\noindent\textbf{\add{Evolutionary Benchmarks and Objective-Space Coverage}}
	
	\add{NSGA-III and MOEA/D were each executed for 20 independent
		runs, with 20,000 objective evaluations per run and 91
		reference directions. The DC nodal power-balance equations
		were eliminated analytically in the evolutionary
		parameterization, while the remaining inequality constraints
		were handled using the same exterior-penalty rule for both
		algorithms. Feasibility was subsequently verified against
		the original unpenalized network model.}
	
	\add{The performance indicators were evaluated using common
		objective-wise normalization and a pooled feasible
		nondominated empirical reference set constructed from all
		four methods. The hypervolume reference point was
		$(1.1,1.1,1.1)$ in normalized coordinates. Because the
		reference set is empirical rather than the exact Pareto
		front, the indicators measure comparative sampled-front
		coverage rather than exact optimality error.}
	
	\begin{table}[htbp]
		\centering
		\caption{\add{Comparison of four methods for the IEEE 5-bus
				three-objective OPF problem. HV and IGD are computed using
				common objective normalization and an empirical reference
				set. Evolutionary results are reported as mean $\pm$
				sample standard deviation over 20 independent runs. Runtime is omitted for homotopy and weighted sum
				because comparable timing measurements are unavailable.}}
		\label{tab:ieee5_comparison}
		\small
		\begin{tabular}{lcccc}
			\toprule
			Method & Feasible solutions & HV $\uparrow$ &
			IGD $\downarrow$ & Runtime (s) \\
			\midrule
			Homotopy
			& 49/49 & 0.5307 & 0.2833 & -- \\
			Weighted sum
			& 49/49 & 0.5307 & 0.2833 & -- \\
			NSGA-III
			& 1.000 & $0.9196\pm0.0025$
			& $0.0349\pm0.0022$ & $5.90\pm1.39$ \\
			MOEA/D
			& 1.000 & $0.8411\pm0.0061$
			& $0.0699\pm0.0017$ & $48.77\pm2.75$ \\
			\bottomrule
		\end{tabular}
	\end{table}
	
	\add{As shown in Table~\ref{tab:ieee5_comparison},
		NSGA-III achieves the highest mean hypervolume
		($0.9196$) and lowest mean IGD ($0.0349$),
		followed by MOEA/D. Both evolutionary methods
		produce broader objective-space coverage than
		the 49-point homotopy and weighted-sum samples.}
	
	\add{The homotopy and weighted-sum methods yield
		essentially identical HV and IGD values,
		consistent with their close agreement in
		objective values at matching weight vectors.
		Their lower HV and higher IGD primarily reflect
		the restricted objective-space coverage of the
		selected weight vectors and should not be
		interpreted as evidence of inaccurate fixed-weight
		solutions.}
	
	\add{NSGA-III also exhibits a shorter mean runtime
		than MOEA/D under the implemented evolutionary
		settings. However, runtime comparisons with
		homotopy and weighted sum require accounting
		for their different computational tasks and
		evaluation budgets.}

	\begin{figure}[htbp]
		\centering
		\includegraphics[width=.90\textwidth]{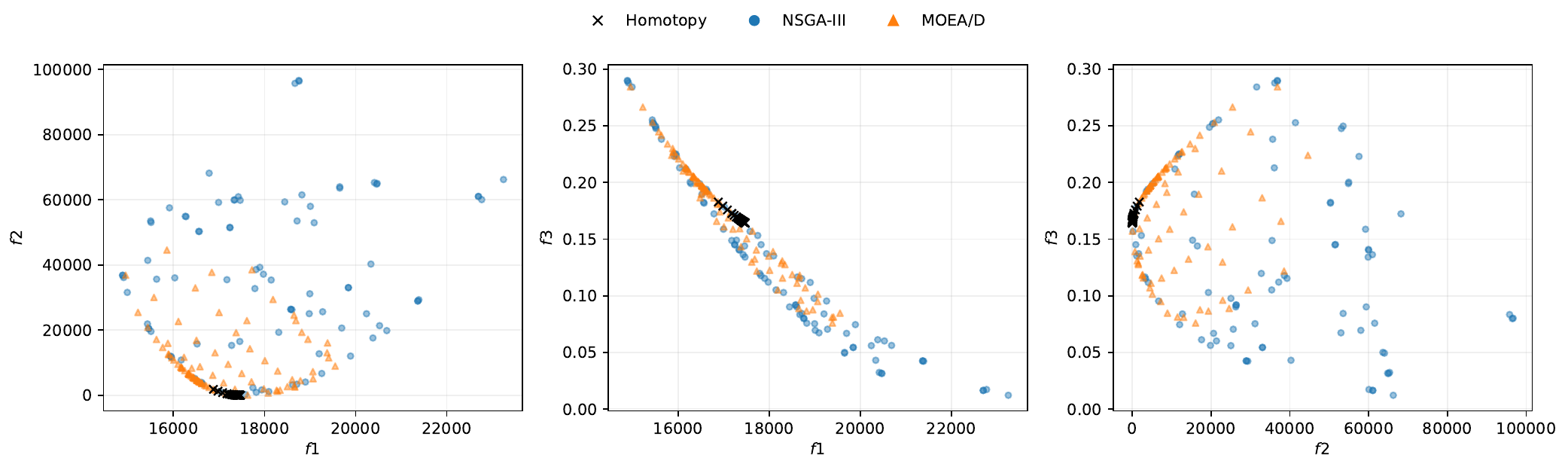}
		\caption{\add{Pairwise objective-space projections for the corrected
				homotopy solutions and representative NSGA-III and MOEA/D runs, selected
				by median IGD within each method. The evolutionary algorithms cover a
				broader portion of the observed objective space; the homotopy points
				correspond to 49 prescribed weights.}}
		\label{fig:ieee5_evolutionary_projections}
	\end{figure}
	
	\begin{figure}[htbp]
		\centering
		\includegraphics[width=.5\textwidth]{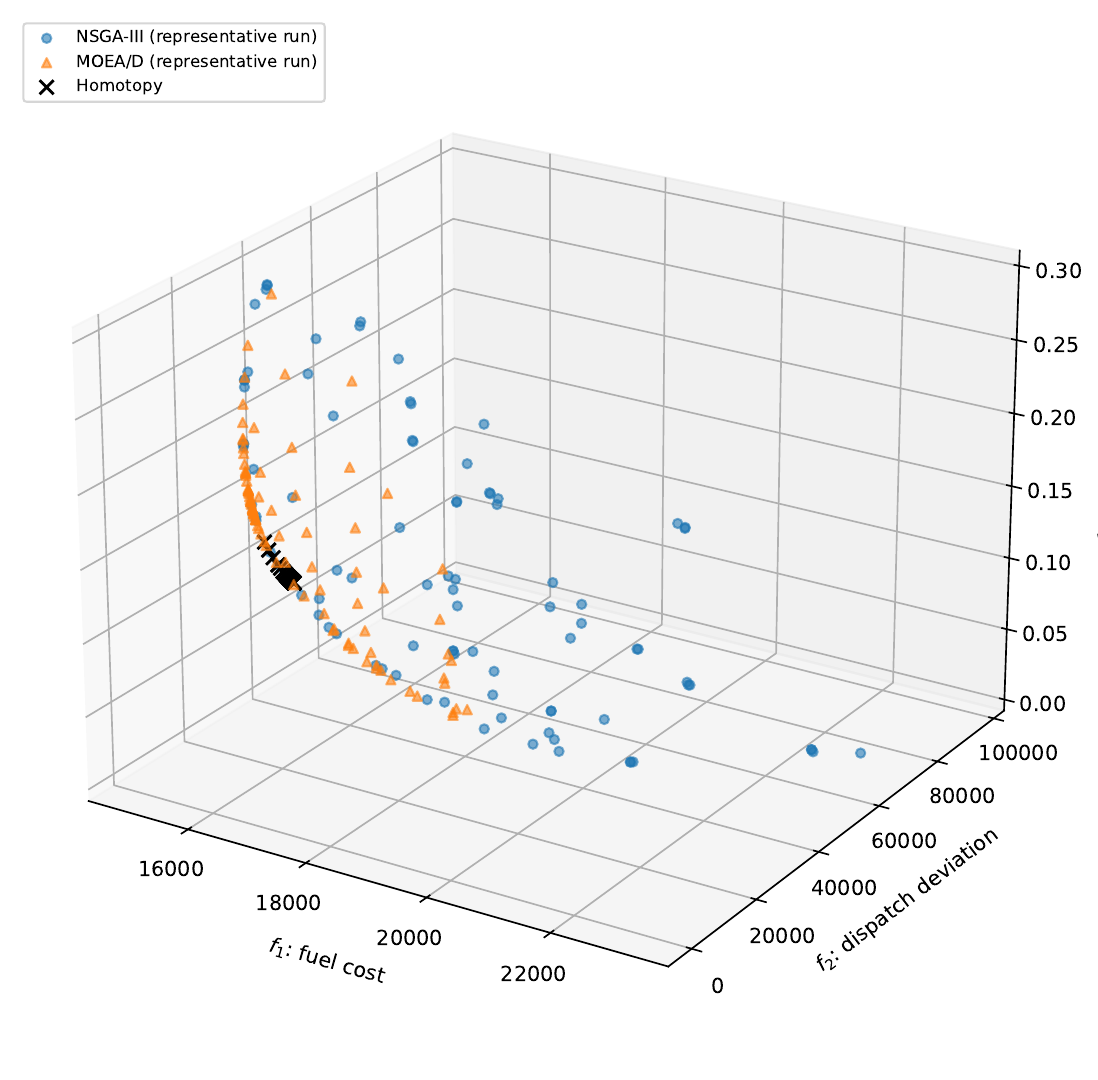}
		
		\caption{\add{Three-dimensional objective-space comparison
				of the 49 corrected homotopy endpoints and representative
				NSGA-III and MOEA/D runs selected by median IGD.
				The evolutionary methods sample a broader region of
				the observed objective space, while the homotopy solutions
				represent the selected scalarization weights.
				All three objectives are displayed on their original scales.}}
		\label{fig:ieee5_evolutionary_3d}
	\end{figure}

	\add{Figures~\ref{fig:ieee5_evolutionary_projections} and
		\ref{fig:ieee5_evolutionary_3d} compare the objective-space
		distributions obtained by the homotopy method, NSGA-III, and MOEA/D.
		The pairwise projections in
		Fig.~\ref{fig:ieee5_evolutionary_projections} show that the
		evolutionary methods explore a broader range of generation costs,
		dispatch-tracking deviations, and network-stress values than the
		49-weight homotopy sample. In particular, NSGA-III produces a more
		widely dispersed set of solutions, whereas MOEA/D exhibits a
		more concentrated distribution along portions of the observed
		trade-off structure.}
	
	\add{The three-dimensional comparison in
		Fig.~\ref{fig:ieee5_evolutionary_3d} further illustrates that
		the homotopy solutions occupy a relatively restricted region
		of the sampled objective space, while the evolutionary methods
		explore additional trade-off regions. This difference reflects
		the distinct sampling strategies: the homotopy method computes
		solutions for prescribed scalarization weights, whereas the
		evolutionary algorithms seek a distributed approximation to
		the Pareto front.}
	
	\add{Consistent with these observations, Table~\ref{tab:ieee5_comparison}
		shows that NSGA-III achieves the largest mean HV and smallest
		mean IGD under the chosen evaluation budget and empirical-reference
		definition. The nearly identical HV and IGD values for homotopy
		and weighted sum are consistent with their closely matched
		fixed-weight solutions. Because the methods generate different
		numbers and distributions of points, these performance indicators
		measure comparative objective-space coverage and do not imply
		that the homotopy endpoints are less accurate. The reported
		runtimes refer to the respective implementations and are not
		normalized by the number of solutions or computational work.}

	
	\begin{example}[\textbf{Three-Objective Problems with Analytically Known Pareto Fronts}]
		\label{EX4}
		
		\add{We consider two three-objective optimization problems
			defined on the same feasible region but with different
			objective functions. These problems allow us to examine
			how objective-function curvature and initialization affect
			the behavior of the proposed homotopy continuation method.}
		Let $\bm{x}=(x_1,x_2)^T$ and define the common feasible set
		\begin{equation}
			\label{eq:example4_feasible}
			\Omega=
			\left\{
			\bm{x}\in\mathbb{R}^2:
			x_1\geq0,\;
			x_2\geq0,\;
			x_1+x_2\leq1
			\right\}.
		\end{equation}
		The feasible set $\Omega$ is a nonempty, compact triangle,
		so assumption (A) holds. Writing the constraints as
		$g_1(\bm{x})=-x_1\leq0$,
		$g_2(\bm{x})=-x_2\leq0$, and
		$g_3(\bm{x})=x_1+x_2-1\leq0$, their gradients are
		$(-1,0)^T$, $(0,-1)^T$, and $(1,1)^T$, respectively.
		At most two constraints are active at any feasible point,
		and the active gradients are linearly independent.
		Thus assumption (B) holds.
		
		Moreover, for the standard outward normal cone
		$\mathcal N_\Omega(\bm{x})$, any
		$\bm{x}^0\in\Omega^0$ satisfies assumption (C).
		Indeed, if
		$\bm{x}^0-\bm{x}\in\mathcal N_\Omega(\bm{x})$
		for some $\bm{x}\in\Omega$, the normal-cone inequality
		implies
		$\|\bm{x}^0-\bm{x}\|^2\leq0$,
		and hence $\bm{x}=\bm{x}^0$.
		This excludes boundary points and establishes
		(C)(i). Since
		$\mathcal N_\Omega(\bm{x})=\{\bm{0}\}$
		for $\bm{x}\in\Omega^0$, condition (C)(ii) follows.
		Therefore, both problems satisfy assumptions
		(A)--(C).
		
		\medskip
		\noindent\add{\textbf{(a) Convex MOO problem.}}
		
		Consider the three-objective problem
		\begin{equation}
			\label{eq:example4_convex}
			\min_{\bm{x}\in\Omega}
			F^+(\bm{x})
			=
			\bigl(x_1^2,\;x_2^2,\;(1-x_1-x_2)^2\bigr).
		\end{equation}
		Every feasible point is Pareto optimal. Indeed, suppose that two
		feasible points $\bm{x}$ and $\widehat{\bm{x}}$ satisfy
		$F^+(\bm{x})\leq F^+(\widehat{\bm{x}})$ componentwise.
		Since $x_1,x_2,\widehat{x}_1,\widehat{x}_2\geq0$, the first two
		objective inequalities imply
		$x_1\leq\widehat{x}_1$ and
		$x_2\leq\widehat{x}_2$.
		The third objective inequality, together with
		$1-x_1-x_2\geq0$ and
		$1-\widehat{x}_1-\widehat{x}_2\geq0$, gives
		$x_1+x_2\geq\widehat{x}_1+\widehat{x}_2$.
		These inequalities can hold simultaneously only if
		$\bm{x}=\widehat{\bm{x}}$. Hence the entire feasible triangle
		$\Omega$ is the Pareto set.

		Since $x_1,x_2$, and $1-x_1-x_2$ are nonnegative
		throughout $\Omega$, we have
		$\sqrt{f_1(\bm{x})}=x_1$,
		$\sqrt{f_2(\bm{x})}=x_2$, and
		$\sqrt{f_3(\bm{x})}=1-x_1-x_2$.
		Their sum is therefore one. As every feasible point
		is Pareto optimal, the exact Pareto front is
		\begin{equation}
			\label{eq:example4_convex_pf}
			\sqrt{f_1}+\sqrt{f_2}+\sqrt{f_3}=1,
			\qquad f_i\geq0.
		\end{equation}
		For any positive normalized weight vector
		$\bm{\lambda}=(\lambda_1,\lambda_2,\lambda_3)^T\in W^{++}$,
		the interior Pareto-criticality conditions are
		\[
		\lambda_1x_1
		=
		\lambda_3(1-x_1-x_2),
		\qquad
		\lambda_2x_2
		=
		\lambda_3(1-x_1-x_2).
		\]
		Using these equations and the identity
		$x_1+x_2+(1-x_1-x_2)=1$, the analytical interior
		Pareto-critical point is
		\begin{equation}
			\label{eq:example4_exact_kkt}
			x_1^*=
			\frac{\lambda_1^{-1}}
			{\sum_{j=1}^3\lambda_j^{-1}},
			\qquad
			x_2^*=
			\frac{\lambda_2^{-1}}
			{\sum_{j=1}^3\lambda_j^{-1}}.
		\end{equation}
		Furthermore, the Hessian of the corresponding positively weighted
		objective is
		\[
		\nabla_{\bm{x}}^2
		\left(\sum_{i=1}^3\lambda_i f_i\right)
		=
		2\begin{pmatrix}
			\lambda_1+\lambda_3 & \lambda_3\\
			\lambda_3 & \lambda_2+\lambda_3
		\end{pmatrix},
		\]
		which is positive definite for every $\bm{\lambda}\in W^{++}$.
		Consequently, $\bm{x}^*$ is the unique minimizer of the
		corresponding weighted objective. Thus, for each positive weight vector, the analytical
		solution \eqref{eq:example4_exact_kkt} is the unique
		global minimizer of the corresponding weighted-sum problem.
		It therefore provides an exact reference for assessing
		the accuracy of the proposed continuation method.

		For the numerical experiment, 21 positive weight vectors were selected
		from a simplex-lattice discretization of $W^{++}$, and four strictly
		feasible initial points were used for each weight, giving 84
		continuation paths. Algorithm~2 successfully followed all 84 paths to
		feasible terminal points. All 84 terminal points were within $10^{-3}$
		in objective space of their corresponding analytical interior
		Pareto-critical points, and no path terminated on the boundary.
		The median and maximum objective-space errors were
		\[
		9\times10^{-6}
		\qquad\text{and}\qquad
		1.14\times10^{-4},
		\]
		respectively. The stationarity and complementarity
		residuals are evaluated using the original
		KKT system at $t=0$, rather than the
		homotopy equations at the terminal
		continuation parameter. The maximum stationarity residual was
		$4.7\times10^{-5}$ and the maximum complementarity residual was
		$2\times10^{-6}$. Using a dense discretization of the exact Pareto
		front as the reference set,
		\[
		\mathrm{GD}=0.003195,
		\qquad
		\mathrm{IGD}=0.108137.
		\]
		The GD value is computed relative to the discretized
		reference front and therefore also reflects the finite
		resolution of that reference set. It need not coincide
		with the objective-space error measured against the
		analytical solution at each prescribed weight.
		The relatively larger IGD reflects the use of only 21 weight vectors
		and the nonlinear mapping from regularly distributed weights to
		Pareto-critical points; a regular weight discretization does not, in
		general, produce a uniform distribution on the Pareto front.
		
		Figure~\ref{fig:example4_convex_front} compares the computed endpoints with the exact Pareto front.

		\begin{figure}[htbp]
			\centering
			\begin{subfigure}[t]{0.48\textwidth}
				\centering
				\includegraphics[width=\linewidth]{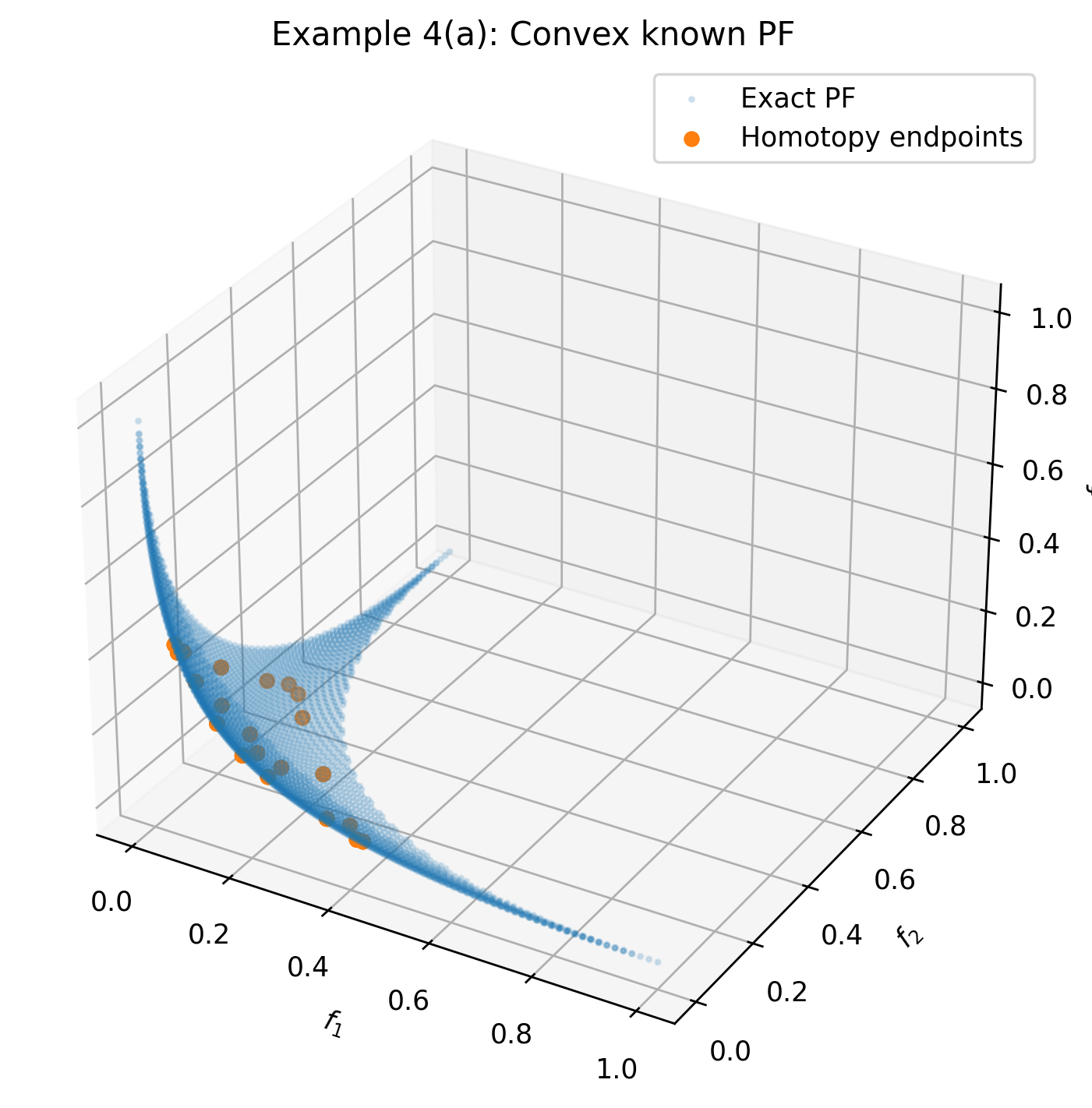}
				\caption{Convex Pareto front.}
				\label{fig:example4_convex_front}
			\end{subfigure}
			\hfill
			\begin{subfigure}[t]{0.48\textwidth}
				\centering
				\includegraphics[width=\linewidth]{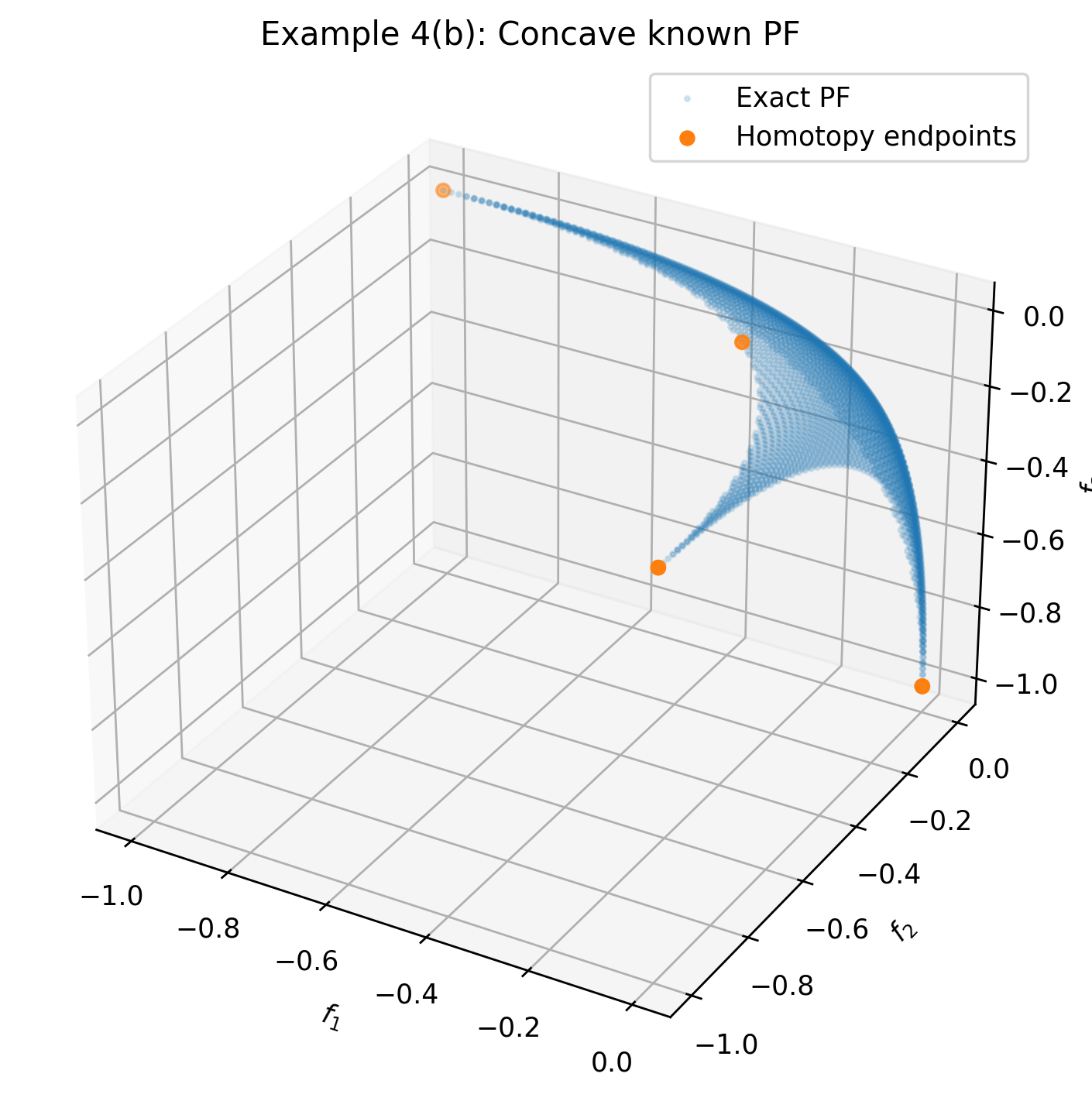}
				\caption{Concave Pareto front.}
				\label{fig:example4_concave_front}
			\end{subfigure}
			\caption{\add{Exact Pareto fronts and computed homotopy endpoints
					for the convex and concave three-objective problems.
					(a) For the convex problem, the analytical front provides
					a reference for assessing endpoint accuracy and coverage
					obtained from 21 prescribed weight vectors.
					(b) For the concave problem, homotopy endpoints obtained
					from generic initializations concentrate on boundary
					branches, leaving much of the interior Pareto surface
					unsampled.}}
			\label{fig:example4_known_fronts}
		\end{figure}
		
		\medskip
		\noindent\add{\textbf{(b) Nonconvex MOO problem.}}
		
		Using the same feasible set $\Omega$, consider
		\begin{equation}
			\label{eq:example4_concave}
			\min_{\bm{x}\in\Omega}
			F^-(\bm{x})
			=
			\bigl(-x_1^2,\;-x_2^2,\;-(1-x_1-x_2)^2\bigr).
		\end{equation}
		Again, every feasible point is Pareto optimal. If
		$F^-(\bm{x})\leq F^-(\widehat{\bm{x}})$ componentwise,
		then the first two objective inequalities imply
		$x_1\geq\widehat{x}_1$ and
		$x_2\geq\widehat{x}_2$, while the third implies
		$x_1+x_2\leq\widehat{x}_1+\widehat{x}_2$.
		Therefore, $\bm{x}=\widehat{\bm{x}}$, and the Pareto set
		remains $\Omega$.
		
		Since $\sqrt{-f_1}=x_1$, $\sqrt{-f_2}=x_2$, and
		$\sqrt{-f_3}=1-x_1-x_2$, the exact Pareto front is
		\begin{equation}
			\label{eq:example4_concave_pf}
			\sqrt{-f_1}+\sqrt{-f_2}+\sqrt{-f_3}=1,
			\qquad f_i\leq0.
		\end{equation}
		The interior Pareto-criticality conditions again yield the
		analytical solution \eqref{eq:example4_exact_kkt}. However,
		the Hessian of the corresponding positively weighted objective is
		\[
		\nabla_{\bm{x}}^2
		\left(\sum_{i=1}^3\lambda_i f_i\right)
		=
		-2\begin{pmatrix}
			\lambda_1+\lambda_3 & \lambda_3\\
			\lambda_3 & \lambda_2+\lambda_3
		\end{pmatrix},
		\]
		which is negative definite for every
		$\bm{\lambda}\in W^{++}$.
		Thus, the analytical interior Pareto-critical point is a
		strict maximum of the corresponding weighted objective
		and cannot be recovered as a weighted-sum minimizer.
		
		Nevertheless, this point belongs to the exact Pareto set
		and satisfies the first-order Pareto-criticality conditions.
		Consequently, the nonconvex problem provides a useful
		test of whether the proposed homotopy method can recover
		interior Pareto-critical solutions that are inaccessible
		to positively weighted-sum minimization.
		
		Figure~\ref{fig:example4_concave_front} illustrates the concentration of the generic-start endpoints on boundary branches.
		
		
		\medskip
		\noindent\textbf{\add{Initialization Sensitivity and Branch Selection.}}
		
		This distinction is important because the positive weights in the
		present homotopy formulation enter the Pareto-criticality/KKT system;
		the proposed method does not independently minimize a weighted sum for
		each weight. Hence unsupported Pareto-critical points are not excluded
		from the terminal homotopy system.
		
		Using the same 21 weights and four generic initial points as in part
		(a), all 84 continuation paths converged to feasible KKT points located
		on the boundary of the feasible triangle. The resulting
		\[
		\mathrm{GD}=4\times10^{-6}
		\]
		shows that the computed points lie essentially on the exact Pareto
		front, whereas
		\[
		\mathrm{IGD}=0.295819
		\]
		indicates poor coverage of the complete Pareto surface. Thus, the principal limitation in this experiment is not the proximity of the computed endpoints to the exact Pareto front, but their concentration on boundary branches, resulting in poor coverage of the interior Pareto surface.

		To examine this behavior more closely, five representative positive
		weight vectors were selected:
		\[
		\begin{split}
			&(1/3,1/3,1/3),\\
			&(0.25,0.375,0.375),\qquad
			(0.375,0.25,0.375),\\
			&(0.375,0.375,0.25),\qquad
			(0.25,0.25,0.50).
		\end{split}
		\]
		For each weight, the analytical interior Pareto-critical point
		$\mathbf{x}^*(\boldsymbol{\lambda})$ was first used as the initial
		point. Algorithm~2 recovered the corresponding interior
		Pareto-critical point for all five weights, with objective-space errors
		on the order of machine precision.
		
		Figure~\ref{fig:example4_concave_branches} shows the decision-space branch diagnostics for these five weights.
		
		\begin{figure}[h]
			\centering
			\begin{subfigure}[t]{.33\textwidth}
				\centering
				\includegraphics[width=\linewidth]{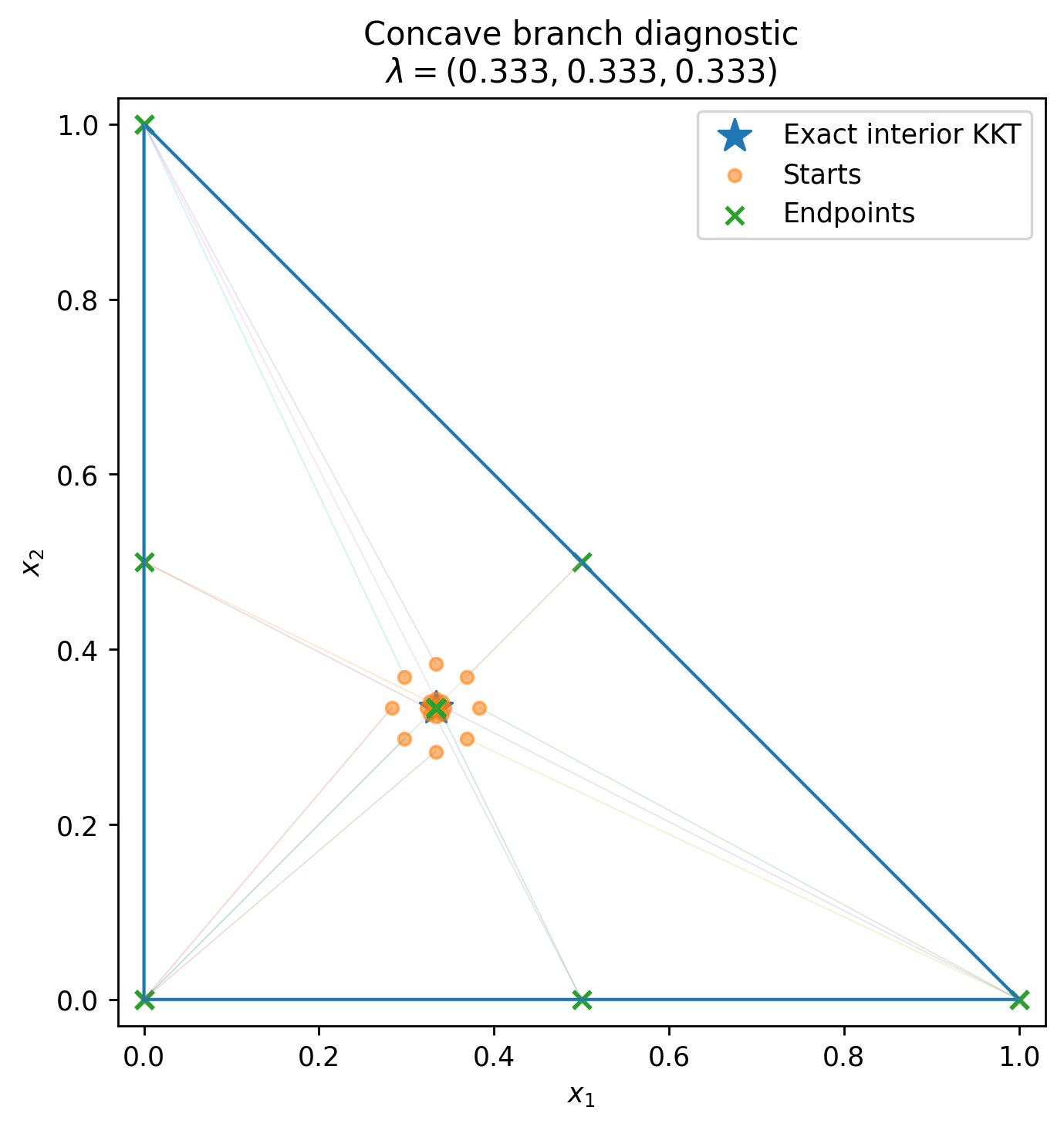}
				\caption{$\bm\lambda=(1/3,1/3,1/3)$.}
			\end{subfigure}\hfill
			\begin{subfigure}[t]{.33\textwidth}
				\centering
				\includegraphics[width=\linewidth]{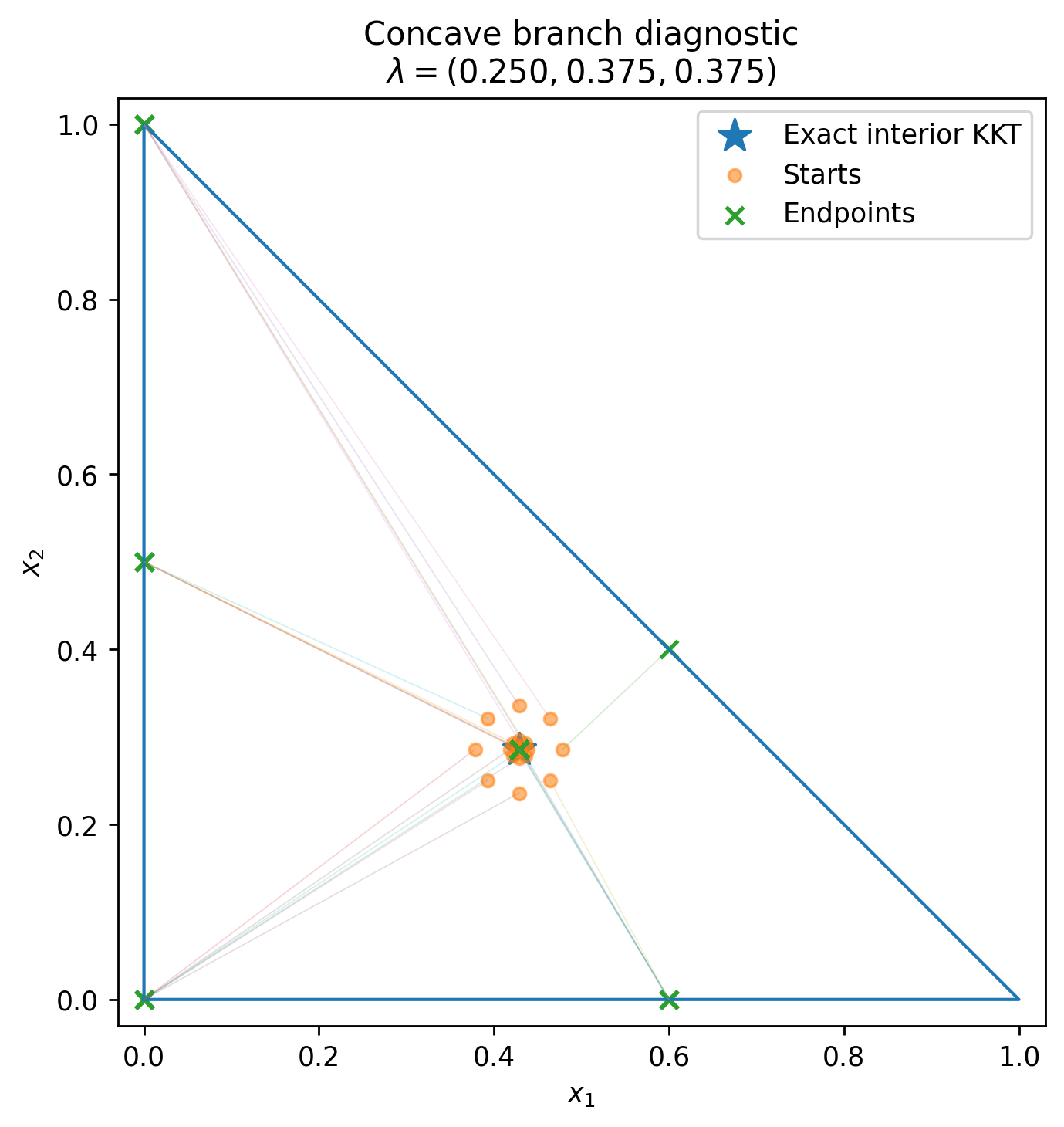}
				\caption{$\bm\lambda=(0.25,0.375,0.375)$.}
			\end{subfigure}
			\medskip
			\begin{subfigure}[t]{.33\textwidth}
				\centering
				\includegraphics[width=\linewidth]{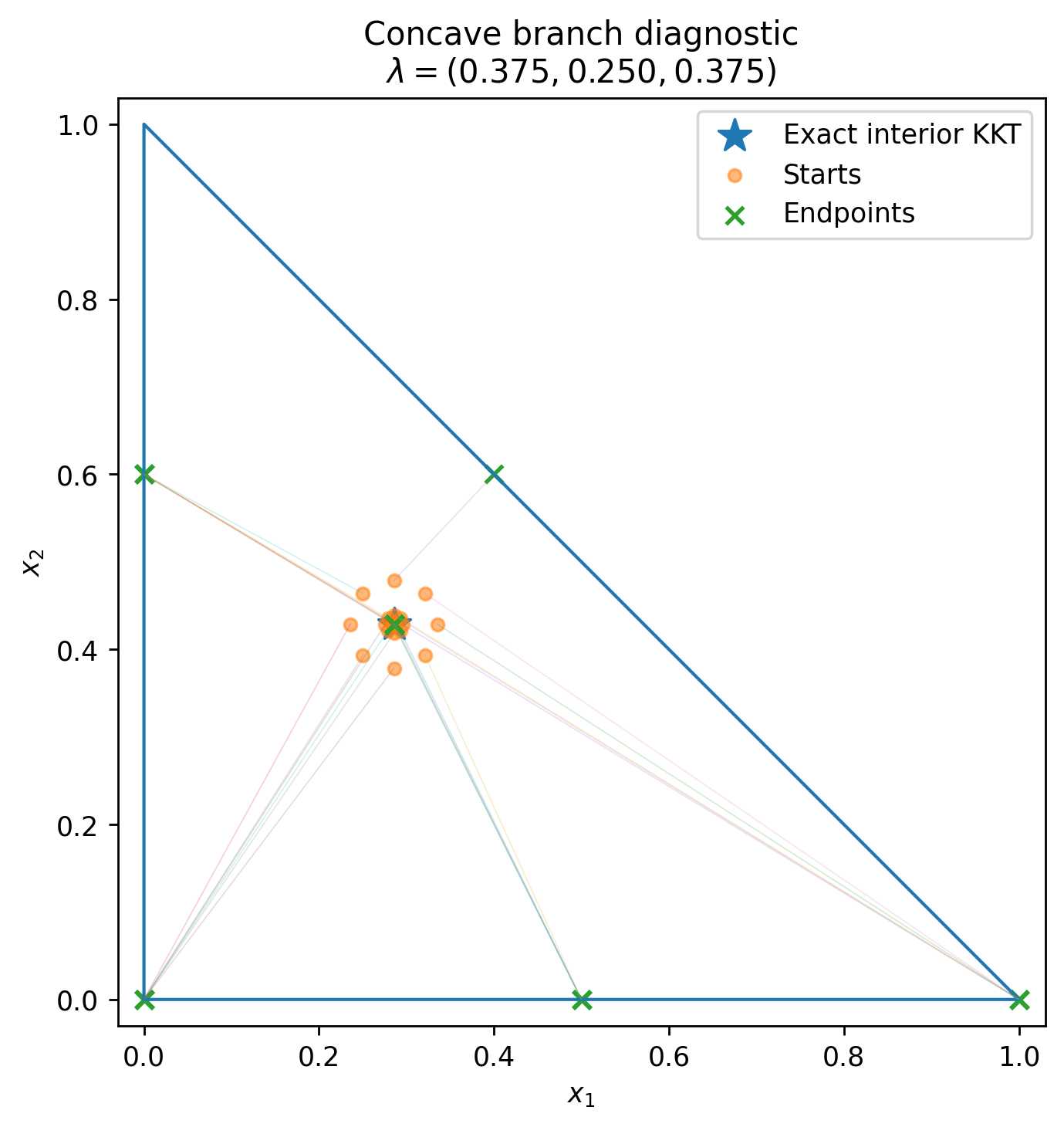}
				\caption{$\bm\lambda=(0.375,0.25,0.375)$.}
			\end{subfigure}\hfill
			\begin{subfigure}[t]{.33\textwidth}
				\centering
				\includegraphics[width=\linewidth]{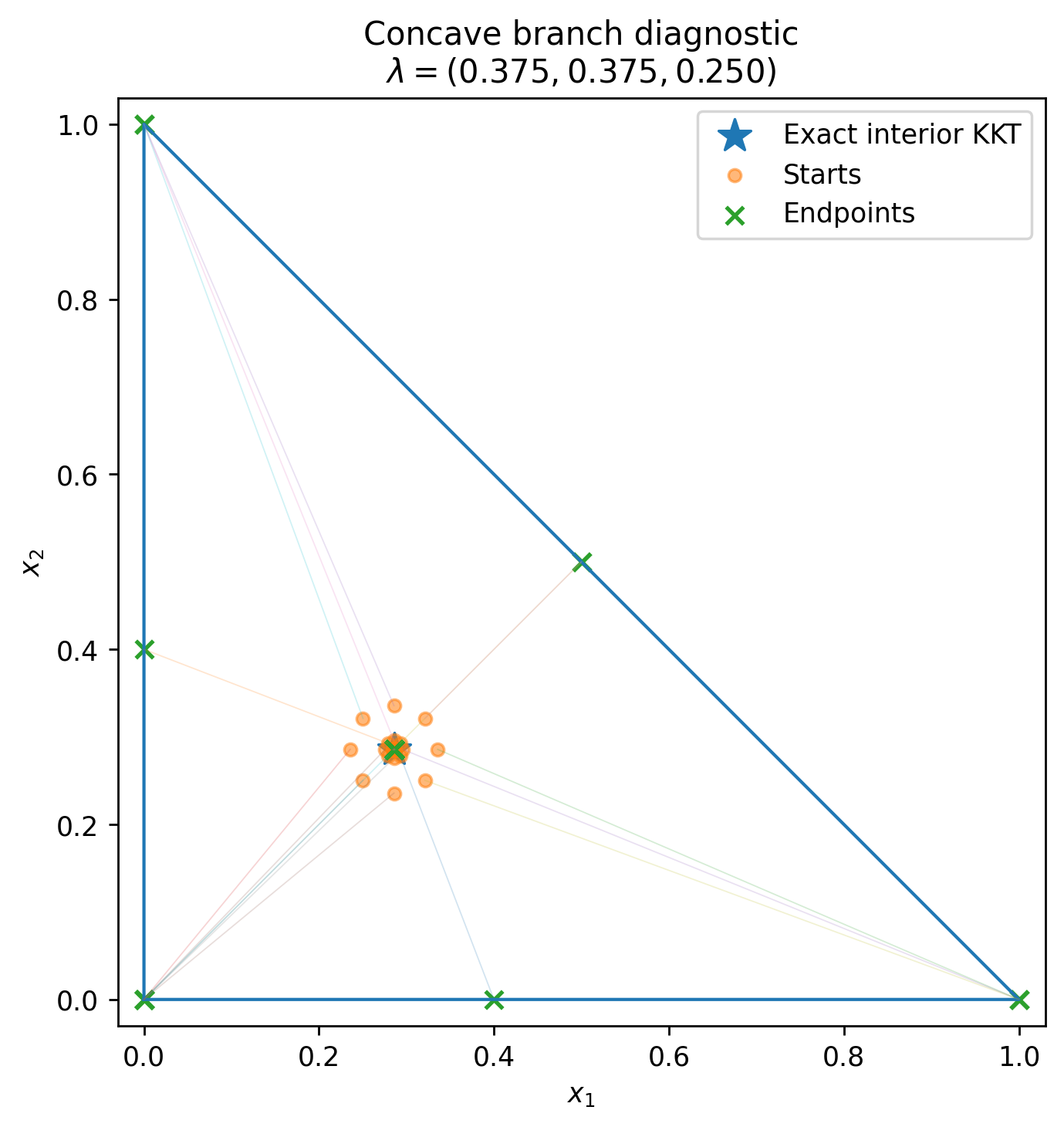}
				\caption{$\bm\lambda=(0.375,0.375,0.25)$.}
			\end{subfigure}
			\medskip
			\begin{subfigure}[t]{.33\textwidth}
				\centering
				\includegraphics[width=\linewidth]{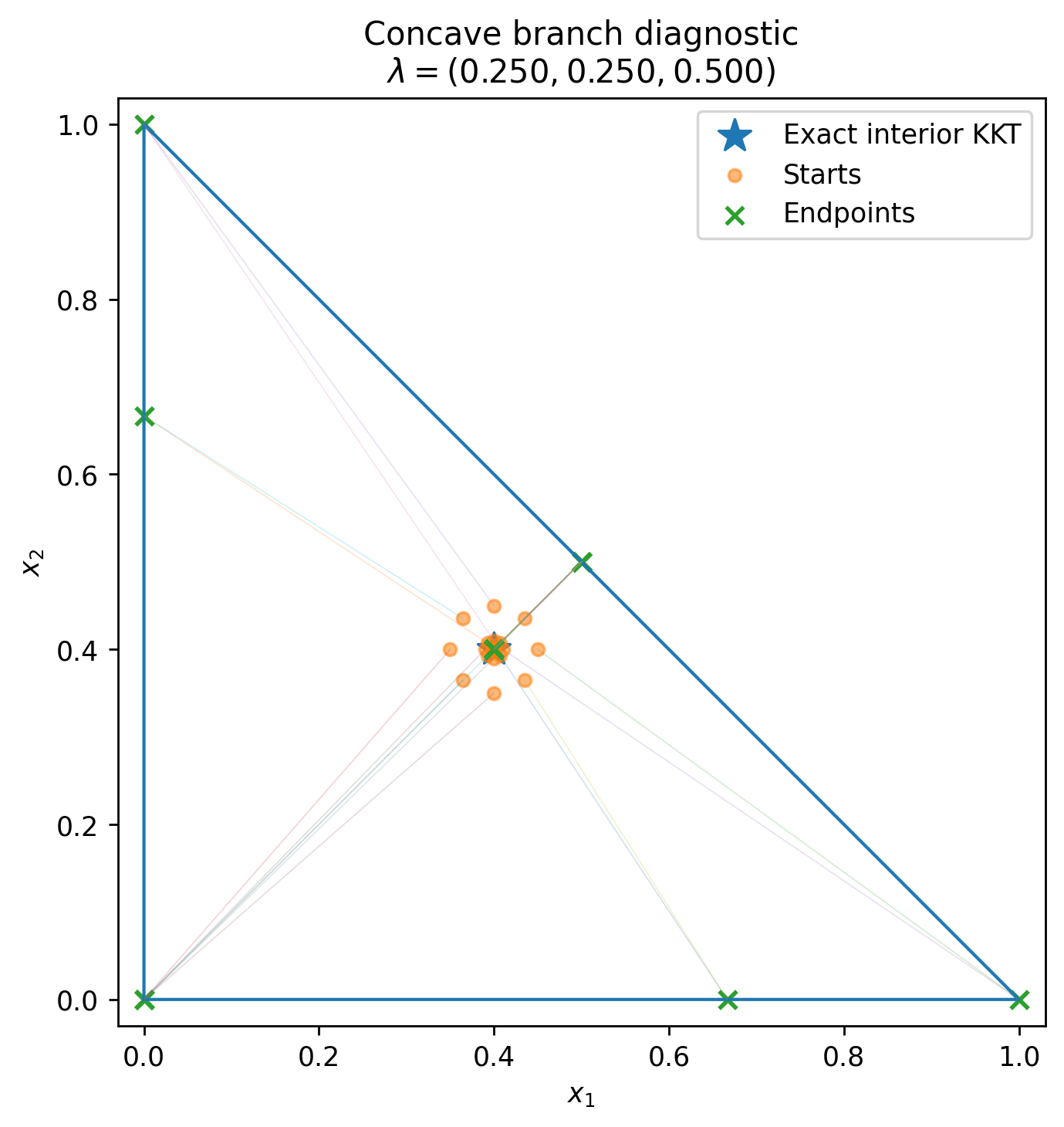}
				\caption{$\bm\lambda=(0.25,0.25,0.50)$.}
			\end{subfigure}
			\caption{\add{Decision-space branch diagnostics for five positive weight vectors in the concave problem. The panels illustrate the dependence of continuation endpoints on initialization, including recovery of the analytical interior stationary point and convergence to boundary branches.}}
			\label{fig:example4_concave_branches}
		\end{figure}
		

		\begin{table}[htbp]
			\centering
			\caption{\add{Numerical comparison of the convex and concave
					three-objective test problems. Each experiment uses 21 weight
					vectors and four initializations per weight.}}
			\label{tab:example4_comparison}
			\small
			\begin{tabular}{lcc}
				\toprule
				Quantity & Convex & Concave \\
				\midrule
				Number of continuation paths & 84 & 84 \\
				Successful and feasible endpoints & 84 & 84 \\
				Near analytical interior KKT point & 84 & 0 \\
				Boundary endpoints & 0 & 84 \\
				Median distance to interior reference & $9.0\times10^{-6}$ & $1.0387$ \\
				Maximum distance to interior reference & $1.14\times10^{-4}$ & $1.0544$ \\
				Maximum stationarity residual
				& $4.7\times10^{-5}$ & $6.8\times10^{-5}$ \\
				Maximum homotopy-system norm
				& $1.05\times10^{-8}$ & $2.94\times10^{-8}$ \\
				GD $\downarrow$ & $0.003195$ & $0.000004$ \\
				IGD $\downarrow$ & $0.108137$ & $0.295819$ \\
				Total runtime (s) & $37.48$ & $37.27$ \\
				\bottomrule
			\end{tabular}
		\end{table}
		
		\begin{table}[h]
			\centering
			\caption{\add{Initialization sensitivity of the concave
					three-objective problem in Example~4. Endpoints are
					classified according to the interior and boundary
					KKT branches.}}
			\label{tab:example4_initialization}
			\small
			\begin{tabular}{cccc}
				\toprule
				Perturbation $\varepsilon$ & Paths
				& Near interior KKT & Boundary endpoints \\
				\midrule
				$0$             & 5  & 5  & 0 \\
				$10^{-6}$       & 40 & 40 & 0 \\
				$10^{-5}$       & 40 & 40 & 0 \\
				$10^{-4}$       & 40 & 37 & 3 \\
				$10^{-3}$       & 40 & 22 & 18 \\
				$10^{-2}$       & 40 & 0  & 40 \\
				$5\times10^{-2}$ & 40 & 0  & 40 \\
				\midrule
				Total & 245 & 144 & 101 \\
				\bottomrule
			\end{tabular}
		\end{table}
		
		\add{Table~\ref{tab:example4_comparison} compares the numerical performance of the proposed homotopy method for the convex and concave three-objective test problems. In both cases, all 84 continuation paths reached feasible endpoints. For the convex problem, all endpoints were close to their corresponding analytical interior KKT solutions, with a maximum objective-space error of approximately $1.14\times10^{-4}$. The small GD and moderate IGD indicate accurate approximation of the sampled portion of the exact Pareto front, although the finite set of prescribed weights does not provide uniform coverage of the entire front.}
		
		The concave problem exhibits substantially different behavior. Although all 84 paths converged to feasible endpoints with small homotopy-system and stationarity residuals, none approached the analytical interior KKT point associated with its prescribed weight. Instead, all endpoints were classified as boundary solutions. The very small GD indicates that these endpoints lie close to the exact Pareto front, whereas the larger IGD reflects incomplete coverage. Thus, numerical convergence and Pareto-front proximity do not, by themselves, establish convergence to a particular stationary branch.
		
		\add{To investigate the sensitivity of interior-branch
			recovery to initialization, the analytical interior
			KKT points were perturbed by different magnitudes.
			Table~\ref{tab:example4_initialization} summarizes
			the resulting branch classifications. For perturbations of magnitude $10^{-5}$ or smaller, all trajectories returned near the interior stationary points. As the perturbation increased, an increasing proportion of trajectories reached boundary branches. For perturbations of $10^{-2}$ and $5\times10^{-2}$, all trajectories reached boundary endpoints rather than the targeted interior stationary points. All 245 continuation runs were reported successful, and 243 satisfied the prescribed feasibility tolerance.}
		
		These results demonstrate that the homotopy formulation can trace stationary branches associated with a nonconvex multiobjective problem, but the branch reached depends strongly on initialization. In particular, the analytical interior stationary points of the concave weighted objectives are not weighted-sum minima. The experiments therefore distinguish convergence to a KKT point from recovery of a specified stationary branch and from coverage of the complete Pareto front.
		

		Taken together, parts (a) and (b) also illustrate that the proposed
		continuation method is directly applicable to problems with more than
		two objectives. For each selected weight and initialization,
		Algorithm~2 follows a one-dimensional homotopy path in the enlarged
		continuation space toward a terminal Pareto-critical point. The fact
		that the three-objective Pareto front is two-dimensional does not
		require an individual continuation path to be two-dimensional.
		
		\medskip
		\noindent\textbf{\add{Comparison of the Convex and Nonconvex Problems.}}
		
		The two problems share the same feasible region and exact
		Pareto set,
		\[
		\mathcal{P}^{+}=\mathcal{P}^{-}=\Omega,
		\]
		but exhibit fundamentally different weighted-sum optimization
		properties. For the convex problem, the positively weighted
		objective has a positive-definite Hessian, and the analytical
		interior Pareto-critical point is its unique global minimizer.
		In contrast, for the nonconvex problem, the corresponding
		Hessian is negative definite, so the same analytical
		interior point is a strict maximum of the weighted objective.
		Although this point remains Pareto optimal, it cannot be
		recovered by positive weighted-sum minimization.
		
		The numerical experiments further demonstrate this
		distinction. In the convex case, the proposed homotopy method
		reliably recovers the analytical interior Pareto-critical
		solutions from the tested initial points. In the nonconvex
		case, however, the continuation endpoints exhibit strong
		sensitivity to initialization, with many paths terminating
		on boundary Pareto-critical branches rather than the
		analytical interior branch. These results illustrate both
		the ability of the proposed method to recover Pareto-critical
		solutions beyond weighted-sum minimizers and the importance
		of initialization and branch selection in nonconvex
		multiobjective optimization.
		
	\end{example}
	\section{Conclusion and Future Work} \label{sec:conclusion}
	This work developed a homotopy continuation framework for constrained MOO problems by embedding the Karush--Kuhn--Tucker (KKT) conditions into a parameterized continuation structure. The proposed homotopy map continuously deforms an auxiliary system into the KKT system of the original MOO problem, yielding a deterministic and theoretically grounded continuation framework for computing Pareto-critical solutions. 
	\del{Under suitable regularity assumptions,
		global convergence of the continuation
		trajectories was established.}
	\add{Under the stated regularity assumptions,
		the theoretical analysis establishes the
		existence of a smooth continuation path
		with a nonempty set of accumulation points
		satisfying the KKT conditions as the
		homotopy parameter approaches zero.}

	The proposed framework was investigated through benchmark optimization problems, disconnected nonconvex many-objective examples, and a multiobjective IEEE 5-bus optimal power flow problem. The numerical experiments demonstrated that the continuation approach can successfully compute feasible Pareto-critical \add{solutions with good} agreement with classical scalarization methods and competitive performance relative to \add{the evolutionary algorithms considered in this study}. 
	\del{In addition, the framework exhibited robust
		behavior for disconnected, degenerate, and
		nonconvex Pareto-critical structures through
		adaptive continuation and branch-following
		strategies.}
	\add{The experiments also illustrate the ability of
		adaptive continuation and multiple initializations
		to explore nonconvex, degenerate, and disconnected
		Pareto-critical structures, although complete
		coverage of these structures is not guaranteed.}
	
	\add{The three-objective problems with analytically
		known Pareto fronts provide further insight into
		the behavior of the proposed method. In particular,
		the homotopy formulation can recover interior
		Pareto-critical points that are not minimizers of
		any positively weighted-sum problem. However, the
		numerical results reveal substantial sensitivity
		to initialization in the nonconvex case, where
		continuation paths may converge to boundary KKT
		branches rather than the targeted interior branch.
		These findings highlight the distinction between
		convergence to a Pareto-critical point and adequate
		coverage of the Pareto front.}
	
	\add{The numerical experiments also identify} practical challenges associated with continuation methods, including step-size control, path stability near singular regions, and initialization sensitivity. \add{Adaptive predictor--corrector strategies were incorporated to improve numerical stability, but
		reliable branch selection and systematic coverage
		remain important areas for further investigation.}
	
	Several directions for future research remain. \add{On the theoretical side, extending the analysis
		to weaker regularity assumptions and developing
		more refined convergence and branch-selection
		results would strengthen the framework. On the computational side,} extending the method to large-scale and high-dimensional MOO problems is of particular interest, especially when Jacobian evaluations and linear system solves dominate the computational cost. 
	
	\del{Additional improvements may be obtained through advanced continuation techniques, parallel path tracking, and machine-learning-assisted parameter selection. Further theoretical developments involving weaker regularity assumptions and refined convergence analysis would also broaden the applicability of the proposed framework.}
	\add{Potential improvements include more efficient
		predictor--corrector schemes, systematic
		initialization and branch-exploration strategies,
		parallel continuation, and adaptive parameter
		selection. Further numerical studies will also
		compare the proposed method with recent constrained
		and many-objective evolutionary algorithms,
		including implementations available through
		pymoo, jMetal, and PlatEMO.}

	\del{Overall, the present study demonstrates
		that homotopy continuation provides an
		effective and flexible deterministic
		framework for constrained MOO, particularly
		for problems involving nonconvexity,
		degeneracy, and disconnected Pareto-critical
		structures.}
	\add{Overall, the theoretical and numerical results
		support homotopy continuation as a flexible
		deterministic framework for computing Pareto-critical
		solutions of constrained MOO problems, including
		selected problems with nonconvex, degenerate, and
		disconnected solution structures. The findings
		also emphasize that convergence to KKT points,
		recovery of particular solution branches, and
		coverage of the Pareto front are distinct objectives
		that require separate theoretical and numerical
		consideration.}
	\section*{Acknowledgements}
	The authors would like to thank the anonymous reviewers for their careful reading of the manuscript and their valuable comments and suggestions, which significantly improved the clarity, organization, and quality of this work.
	
	\appendix  
	\begin{appendices}
		\section*{Appendices}
		\addcontentsline{toc}{section}{Appendices}
		The following appendices summarize the classical scalarization and evolutionary algorithms used for comparison in the numerical experiments.
		\section{Classical Scalarization Methods}
		\label{appendix:A}
		\subsection{The Weighted Sum Method (WSM)}
		The Weighted Sum Method (WSM) is arguably the most straightforward scalarization technique, transforming MOO into a single objective function by taking a linear convex combination of the objectives \citet{marler2004survey, van2012evaluation,gass1955computational}. 
		Given the $p$ objectives $f_i(\bm{x})$, the WSM formulation is defined by minimizing the aggregated objective function $F(\bm{x},\bb{w})$:
		\begin{mini!}{\bm{x} \in \Omega}{F(\bm{x},\bb{w}) = \sum_{i=1}^{p} w_i f_i(\bm{x})}{}{}\label{wsm}\tag{5}
		\end{mini!}
		The weight vector $\mathbf{w} = [w_1, w_2, \dots, w_p]^T$ is chosen by the decision-maker, where $w_i \geq 0$ for all $i$ and typically $\sum_{i=1}^{p} w_i = 1$. When all weights are strictly positive, any minimizer of~\eqref{wsm} is Pareto optimal. By varying the weight vector, different Pareto-optimal solutions can be generated; however, the Weighted Sum Method may fail to identify Pareto-optimal solutions located on nonconvex regions of the Pareto front. The resulting single-objective optimization problem is solved using the Sequential Least Squares Programming (SLSQP) algorithm\,\citet{nocedal2006numerical}.
		\begin{algorithm}[h!]
			\caption{Weighted Sum Method (WSM) for Pareto Front Generation}
			\label{alg:wsm}
			\begin{algorithmic}[1]
				\item[]  \hspace{-2em} \textbf{Require:} Objective functions $\{f_i(\bm{x})\}_{i=1}^p$, feasible set $\Omega$, 
				weight set $\mathcal{W} \subset \{\bm{w} \in \mathbb{R}^p \mid w_i \ge 0, \ \sum_{i=1}^p w_i = 1\}$
				\item[] \hspace{-2em} \textbf{Ensure:} Approximate Pareto set $\mathcal{P}_{\text{WSM}}$ and Pareto front $\mathcal{F}_{\text{WSM}}$
				\STATE $\mathcal{P}_{\text{WSM}} \gets \emptyset$ \{\emph{Set of decision vectors}\}
				\STATE $\mathcal{F}_{\text{WSM}} \gets \emptyset$ \{\emph{Set of objective vectors}\}
				\FOR{$\bm{w} \in \mathcal{W}$}
				\STATE Define scalarized objective:
				$F(\bm{x},\bm{w}) = \sum_{i=1}^{p} w_i f_i(\bm{x})$
				\STATE Solve the single-objective problem
				$\min_{\bm{x} \in \Omega} F(\bm{x},\bm{w})$
				and obtain a solution $\bm{x}^{\star}(\bm{w})$
				\STATE Compute objective vector 
				$\bm{f}^{\star}(\bm{w}) = \bigl(f_1(\bm{x}^{\star}(\bm{w})),\dots,f_p(\bm{x}^{\star}(\bm{w}))\bigr)$
				\STATE Add solution to candidate sets:
				$\mathcal{P}_{\text{WSM}} \gets \mathcal{P}_{\text{WSM}} \cup \{\bm{x}^{\star}(\bm{w})\}, \quad
				\mathcal{F}_{\text{WSM}} \gets \mathcal{F}_{\text{WSM}} \cup \{\bm{f}^{\star}(\bm{w})\}$
				\ENDFOR
				\STATE \textbf{return} $\mathcal{P}_{\text{WSM}}, \mathcal{F}_{\text{WSM}}$
			\end{algorithmic}
		\end{algorithm}
		\subsection{The $\epsilon$-Constraints Method (ECM)}
		The $\epsilon$-Constraints Method (ECM) overcomes the convexity limitation of the WSM by maintaining one objective function as the primary target for minimization while converting all remaining $p-1$ objectives into inequality constraints\,\citet{haimes1971bicriterion}.
		Assuming objective $f_m(\bm{x})$ is selected as the primary objective for minimization, the ECM is formulated as a single-objective problem subject to bounds defined by $\epsilon$:
		\begin{mini!}{\bm{x} \in \Omega}{f_m(\bm{x})}{\label{eq:epsilon-constraint}}{}\tag{6}
			\addConstraint{f_k(\bm{x}) \leq \epsilon_k, \quad \text{for } k=1, \dots, p, \quad k \neq m}{}.\tag{7}
		\end{mini!}
		Here, $\epsilon_k$ represents the maximum acceptable value for the constrained objective $f_k$. It can  be shown that a solution of \eqref{eq:epsilon-constraint} is weakly Pareto optimal. Furthermore, a solution $\bm{x}^*$ is Pareto optimal if and only if it is a solution of \eqref{eq:epsilon-constraint} for every $m = 1, \cdots, p$, where $\epsilon_k = f_k(\bm{x}^*)$ for $k=1, \cdots, p, k\neq m$. By iteratively minimizing $f_m(\mathbf{x})$ over a range of systematically varied $\epsilon_k$ values, the full Pareto front can be traced, including non-convex sections.
		\begin{algorithm}[h!]
			\caption{$\epsilon$-Constraint Method (ECM) for Pareto Front Generation}
			\label{alg:ecm}
			\begin{algorithmic}[1]
				\item[]  \hspace{-2em} \textbf{Require:} Objective functions $\{f_i(\bm{x})\}_{i=1}^p$, feasible set $\Omega$, index $m$ of the primary objective $f_m$, \\
				\hspace{2.8em} grid of $\epsilon$-vectors $\{\boldsymbol{\epsilon}^{(j)}\}_{j=1}^{N}$ 
				where 
				$\boldsymbol{\epsilon}^{(j)} = (\epsilon_1^{(j)},\dots,\epsilon_p^{(j)})$ 
				and $\epsilon_m^{(j)}$ is unused.
				\item[] \hspace{-2em} \textbf{Ensure:} Approximate Pareto set $\mathcal{P}_{\epsilon}$ and Pareto front $\mathcal{F}_{\epsilon}$
				\STATE $\mathcal{P}_{\epsilon} \gets \emptyset$ \{\emph{Set of decision vectors}\}
				\STATE $\mathcal{F}_{\epsilon} \gets \emptyset$ \{\emph{Set of objective vectors}\}
				\FOR{each $\boldsymbol{\epsilon}^{(j)}$ in the grid}
				\STATE Define constrained problem:
				$\min_{\bm{x} \in \Omega} \quad f_m(\bm{x})$                s.t. $f_k(\bm{x}) \le \epsilon_k^{(j)}, \quad k = 1,\dots,p, \ k \neq m$
				\STATE Solve the above single-objective problem and obtain a solution $\bm{x}^{\star}(\bm{\epsilon}^{(j)})$
				\STATE Compute objective vector 
				$\bm{f}^{\star}(\bm{\epsilon}^{(j)}) = \bigl(f_1(\bm{x}^{\star}),\dots,f_p(\bm{x}^{\star})\bigr)$
				\STATE Update candidate sets:
				$\mathcal{P}_{\epsilon} \gets \mathcal{P}_{\epsilon} \cup \{\bm{x}^{\star}(\bm{\epsilon}^{(j)})\}, \quad
				\mathcal{F}_{\epsilon} \gets \mathcal{F}_{\epsilon} \cup \{\bm{f}^{\star}(\bm{\epsilon}^{(j)})\}$
				\ENDFOR
				\STATE \textbf{return} $\mathcal{P}_{\epsilon}, \mathcal{F}_{\epsilon}$
			\end{algorithmic}
		\end{algorithm}
		\subsection{The Global Criterion Method (GCM)}
		
		\del{The Global Criterion Method (GCM), also known as the
			Goal Programming or Compromise Programming method,
			transforms MOO into minimizing the distance from an
			unattainable ``Ideal Point'' ($\bm{f}^*$).
			The ideal point $\bm{f}^*$ is found by individually
			optimizing each objective function $f_i(\bm{x})$.}
		\add{The Global Criterion Method (GCM) used here is
			a particular form of the Method of Weighted Metrics,
			also known as compromise programming
			\citet{miettinen1999nonlinear}.
			It minimizes the distance between the objective
			vector and the ideal point $\bm{f}^*$, whose
			components are obtained by individually minimizing
			each objective function over the feasible set.}
		The GCM minimizes the deviation from the ideal point $\bm{f}^*$ using an $L_p$-norm distance metric:
		\begin{mini!}{\bm{x} \in \Omega}{D_p(\bm{x}) = \left( \sum_{i=1}^{M} \left| \frac{f_i(\bm{x}) - f_i^*}{s_i} \right|^q \right)^{1/q}}{}{}\tag{8}
		\end{mini!}
		where $f_i^*$ is the ideal value (minimum) for objective $i$, and $s_i$ is a normalization term, typically the range of the objective function ($s_i = f_i^{\max} - f_i^{\min}$) \citet{marler2004survey}. The use of $p=2$, corresponding to the $L_2$-norm (or Euclidean norm), is the most common approach, minimizing the geometric distance in the objective space. 
		\del{Despite its popularity and wide range
			of applications, there is no guarantee
			that it provides a Pareto optimal solution.}
		\add{For strictly positive scaling factors $s_i$ and a
			finite norm exponent $p\geq 1$, every global minimizer
			of this ideal-point scalarization is Pareto optimal,
			provided that $f_i^*$ is a lower bound on $f_i(\bm{x})$
			over the feasible set for every objective
			\citet{miettinen1999nonlinear}.
			This guarantee applies to global minimizers of the
			scalarized problem and does not necessarily extend
			to local stationary points returned by a numerical
			optimization solver.}

		\begin{algorithm}[H]
			\caption{Global Criterion Method (GCM) Based on Ideal Point Deviation}
			\label{alg:gcm}
			\begin{algorithmic}[1]
				\item[] \hspace{-2em} \textbf{Require:} Objective functions $\{f_i(\bm{x})\}_{i=1}^p$, feasible set $\Omega$, norm parameter $\add{q} \ge 1$
				\item[] \hspace{-2em} \textbf{Ensure:} Approximate Pareto solution $\bm{x}^{\star}$ and objective vector $\bm{f}^{\star}$
				\item[] \hspace{-1.6em} \textbf{Compute Ideal Values}    
				\FOR{$i = 1,\dots,p$}
				\STATE Solve
				$f_i^* = \min_{\bm{x} \in \Omega} f_i(\bm{x})$
				and store the solution $\bm{x}_i^*$
				\ENDFOR
				\STATE Estimate scaling factors $s_i$ (e.g., $s_i = f_i^{\max} - f_i^{\min}$ using bounds or sampled values)
				\item[] \hspace{-1.6em} \textbf{Define Global Criterion}
				\STATE Define the scalar deviation measure
				$D_p(\bm{x}) = 
				\left( 
				\sum_{i=1}^{p} 
				\left| 
				\frac{f_i(\bm{x}) - f_i^*}{s_i} 
				\right|^q 
				\right)^{1/q}$
				\item[] \hspace{-1.6em} \textbf{Solve the Global Criterion Problem}
				\STATE Solve
				$\min_{\bm{x} \in \Omega} D_p(\bm{x})$
				to obtain $\bm{x}^{\star}$
				\STATE Compute the corresponding objective vector
				$\bm{f}^{\star} = \bigl(f_1(\bm{x}^{\star}),\dots,f_p(\bm{x}^{\star})\bigr)$
				\STATE \textbf{return:} $\bm{x}^{\star},\; \bm{f}^{\star}$
			\end{algorithmic}
		\end{algorithm}
		\subsection{The Lexicographic Method}
		The Lexicographic Method is an \textit{a priori} approach characterized by the decision-maker assigning an absolute priority ranking to the objectives. This ranking dictates that the most important objective must be optimized first, followed sequentially by the next most important, subject to maintaining the optimal result of all higher-ranked objectives \citet{isermann1982linear, miettinen1999nonlinear, ehrgott2005multicriteria}.
		A lexicographic minimization problem defines an ordered set of functions $\bm{f}=[f_1,f_2,\dots,f_p]$, where $f_1$ is deemed the most important and $f_p$ the least important. The problem is concisely written as:
		\begin{mini!}{\bm{x} \in \Omega}{\text{lex min } [f_1(\bm{x}), f_2(\bm{x}), \dots, f_p(\bm{x})]}{}{}\tag{9}
		\end{mini!}
		This notation signifies that $f_1$ has overwhelming priority over $f_2$ ($f_1 \gg f_2$), and so forth down the hierarchy.
		The solution methodology for the lexicographic problem involves solving a strict sequence of $p$ single-objective optimization problems:
		\begin{enumerate}[label=\textbf{Step \arabic*:}, leftmargin=*]
			\item \textbf{$k=1$:} Solve the unconstrained minimization problem for the most important objective:
			$$f_1^* = \min_{\bm{x} \in \Omega} f_1(\bm{x})$$
			\item \textbf{$k=2$:} Solve the problem for the second objective, constrained by the optimal value obtained in the previous step. Define a new feasible set $\Omega_2$:
			$$\Omega_2 = \{ \bm{x} \in \Omega \mid f_1(\bm{x}) = f_1^* \}$$
			The problem becomes: $\displaystyle f_2^* = \min_{\bm{x} \in \Omega_2} f_2(\bm{x})$.
			\item \textbf{$k=3, \dots, p$:} Continue this iterative process. At step $k$, the feasible set $\Omega_k$ is constrained by the optimal values $f_1^*, f_2^*, \dots, f_{k-1}^*$, ensuring the previous optimalities are maintained as hard equality constraints.
		\end{enumerate}
		If any sequential problem is infeasible or unbounded, the process halts, and no solution exists for the full MOO under the specified priority structure.
		\begin{algorithm}[H]
			\caption{Lexicographic Sequential Solution Algorithm}
			\begin{algorithmic}[1]
				\STATE \textbf{Input:} Objective functions $\{f_1, f_2, \dots, f_p\}$, feasible region $\Omega$
				\STATE \textbf{Initialize:} Set $k \gets 1$, $\Omega_1 \gets \Omega$
				\WHILE{$k \leq p$}
				\STATE \textbf{Step 1: Solve the $k^{th}$ Optimization Problem}
				\STATE \hspace{1em} Find the optimal value: 
				\STATE \hspace{2em} $\displaystyle f_k^* = \min_{\bm{x} \in \Omega_k} f_k(\bm{x})$
				\STATE \textbf{Step 2: Update Feasible Set}
				\STATE \hspace{1em} Define the next feasible region:
				\STATE \hspace{2em} $\Omega_{k+1} = \{ \bm{x} \in \Omega_k \mid f_k(\bm{x}) = f_k^* \}$
				\STATE \textbf{Step 3: Check Feasibility}
				\STATE \hspace{1em} If $\Omega_{k+1}$ is empty or unbounded:
				\STATE \hspace{2em} \textbf{Terminate:} No feasible solution exists under the current priority structure.
				\STATE \textbf{Step 4: Increment Priority Level}
				\STATE \hspace{1em} $k \gets k + 1$
				\ENDWHILE
				\STATE \textbf{Step 5: Termination}
				\STATE \hspace{1em} Output the solution $\bm{x}^*$ satisfying:
				\STATE \hspace{2em} $f_1(\bm{x}^*) = f_1^*, ~ f_2(\bm{x}^*) = f_2^*, ~ \dots, ~ f_p(\bm{x}^*) = f_p^*$
				\STATE \textbf{Output:} Lexicographically optimal solution $\bm{x}^*$
			\end{algorithmic}
		\end{algorithm}
		\section{Evolutionary Algorithm}
		\label{appendix:B}
		\subsection{The Non-Dominated Sorting Genetic Algorithm II (NSGA-II)}
		The Non-Dominated Sorting Genetic Algorithm II (NSGA-II) is a fast elitist multiobjective evolutionary algorithm introduced by Deb et al.~\citet{deb2002fast}. The Non-Dominated Sorting Genetic Algorithm II (NSGA-II) is an advanced, fast elitist Multi-Objective Genetic Algorithm (MOGA). It was developed to overcome the computational complexity and lack of effective elitism found in earlier non-domination-based evolutionary algorithms. NSGA-II is a highly robust meta-heuristic approach often applied to complex, large-scale problems such as optimizing resource allocation in supply chains, where it often yields better operational outcomes than traditional methods\,\citet{acerce2025application}. 
		NSGA-II utilizes an elitist approach, ensuring that the best solutions found across generations are retained. The algorithm proceeds through a standard generational cycle:
		\begin{enumerate}[leftmargin=*]
			\item A parent population ($P_t$) generates an offspring population ($Q_t$) via genetic operators (selection, crossover, and mutation) \citet{acerce2025application}.
			\item The parent and offspring populations are combined into a larger population $R_t = P_t \cup Q_t$.
			\item $R_t$ is sorted based on non-domination into different Pareto fronts ($F_1, F_2, F_3, \dots$).
			\item The next generation population ($P_{t+1}$) is formed by selecting individuals sequentially, starting from the best front ($F_1$) and proceeding to subsequent fronts until the population limit $N$ is met.
		\end{enumerate}
		\begin{algorithm}[H]
			\caption{Non-dominated Sorting Genetic Algorithm II (NSGA-II)}
			\begin{algorithmic}[1]
				\STATE \textbf{Input:} Population size $N$, number of generations $T$, genetic operators (selection, crossover, mutation)
				\STATE \textbf{Initialize:} Generate an initial parent population $P_0$ of size $N$
				\FOR{$t = 0$ to $T-1$}
				\STATE \textbf{Step 1: Generate Offspring Population}
				\STATE \hspace{1em} Apply genetic operators (selection, crossover, and mutation) on $P_t$ to produce offspring $Q_t$
				\STATE \textbf{Step 2: Combine Parent and Offspring Populations}
				\STATE \hspace{1em} Form the combined population $R_t = P_t \cup Q_t$
				\STATE \textbf{Step 3: Non-dominated Sorting}
				\STATE \hspace{1em} Sort $R_t$ into Pareto fronts: $F_1, F_2, F_3, \dots$
				\STATE \textbf{Step 4: Create the Next Generation}
				\STATE \hspace{1em} Initialize $P_{t+1} = \emptyset$
				\STATE \hspace{1em} Sequentially add individuals from $F_1, F_2, \dots$ into $P_{t+1}$ until $|P_{t+1}| = N$
				\STATE \textbf{Step 5: Apply Crowding Distance (if needed)}
				\STATE \hspace{1em} If the last front $F_k$ exceeds capacity $N$, select individuals based on crowding distance to preserve diversity
				\ENDFOR
				\STATE \textbf{Output:} Final Pareto-optimal set $P_T$
			\end{algorithmic}
		\end{algorithm}
		The key to its effectiveness is a hierarchical selection strategy that relies on two primary mechanisms: the Fast Non-Dominated Sorting (FNDS) procedure for convergence, and the Crowding Distance calculation for diversity preservation.
		
		\subsection{The Non-Dominated Sorting Genetic Algorithm III (NSGA-III)}
		
		\add{The Non-Dominated Sorting Genetic Algorithm III (NSGA-III), introduced by Deb and Jain~\citet{6600851}, extends NSGA-II to many-objective optimization problems involving four or more objectives. While NSGA-II uses crowding distance to maintain diversity, NSGA-III employs a set of predefined reference directions to promote a well-distributed approximation of the Pareto front. This approach helps address the loss of selection pressure that can occur when many solutions become mutually non-dominated in higher-dimensional objective spaces.}
		
		NSGA-III follows an elitist evolutionary framework similar to NSGA-II:
		\begin{enumerate}[leftmargin=*]
			\item Generate an offspring population $Q_t$ from the parent population $P_t$ using selection, crossover, and mutation.
			\item Combine the populations to form $R_t=P_t\cup Q_t$ and perform non-dominated sorting.
			\item Select complete Pareto fronts until the next front would exceed the population size $N$.
			\item Normalize the objective values and associate individuals with reference directions.
			\item Fill the remaining population slots using reference-direction-based niching to maintain diversity.
		\end{enumerate}
		\begin{algorithm}[H]
			\caption{Non-Dominated Sorting Genetic Algorithm III (NSGA-III)}
			\begin{algorithmic}[1]
				\STATE \textbf{Input:} Population size $N$, generations $T$, reference directions $\mathcal{Z}$
				\STATE \textbf{Initialize:} Generate population $P_0$ of size $N$
				\FOR{$t=0$ to $T-1$}
				\STATE Generate offspring $Q_t$ using genetic operators
				\STATE Form $R_t=P_t\cup Q_t$
				\STATE Perform non-dominated sorting of $R_t$
				\STATE Add complete Pareto fronts to $P_{t+1}$ while capacity permits
				\IF{$|P_{t+1}|<N$}
				\STATE Normalize objective values of candidate solutions
				\STATE Associate candidates with reference directions in $\mathcal{Z}$
				\STATE Apply reference-direction niching to fill remaining slots
				\ENDIF
				\ENDFOR
				\STATE \textbf{Output:} Final non-dominated solutions from $P_T$
			\end{algorithmic}
		\end{algorithm}
		\add{The principal distinction between NSGA-II and NSGA-III is their diversity-preservation mechanism. NSGA-III replaces crowding-distance-based survival selection with reference-direction-based niching, making it particularly suitable for many-objective problems such as the four-objective benchmark considered in this study.}

		\subsection{The Constrained Two-Archive Evolutionary Algorithm (C-TAEA)}
		
		\add{The Constrained Two-Archive Evolutionary Algorithm (C-TAEA), proposed by Li et al.~\citet{8413136}, is designed for constrained many-objective optimization problems. Unlike conventional evolutionary algorithms that rely primarily on a single population, C-TAEA maintains two complementary archives: a convergence archive (CA) and a diversity archive (DA). The CA emphasizes convergence toward feasible, high-quality solutions, whereas the DA promotes diversity and exploration, including regions that may be difficult to reach under strict constraint handling.}
		
		The algorithm proceeds through the following steps:
		\begin{enumerate}[leftmargin=*]
			\item Initialize the convergence archive $CA_0$ and diversity archive $DA_0$.
			\item Select parents from the two archives and generate offspring using genetic operators.
			\item Update the convergence archive by considering constraint satisfaction, convergence, and diversity.
			\item Update the diversity archive to preserve solutions in underrepresented regions of the objective space.
			\item Repeat the archive updates until the prescribed number of generations is reached.
		\end{enumerate}
		\begin{algorithm}[H]
			\caption{Constrained Two-Archive Evolutionary Algorithm (C-TAEA)}
			\begin{algorithmic}[1]
				\STATE \textbf{Input:} Archive size $N$, number of generations $T$
				\STATE \textbf{Initialize:} Construct archives $CA_0$ and $DA_0$
				\FOR{$t=0$ to $T-1$}
				\STATE Select parents from $CA_t$ and $DA_t$
				\STATE Generate offspring $Q_t$ using genetic operators
				\STATE Update $CA_{t+1}$ using feasibility, convergence, and diversity criteria
				\STATE Update $DA_{t+1}$ to maintain diversity and exploration
				\ENDFOR
				\STATE \textbf{Output:} Feasible non-dominated solutions from $CA_T$
			\end{algorithmic}
		\end{algorithm}
		\add{The main advantage of C-TAEA is its cooperative two-archive strategy, which balances convergence toward the constrained Pareto front with exploration of alternative search regions. This feature is especially relevant to constrained problems with disconnected feasible components, although successful recovery of all components is not guaranteed.}
		
		
		\subsection{Multiobjective Evolutionary Algorithm Based on
			Decomposition (MOEA/D)}
		
		The Multiobjective Evolutionary Algorithm Based on
		Decomposition (MOEA/D), introduced by Zhang and
		Li~\citet{zhang2007moea}, is a population-based method
		that decomposes a multiobjective optimization problem
		into a collection of scalar optimization subproblems.
		Each subproblem is associated with a weight vector,
		and neighboring subproblems exchange information
		during the evolutionary search.
		
		A commonly used decomposition is the weighted
		Tchebycheff formulation:
		\begin{equation}
			\min_{\bm{x}\in\Omega}
			g^{\mathrm{te}}(\bm{x}\mid\bm{\lambda},\bm{z}^{\ast})
			=
			\max_{1\leq i\leq p}
			\left\{
			\lambda_i
			\left|f_i(\bm{x})-z_i^{\ast}\right|
			\right\},
		\end{equation}
		where $\bm{\lambda}$ is a nonnegative weight vector
		and $\bm{z}^{\ast}$ is an ideal or reference point
		in objective space.
		
		MOEA/D evolves solutions to these scalar subproblems
		simultaneously, using neighborhood-based selection
		and variation operators. The resulting population
		provides an approximation of the Pareto front,
		with its distribution influenced by the weight
		vectors, decomposition function, and evolutionary
		search parameters.
		
		Unlike the proposed homotopy continuation method,
		which follows a parameterized system of KKT equations,
		MOEA/D approximates the Pareto front through
		population-based optimization of scalar subproblems.
		
	\end{appendices}
	\bibliography{references}
\end{document}